\documentclass[reqno, a4paper, 11pt]{amsart}
\usepackage{amssymb}
\usepackage{mathrsfs}
\usepackage{amssymb, url, color, pb-diagram, graphicx, amscd, mathrsfs}
\usepackage[colorlinks=true, bookmarks=true, pdfstartview=FitH, pagebackref=true, linktocpage=true, linkcolor = magenta, citecolor = blue]{hyperref}
\usepackage[nodayofweek]{datetime}
\usepackage{graphicx}
\usepackage{times, txfonts}
\usepackage[pagewise,running,mathlines]{lineno}
\usepackage[subnum]{cases}
\usepackage[compress]{cite}
\DeclareMathAlphabet{\mathpzc}{OT1}{pzc}{m}{it}

\newtheorem{theorem}{Theorem}[section]
\newtheorem{lemma}[theorem]{Lemma}
\newtheorem{corollary}[theorem]{Corollary}
\newtheorem{proposition}[theorem]{Proposition}

\theoremstyle{definition}
\newtheorem{definition}[theorem]{Definition}

\theoremstyle{remark}

\numberwithin{equation}{section}

\def\XXint#1#2#3{{\setbox0=\hbox{$#1{#2#3}{\int}$}
 \vcenter{\hbox{$#2#3$}}\kern-.5\wd0}}

\DeclareMathOperator{\T}{\mathcal T}
\DeclareMathOperator{\U}{\mathscr U}
\DeclareMathOperator{\E}{\mathscr E}
\DeclareMathOperator{\SL}{\mathcal L}
\DeclareMathOperator{\R}{{\it R}_{\theta({\it t})}}
\DeclareMathOperator{\Ro}{{\it R}_{\theta_0}}

\DeclareMathOperator{\dv}{\theta_0\wedge(d\theta_0)^{\it n}}
\DeclareMathOperator{\dvg}{\theta\wedge(d\theta)^{\it n}}
\DeclareMathOperator{\dva}{\theta_{\it a}\wedge(d\theta_{\it a})^{\it n}}
\DeclareMathOperator{\dvh}{\theta_{\mathbb{H}^n}\wedge(d\theta_{\mathbb{H}^n})^{\it n}}
\newcommand{\definedas}{\mathrel{\raise.095ex\hbox{\rm :}\mkern-5.2mu=}}
\def\FS{\mathbf{S}_1^2(M)}
\def\Vol{{\mathop{\rm vol}}}

\let\epsilon\varepsilon

\allowdisplaybreaks

\begin{document}

\title[Prescribed Webster scalar curvature Flow]{Prescribed Webster scalar curvature flow on CR manifold}

\author[P. T. Ho]{Pak Tung Ho}
\address[P. T. Ho]{Department of Mathematics, Sogang University, Seoul 04107, Korea.}
\email{\href{mailto: P.T. Ho <paktungho@yahoo.com.hk>}{paktungho@yahoo.com.hk}}

\author[Q. A. Ng\^{o}]{Qu\^{o}\hspace{-0.5ex}\llap{\raise 1ex\hbox{\'{}}}\hspace{0.5ex}c Anh Ng\^{o}}
\address[Q. A. Ng\^{o}]{University of Science, Vi\^{e}t Nam National University, Ha N\^{o}i, Vi\^{e}t Nam}
\email{\href{mailto: Q. A. Ngo <nqanh@vnu.edu.vn>}{nqanh@vnu.edu.vn}}

\author[H. Zhang]{Hong Zhang}
\address[H. Zhang]{School of Mathematics, South China Normal University, No.55 West Zhongshan Avenue , Guangzhou, Guangdong, China, 510631.}
\email{\href{mailto: H. Zhang <zhanghong@scnu.edu.cn>}{zhanghong@scnu.edu.cn}}
\thanks{}



\date{\today \ at \currenttime}


\setpagewiselinenumbers
\setlength\linenumbersep{110pt}

\begin{abstract}
Using the prescribed Webster scalar curvature flow,
we prove some existence results on
the 3-dimensional compact CR manifold with nonnegative CR Yamabe constant.
\end{abstract}

\maketitle
\section{introduction}
Let $(M, \theta_0)$ be a compact strictly pseudoconvex CR manifold of real dimension $2n+1$ with a background contact form $\theta_0$.
Similar to the Yamabe problem, the CR Yamabe problem is to find a contact form $\theta$
conformal to $\theta_0$ such that its Webster scalar curvature
is constant. This was first introduced by Jerison and Lee
in \cite{Jerison&Lee1}, and was solved
in  \cite{CMP,CCP,Gamara1,Gamara2,Jerison&Lee2,Ho&Kim}.

The CR Yamabe flow is the geometric flow
introduced to solve the CR Yamabe problem. It is
defined as the evolution of $\theta(t)$:
\begin{equation}\label{CRYamabeFlow}
  \frac{\partial}{\partial t}\theta(t)=-(R_{\theta(t)}-\overline{R}_{\theta(t)})\theta(t),
\end{equation}
where $R_{\theta(t)}$ is the Webster scalar curvature of the contact form $\theta(t)$, and $\overline{R}_{\theta(t)}$ is the average of
$R_{\theta(t)}$:
$$\overline{R}_{\theta(t)}=\frac{\int_MR_{\theta(t)}\dv}{\int_M\dv}.$$
We refer the readers to \cite{CC,CCW,CLW,Ho0,Ho1,HSW,HW,SW}
for results related to the CR Yamabe flow.

Let $f$ be a smooth function on $M$.
As a generalization of the CR Yamabe problem,
the  prescribed Webster scalar curvature problem is to find a contact form $\theta$ conformal to $\theta_0$ such that the its Webster scalar curvature $R_\theta$ is the given function $f$. If we write $\theta=u^{2/n}\theta_0$, then the prescribed Webster scalar curvature problem is equivalent to solving the semi-linear equation
\begin{equation}\label{PWC}
 -(2+2/n)\Delta_{\theta_0}u+\Ro u=fu^{1+2/n},
 \end{equation}
where $\Delta_{\theta_0}$ and $\Ro$ are, respectively, the sub-Laplacian and the Webster scalar curvature of the contact form $\theta_0$.
See \cite{CEG,FU,Ho&Kim1,MU,NZ,RG,SG} and references therein for results related to the prescribed Webster scalar curvature problem.

Similar to the CR Yamabe flow,
the prescribed Webster scalar curvature flow is introduced
to study the prescribed Webster scalar curvature problem,
and is defined  as
\begin{equation*}
\frac{\partial}{\partial t}\theta(t)=-(R_{\theta(t)}-\lambda(t)f)\theta(t).
\end{equation*}
Note that when $f\equiv 1$, the prescribed Webster scalar curvature flow reduces to the CR Yamabe flow.
In \cite{Ho4,Ho2,Ho3}, the first author studied the
prescribed Webster scalar curvature flow on the sphere $S^{2n+1}$, when the prescribed function $f$ is positive.
On the other hand, in \cite{Ho5},
the first author studied the
prescribed Webster scalar curvature flow on the $3$-dimensional CR manifold
with negative CR Yamabe constant, when $f$ is assumed to be negative.

In this paper, we study the prescribing Webster scalar problem
on $3$-dimensional CR manifold with nonnegative CR Yamabe constant.
By using the prescribed Webster scalar flow,
we prove some existence results. Unlike the results
mentioned above, the prescribed function $f$ here in our existence results
is allowed to change sign. More precisely, our main results read as

\begin{theorem}\label{main1}
 Let $(M,\theta_0)$ be a compact strictly pseudoconvex CR manifold of real dimension $3$ with positive CR Yamabe constant.  Suppose that $M$ is not conformally equivalent to the CR sphere $\mathbb{S}^{3}$ and the CR Paneitz operator of $M$ is nonnegative. Also, let $f$ be a smooth function on $(M,\theta_0)$ with $\sup_Mf>0$. Then there exists a positive constant $c_*$ depending only on $(M,\theta_0)$
 such that if
 \begin{equation}\label{SupSet}
  \bigg\{x\in M: f(x)=\sup_Mf\bigg\}\subseteq\bigg\{x\in M: f(x)>0~~\mbox{and}~~\Delta_{\theta_0}f(x)/f(x)>-c_*\bigg\},
 \end{equation}
then there exist a positive solution to the prescribed Webster scalar curvature equation \eqref{PWC}.
\end{theorem}

Comparing Theorem \ref{main1} with 
the result of Escobar and Schoen in \cite{ES}
where $f$ is only assumed to be positive somewhere in the dimension $3$, one can see that we need an extra condition (\ref{SupSet}), which shows the difference between the Riemannian case and the CR case. 
The reason why we need this extra condition (\ref{SupSet})
is as follows: In view of Lemma \ref{analysingsigmak} below,
the dominant term is $O(\epsilon^{2n})$, and the expansion of the prescribed function $f$ around a blow-up point $a$ is something like $f(a)+A\epsilon+B\epsilon^2+O(\epsilon^3)$, where $A$ and $B$ involve the gradient and Laplacian of $f$ at $a$. When $n=1$, we can see that the dominant term is $O(\epsilon^2)$, but $\epsilon^2$ also appears in the expansion of $f$. Since $B$ is generally not small, we would lose control if we do not impose any extra condition on $f$. While in the Riemannian case, the dominant term is $O(\epsilon^{n-2})$, and the expansion of the prescribed function $f$ is the same as above. When $n=3$, the the dominant term becomes $\epsilon$. Although  $\epsilon$ appears in the expansion of $f$, the coefficient $A$ is $o(1)$ as time goes to infinity, since $A$ involves only gradient of $f$ at $a$. Therefore, in the Riemannian case, one has full control without any extra condition.

\begin{theorem}\label{main2}
  Let $(M,\theta_0)$ be a compact strictly pseudoconvex CR manifold of real dimension $3$ with vanishing CR Yamabe constant.  Also, let $f$ be a smooth function on $(M,\theta_0)$ satisfying
\begin{equation}\label{1.4}
\sup_Mf>0~~\mbox{and}~~\int_Mf~\dv<0.
\end{equation}
Then there exists a positive solution to the prescribed Webster scalar curvature equation \eqref{PWC}.
\end{theorem}

We remark that Theorem \ref{main2} is sharp in the following sense:
if \eqref{PWC} is solvable on
$(M,\theta_0)$ with vanishing CR Yamabe constant,
then either $f\equiv 0$ or $f$ must satisfies \eqref{1.4}.
To see this, suppose that $u$
is a positive solution to \eqref{PWC}
with $f\not\equiv 0$.
Since  $(M,\theta_0)$ has vanishing CR Yamabe constant,
we may assume, by a conformal change if necessary,
that $\Ro\equiv 0$. Then \eqref{PWC} reduces to
\begin{equation}\label{1.5}
-(2+2/n)\Delta_{\theta_0}u=fu^{1+2/n}.
\end{equation}
Integrating \eqref{1.5} over $M$ yields
$\displaystyle\int_Mfu^{1+2/n}\dv=0.$
Since $u>0$ and $f\not\equiv 0$, this implies that $f$ changes sign,
and in particular, $\displaystyle\sup_Mf>0$. On the other hand,
multiplying $u^{-1-2/n}$ to \eqref{1.5} and integrating it over $M$,
we obtain
$$\int_Mf\dv=-(2+2/n)(1+2/n)\int_M u^{-2-2/n}|\nabla_{\theta_0}u|_{\theta_0}\dv<0.$$

We also remark that most of the arguments in the proof
of  Theorems \ref{main1} and \ref{main2}
are still valid in dimension $2n+1$.
In fact, we wonder if  Theorems \ref{main1} and \ref{main2}
are still true  when the dimension is $2n+1$.
We hope to come back to this in future.
As a result, we will try to write down the proof
in the general case when the dimension is $2n+1$,
and point out exactly the part when the assumption $n=1$ is needed.

The method in this paper is inspired by Mayer \cite{Ma}, 
in which he considered the prescribed scalar curvature problem
on a closed $n$-dimensional Riemannian manifold of positive Yamabe invariant, where $3\leq n\leq 5$.

\section*{Acknowledgement}
The first author was supported by the Basic Science Research Program through the National Research Foundation of Korea (NRF) funded by the Ministry of Education, Science and Technology (Grant No. 2020R1A6A1A03047877\\ and 2019R1F1A1041021).
The third author would like to acknowledge the support from the Young Scientists Fund of the NSFC (Grant No. 11701544).

\section{Prescribed Webster scalar curvature flow}

Let $(M,\theta_0)$ be a compact strictly pseudoconvex CR manifold of real dimension $2n+1$ with a contact form $\theta_0$, $f$ a smooth function on $M$.
As in Theorems \ref{main1} and \ref{main2},
the CR Yamabe constant of $(M,\theta_0)$ is assumed to be nonnegative.
Therefore, by a conformal change if necessary, we may assume that the Webster scalar curvature of $\theta_0$ is
positive (is zero respectively) when the CR Yamabe constant is positive (is zero respectively).

We consider the prescribed Webster scalar curvature flow, which is
the evolution equation of
 $\theta(t)$ satisfying
\begin{equation}\label{ContactFormFlow}
\left\{
\begin{array}{l}
  \frac{\partial}{\partial t}\theta(t)=-(R_{\theta(t)}-\lambda(t)f)\theta(t), \\[0.5em]
  \theta(0)=\theta^0,
\end{array}
\right.
\end{equation}
where $R_{\theta(t)}$ is the Webster scalar curvature of the contact form $\theta(t)$, $\lambda(t)$ is to be determined and $\theta^0$ is an initial conformal contact form.
If we write $\theta(t)=u(t)^{2/n}\theta_0$ and $\theta^0=u_0^{2/n}\theta_0$ for some smooth functions $u(t)$ and $u_0$, then we can obtain an evolution equation for $u(t)$ from \eqref{ContactFormFlow}.
\begin{equation}
 \label{WebScalFlow}
 \left\{
 \begin{array}{l}
  \dfrac{\partial}{\partial t}u(t)=-\dfrac n2\big(R_{\theta(t)}-\lambda(t)f\big)u(t),\\
  u(0)=u_0.
 \end{array}
\right.
\end{equation}
Here, in terms of $u$, $R_{\theta(t)}$ can be written as
\begin{equation}
 \label{WebScal}
 R_{\theta(t)}=u^{-1-\frac2n}\bigg(-\big(2+2/n\big)\Delta_{\theta_0}u(t)+R_{\theta_0}u(t)\bigg).
\end{equation}
To deal with the case where $f$ changes sign, it is convenient to consider the energy functional associated with the prescribed Webster scalar curvature problem
\begin{equation}
 \label{EnergyFunctional}
 \E(u)=\int_M \Big(\big(2+2/n\big)|\nabla_{\theta_0}u|_{\theta_0}^2+R_{\theta_0}u^2\Big)\dv
\end{equation}
with the constraint condition
\begin{equation}
 \label{ConstaintCond}
 X_f:=\Big\{u\in \FS: u\geqslant0, \int_Mfu^{2+2/n}\dv=1\Big\}.
\end{equation}
The constraint condition \eqref{ConstaintCond} above reminds us that we should  keep the quantity $\int_Mfu(t)^{2+2/n}\dv$ unchanged during the evolution. That is
\begin{eqnarray}
 \label{Fixintoff}
 &&0=\frac{d}{dt}\int_Mfu(t)^{2+2/n}\dv=\big(2+2/n\big)\int_Mfu_tu(t)^{1+2/n}\dv\nonumber\\
 &&\quad=-(n+1)\int_M\Big(fR_{\theta(t)}-\lambda(t)f^2\Big)\dvg.
\end{eqnarray}
Solving \eqref{Fixintoff}, we find that
\begin{equation}
 \label{lambda}
 \lambda(t)=\frac{\int_MfR_{\theta(t)}\dvg}{\int_Mf^2\dvg}.
\end{equation}
If we let $u_0\in X$, from \eqref{Fixintoff} we then have
\begin{equation}
 \label{UnitIntoff}
 \int_Mfu(t)^{2+2/n}\dv\equiv1.
\end{equation}

To end this section, we collect some well-known results for later reference.
\begin{itemize}
\item[(i)] Given a closed CR-manifold $(M,\theta)$, the CR-Yamabe constant, denoted by $\mu(M)$, is defined by (see \cite{Jerison&Lee1} for example)
\begin{equation}\label{CRYamabeConst}
\mu(M)=\inf_{\big\{u\geqslant0,~0\not\equiv u\in\FS\big\}}\frac{\int_M(2+2/n)|\nabla_\theta u|^2+R_\theta u^2\dvg}{\Big(\int_Mu^{2+2/n}\dvg\Big)^\frac{n}{n+1}}.
\end{equation}
Here, $\FS$ is the Folland-Stein space of $(M,\theta)$.
\item[(ii)]  Given any $\epsilon>0$, there exist $B>0$ such that for $u\in \FS$
\begin{eqnarray}
 \label{CRYamabeIneq}
 \bigg(\int_Mu^{2+2/n}\dvg\bigg)^\frac{n}{n+1}&\leqslant&\big(K_n+\epsilon\big)\int_M|\nabla_\theta u|^2\dvg\nonumber\\
&&\qquad\qquad\qquad+B\int_Mu^2\dvg,
\end{eqnarray}
where $K_n=1/(2\pi n^2)$ (see Theorem A.1 in \cite{Ho&Kim}). In particular, there exists a positive constant $\mathscr C$ such that
\begin{equation}
  \label{FSEmbedding}
  \bigg(\int_Mu^{2+2/n}\dvg\bigg)^\frac{n}{n+1}\leqslant\mathscr C\bigg(\int_M(|\nabla_\theta u|^2+u^2)\dvg\bigg).
\end{equation}

\end{itemize}


\section{Global existence of the flow}
\begin{lemma}
 \label{EnegyDecay}
 For a smooth solution $u(t)$ of the flow \eqref{WebScalFlow} there holds
 $$\frac{d}{dt}\E(u(t))=-n\int_M\Big(\R-\lambda(t)f\Big)^2\dvg.$$
 In particular, the energy function $\E$ is non-increasing along the flow.
\end{lemma}
\begin{proof}
 By \eqref{WebScalFlow} and \eqref{WebScal}, we can compute
 \begin{equation*}
 \begin{split}
  \frac{d}{dt}\E(u(t))&=2\int_M\big(2+2/n\big)\langle\nabla_{\theta_0}u,\nabla_{\theta_0}u_t\rangle+\Ro uu_t\dv\\
  &=2\int_M\bigg(-\big(2+2/n\big)\Delta_{\theta_0}u+\Ro\bigg)u_t\dv\\
  &=-n\int_M\R\Big(\R-\lambda(t)f\Big)\dvg\\
  &=-n\int_M\Big(\R-\lambda(t)f\Big)^2\dvg,
 \end{split}
 \end{equation*}
 where we have used \eqref{Fixintoff} in the last equality.
\end{proof}
Lemma \ref{EnegyDecay} implies  a uniform priori bound
\begin{equation}
 \label{EnergyBd}
 \E(u(t))+n\int_0^t\int_M(\R-\lambda(t)f)^2\dvg \mathrm{d}t\leqslant\E(u_0),\quad\mbox{ for all }~t>0.
\end{equation}
Now, we derive the evolution equation for the Webster scalar curvature $\R$.
\begin{lemma}\label{EvoWebScal}
 Assume that $u(t)$ is a smooth solution of the flow \eqref{WebScalFlow}. Then one has
 $$\frac{\partial}{\partial t}(\lambda(t)f-\R)=(n+1)\Delta_{\theta(t)}(\lambda(t)f-\R)+(\lambda(t)f-\R)\R+\lambda^\prime(t)f,$$
 where $\lambda^\prime(t)$ is given by
 \[
  \begin{split}
\lambda^\prime(t)&=-\bigg(\int_Mf^2\dvg\bigg)^{-1}\bigg[(n+1)\int_Mf\Delta_{\theta(t)}(\lambda(t)f-\R)\dvg\\
&\quad+n\int_Mf(\lambda(t)f-\R)^2\dvg+\lambda(t)\int_Mf^2(\lambda(t)f-\R)\dvg\bigg].
  \end{split}
 \]
\end{lemma}
\begin{proof}
 Using \eqref{WebScal} and \eqref{WebScalFlow} we have
 \begin{equation*}
  \begin{split}
   \partial_t\R=&\Big(-1-\frac2n\Big)u^{-2-\frac2n}u_t\bigg(-\big(2+2/n\big)\Delta_{\theta_0}u+\Ro u\bigg)\\
   &\quad+u^{-1-\frac2n}\bigg(-\big(2+2/n\big)\Delta_{\theta_0}u_t+\Ro u_t\bigg)\\
   =&\Big(-\frac n2-1\Big)(\lambda(t)f-\R)\R-(n+1)u^{-1-\frac2n}\Delta_{\theta_0}\Big((\lambda(t)f-\R)u\Big)\\
   &\quad+\frac n2u^{-\frac2n}\Ro(\lambda(t)f-\R)\\
   =&\Big(-\frac n2-1\Big)(\lambda(t)f-\R)\R-(n+1)u^{-1-\frac2n}\Big[u\Delta_{\theta_0}(\lambda(t)f-\R)\\
   &\quad+2\langle\nabla_{\theta_0}(\lambda(t)f-\R),\nabla_{\theta_0}u\rangle\Big]+\Big[-(n+1)u^{-1-\frac2n}(\lambda(t)f-\R)\Delta_{\theta_0}u\\
  &\quad +\frac n2u^{-\frac2n}\Ro(\lambda(t)f-\R)\Big]\\
  =&\Big(-\frac n2-1\Big)(\lambda(t)f-\R)\R-(n+1)\Delta_\theta(\lambda(t)f-\R)\\
  &\quad+\frac n2(\lambda(t)f-\R)\R\\
  =&-(n+1)\Delta_\theta(\lambda(t)f-\R)-(\lambda(t)f-\R)\R.
  \end{split}
 \end{equation*}
Notice that in the second last equality above we have used the identity
$$\Delta_\theta\varphi=u^{-1-\frac2n}\Big(u\Delta_{\theta_0}\varphi+2\langle\nabla_{\theta_0}\varphi, \nabla_{\theta_0}u\rangle\Big)~~\mbox{ for any }\varphi.$$
As for $\lambda^\prime(t)$, we use \eqref{lambda}, the evolution equation for $\R$ above and \eqref{WebScalFlow} to calculate it as follows.
\[
 \begin{split}
  \lambda^\prime(t)&=\bigg(\int_Mf^2\dvg\bigg)^{-1}\bigg[\int_Mf\partial_t\R\dvg+\int_Mf\R\partial_t(u^{2+2/n})\dv\bigg]\\
  &\quad-\bigg(\int_Mf^2\dvg\bigg)^{-2}\int_Mf\R\dvg\int_Mf^2\partial_t(u^{2+2/n})\dv\\
  &=\bigg(\int_Mf^2\dvg\bigg)^{-1}\bigg[-(n+1)\int_Mf\Delta_{\theta(t)}(\lambda(t)f-\R)\dvg\\
  &\quad+n\int_Mf\R(\lambda(t)f-\R)\dvg-(n+1)\lambda(t)\int_Mf^2(\lambda(t)f-\R)\dvg\bigg]\\
  &=-\bigg(\int_Mf^2\dvg\bigg)^{-1}\bigg[(n+1)\int_Mf\Delta_{\theta(t)}(\lambda(t)f-\R)\dvg\\
&\quad+n\int_Mf(\lambda(t)f-\R)^2\dvg+\lambda(t)\int_Mf^2(\lambda(t)f-\R)\dvg\bigg].
 \end{split}
\]
\end{proof}
\begin{lemma}\label{VolBd}
 There exist two positive  constant $c_1$ and $c_2$ such that
 $$c_1\leqslant \Vol(M,\theta(t))\leqslant c_2,$$
 where $c_1$ depends only on $f, M$,
  and $c_2$ depends only on $n, f, M, u_0$.
\end{lemma}
\begin{proof}
 First, we show the lower bound. To see this, from \eqref{UnitIntoff} it follows that
 $$1=\int_Mf\dvg\leqslant \left(\sup_Mf\right)\Vol(M,\theta(t)).$$
 This implies that
 $$\Vol(M,\theta(t))\geqslant\frac{1}{\sup_Mf}=:c_1.$$
Next, we prove the upper bound. We split the argument into two cases.

\noindent{\bf Case 1}.\quad $\Ro>0$, i.e., positive case. In this case, it follows from Folland-Stein embedding (\ref{FSEmbedding}) and the energy bound \eqref{EnergyBd} that
$$\Vol(M,\theta(t))=\|u\|_{L^{2+2/n}}^{2+2/n}\leqslant c_0\E(u)^{1+1/n}\leqslant c_0\E(u_0)^{1+1/n},$$
for some positive constant $c_0$.

\noindent{\bf Case 2}.\quad $\Ro\equiv0$, i.e., null case. In this case, we first use Poincar\'{e} type inequality (see \cite{Lu} for example) and \eqref{EnergyBd} to obtain
\begin{equation}
 \label{Pointype}
 \|u-\bar{u}\|_{L^{2+2/n}}\leqslant c_0\|\nabla_{\theta_0}u\|_{L^2}
 =c_0\sqrt{\frac{n\E(u)}{2n+2}}\leqslant c_0\sqrt{\frac{n\E(u_0)}{2n+2}}.
\end{equation}
It remains to bound $\bar{u}$. Noting that $\int_Mf\dv<0$, we have by \eqref{UnitIntoff} that
\begin{equation}\label{AverageBd1}
\begin{split}
 \bar{u}^{2+2/n}\int_M-f\dv&=-1+\int_Mf\Big(u^{2+2/n}-\bar{u}^{2+2/n}\Big)\dv\\
 &\leqslant\sup_M|f|\int_M\Big|u^{2+2/n}-\bar{u}^{2+2/n}\Big|\dv
 \end{split}
\end{equation}
Recall that for any $a,b>0$, we have the elementary inequality
\[
 \begin{split}
  \Big|a^{2+2/n}-b^{2+2/n}\Big|&\leqslant\Big(2+2/n\Big)(a+b)^{1+2/n}|a-b|\\
  &\leqslant\Big(2+2/n\Big)2^{2/n}\Big(|a-b|^{1+2/n}+(2b)^{1+2/n}\Big)|a-b|\\
  &\leqslant c(n)\Big(|a-b|^{2+2/n}+b^{1+2/n}|a-b|\Big),
 \end{split}
\]
where $c(n)$ is a positive constant depending only on $n$. Applying this inequality with $a=u$ and $b=\bar{u}$ we get
\begin{equation}
 \label{AverageBd2}
   \Big|u^{2+2/n}-\bar{u}^{2+2/n}\Big|\leqslant c(n)\Big(|u-\bar{u}|^{2+2/n}+\bar{u}^{1+2/n}|u-\bar{u}|\Big).
\end{equation}
Substituting \eqref{AverageBd2} into \eqref{AverageBd1} and using H\"{o}lder's inequality and Young's inequality, we obtain
\begin{equation*}
\begin{split}
 \bar{u}^{2+2/n}&\leqslant\frac{c(n)\sup_Mf}{\int_M-f\dv}\Big(\|u-\bar{u}\|_{L^{2+2/n}}^{2+2/n}+\bar{u}^{1+2/n}\Vol(M,\theta_0)^\frac{n+2}{2(n+1)}\|u-\bar{u}\|_{L^{2+2/n}}\Big)\\
 &\leqslant\frac12\bar{u}^{2+2/n}+c(n,f,M)\|u-\bar{u}\|_{L^{2+2/n}}^{2+2/n},
 \end{split}
\end{equation*}
for some positive constant $c(n,f,M)$ depending only on $n, f$ and $M$. This together with \eqref{Pointype} yields
$$\bar{u}\leqslant c(n,f,M,u_0),$$
for some positive constant $c(n,f,M,u_0)$ depending only on $n, f, M$ and $u_0$. Then from Minkowski's inequality and \eqref{Pointype}, it follows that there exists a positive constant $c_2$ depending only on $n, f, M$ and $u_0$ such that
$$\Vol(M,\theta(t))\leqslant c_2,$$
which completes the proof.
\end{proof}
The proof of the lemma above also implies the $S_1^2$-bound.
\begin{corollary}\label{S12Bd}
 There exists a positive constant $c_3$ depending only on $n, M, f, u_0$ such that
 $$\|u(t)\|_{\FS}\leqslant c_3.$$
\end{corollary}
Next lemma shows that the normalized factor $\lambda$ keeps bounded
as long as the flow exists.
\begin{lemma}\label{lambdaBd}
 There exists a positive constant $\Lambda_0$ depending only on $n, f, M, u_0$ such that
 $$|\lambda(t)|\leqslant\Lambda_0.$$
\end{lemma}
\begin{proof}
 On one hand, by \eqref{WebScal} and integration by parts we can write
 \[
 \begin{split}
  \int_Mf&\R\dvg=\int_Mfu\Big(-\big(2+2/n\big)\Delta_{\theta_0}u+\Ro u\Big)\dv\\
  &\quad=\int_M(2+2/n)\langle\nabla_{\theta_0}(fu),\nabla_{\theta_0}u\rangle+fR_{\theta_0}u^2\dv\\
  &\quad=\int_Mf\Big[(2+2/n)|\nabla_{\theta_0}u|^2+R_{\theta_0}u^2\Big]\dv\\
  &\qquad\qquad\qquad+\int_M(2+2/n)u\langle\nabla_{\theta_0}f,\nabla_{\theta_0}u\rangle\dv.
 \end{split}
 \]
 Hence,  from Corollary \ref{S12Bd} it follows that
 \[
 \begin{split}
  \bigg|  \int_Mf\R\dvg\bigg|&\leqslant\sup_M\Big(|f|\Big(2+2/n+R_{\theta_0}\Big)\Big)\|u\|_{\FS}^2+(2+2/n)\sup_M|\nabla_{\theta_0}f|\|u\|_{\FS}^2\\
  &\leqslant c_3\Big[\sup_M\Big(|f|\Big(2+2/n+R_{\theta_0}\Big)\Big)+(2+2/n)\sup_M|\nabla_{\theta_0}f|\Big].
 \end{split}
 \]
 On the other hand, using H\"{o}lder's inequality, \eqref{UnitIntoff} and Lemma \ref{VolBd} we can obtain
 $$1=\bigg(\int_Mf\dvg\bigg)^2\leqslant\Vol(M,\theta(t))\int_Mf^2\dvg\leqslant c_2\int_Mf^2\dvg,$$
 that is
 \begin{equation}
  \label{LowBdIntf2}
  \int_Mf^2\dvg\geqslant c_2^{-1}.
 \end{equation}
 Substituting the estimates above into \eqref{lambda} we thus find a positive constant $\Lambda_0$ depending only on $n, f, M, u_0$ such that  $|\lambda(t)|\leqslant\Lambda_0.$
\end{proof}

To obtain the global existence of the flow \eqref{WebScalFlow}, we need to bound the derivative of $\lambda(t)$.  To this end,  we let
$$F_p(\theta(t))=\int_M|\R-\lambda(t)|^p\dvg,\quad p>0$$
and try to show that $F_2(\theta(t))$ is bounded. We calculate the derivative of $F_p(\theta(t))$.  By Lemma \ref{EvoWebScal}, we have
\begin{equation}\label{DerofFp}
 \begin{aligned}[b]
&\frac{d}{dt}F_p(\theta(t))=\frac{d}{dt}\int_M|\lambda(t)f-\R|^p\dvg\\
&\quad=p\int_M(\lambda(t)f-\R)_t(\lambda(t)f- \R)|\lambda(t)f-\R|^{p-2}\dvg\\
&\qquad+(n+1)\int_M|\lambda(t)f-\R|^p(\lambda(t)f-\R)\dvg\\
&\quad=-p(n+1)\int_M\Delta_{\theta(t)}(\lambda(t)f-\R)(\lambda(t)f-\R)\\
&\qquad|\lambda(t)-\R|^{p-2}\dvg+p\int_M\R|\lambda(t)f-\R|^p\dvg\\
&\qquad+\lambda^\prime(t)p\int_Mf(\lambda(t)f- \R)|\lambda(t)f-\R|^{p-2}\dvg\\
&\qquad+(n+1)\int_M|\lambda(t)f-\R|^p(\lambda(t)f-\R)\dvg\\
&\quad=-\frac{4(p-1)(n+1)}{p}\int_M\Big|\nabla_{\theta(t)}|\lambda(t)f-\R|^{p/2}\Big|^2\dvg\\
&\qquad+p\int_M\lambda(t)f|\lambda(t)f-\R|^p\dvg\\
&\qquad+\lambda^\prime(t)p\int_Mf(\lambda(t)f- \R)|\lambda(t)f-\R|^{p-2}\dvg\\
&\qquad+(n+1-p)\int_M|\lambda(t)f-\R|^p(\lambda(t)f-\R)\dvg.
 \end{aligned}
\end{equation}
\begin{lemma}\label{BdofF2}
 There exists a positive constant $c_4(n, f, M, u_0)$ such that $F_2(\theta(t))\leqslant c_4$.
\end{lemma}
\begin{proof}
 Setting $p=2$ in \eqref{DerofFp} and using \eqref{Fixintoff} and Lemma \ref{lambdaBd} we have
 \[
 \begin{split}
 \frac{d}{dt}F_2(\theta(t)) &=-2(n+1)\int_M\Big|\nabla_{\theta(t)}(\lambda(t)f-\R)\Big|^2\dvg\\
 &\qquad+2\int_M\lambda(t)f(\lambda(t)f-\R)^2\dvg\\
&\qquad+2\lambda^\prime(t)\int_Mf(\lambda(t)f- \R)\dvg\\
&\qquad+(n-1)\int_M(\lambda(t)f-\R)^3\dvg\\
 &=(n+1)\int_M\lambda(t)f(\lambda(t)f-\R)^2\dvg\\
 &\qquad-(2+2/n)\int_M\Big|\nabla_{\theta(t)}(\lambda(t)f-\R)\Big|^2\dvg\\
 &\qquad-(n-1)\bigg[(2+2/n)\int_M\Big|\nabla_{\theta(t)}(\lambda(t)f-\R)\Big|^2\\
 &\qquad+\R(\lambda(t)f-\R)^2\dvg\bigg]\\
 &\leqslant(n+1)\Lambda_0\left(\sup_Mf\right)F_2(\theta(t)),
 \end{split}
 \]
where the term in the bracket in the second equality is non-negative due to the conformal invariant and the assumption of non-negative CR-Yamabe constant. Now, integrating both sides of the above differential inequality from $0$ to $t$ and using \eqref{EnergyBd}, we can obtain
\[
\begin{split}
F_2(\theta(t))&\leqslant F_2(\theta^0)+(n+1)\Lambda_0\left(\sup_Mf\right)\int_0^tF_2(\theta(s))~\mathrm{d}s\\
&\leqslant F_2(\theta^0)+(n+1)\Lambda_0\left(\sup_Mf\right)\E(u_0)=:c_4.
\end{split}
\]
It is clear that $c_4$ depends only on $n, f, M,$ and $u_0$.
\end{proof}
With the help of the boundedness of $F_2(\theta(t))$, we can control the derivative of $\lambda(t)$.
\begin{lemma}\label{DeroflambdaBd}
 There exists a positive constant $C$ depending only on $n, f, M,$ and $u_0$ such that $|\lambda^\prime(t)|\leqslant C\sqrt{F_2(\theta(t))}.$ In particular, there exists a positive constant $\Lambda_1$ depending only on $n, f, M,$ and $u_0$ such that $|\lambda^\prime(t)|\leqslant\Lambda_1$.
\end{lemma}
\begin{proof}
 It follows from \eqref{LowBdIntf2}
 and Lemmas \ref{EvoWebScal}, \ref{VolBd} and \ref{lambdaBd}  that
 \begin{eqnarray}\label{DelaBd0}
  |\lambda^\prime(t)|&\leqslant&(n+1)c_2\bigg|\int_Mf\Delta_{\theta(t)}(\lambda(t)f-\R)\dvg\bigg|+nc_2\left(\sup_M|f|\right)F_2(\theta(t))\nonumber\\
  &&+c_2^{3/2}\Lambda_0\left(\sup_Mf^2\right)\sqrt{F_2(\theta(t))}.
 \end{eqnarray}
It remains to estimate the first term on the right hand side of the inequality above. By integration by parts, we rewrite  the integral as follows:
\begin{eqnarray}\label{DelaBd1}
 &&\int_Mf\Delta_{\theta(t)}(\lambda(t)f-\R)\dvg\nonumber\\
 &&\quad=\int_M(\lambda(t)f-\R)\Delta_{\theta(t)}f\dvg\nonumber\\
 &&\quad=\int_M\Big(u^2\Delta_{\theta_0}f+2u\langle\nabla_{\theta_0}f,\nabla_{\theta_0}u\rangle\Big)(\lambda(t)f-\R)\dv\nonumber\\
 &&\quad=\int_Mu^2\Delta_{\theta_0}f(\lambda(t)f-\R)\dv\nonumber\\
 &&\qquad+\int_M2u\langle\nabla_{\theta_0}f,\nabla_{\theta_0}u\rangle(\lambda(t)f-\R)\dv\nonumber\\
 &&\quad=:\mathrm{I}+\mathrm{II}.
\end{eqnarray}
{\bf Estimate of I}: By H\"{o}lder's inequality and Lemma \ref{VolBd} we can estimate
\begin{eqnarray}
 \label{DelaBd2}
 |\mathrm{I}|&\leqslant&\sup_M|\Delta_{\theta_0}f|\int_Mu^{1-\frac1n}|\lambda(t)f-\R|u^{1+\frac1n}\dv\nonumber\\
 &\leqslant&\sup_M|\Delta_{\theta_0}f|\bigg(\int_Mu^{2-\frac2n}\dv\bigg)^\frac12\sqrt{F_2(\theta(t))}\nonumber\\
 &\leqslant&\sup_M|\Delta_{\theta_0}f|\Vol(M,\theta(t))^\frac{n-1}{2(n+1)}\Vol(M,\theta_0)^\frac{1}{n+1}\sqrt{F_2(\theta(t))}\nonumber\\
 &\leqslant&\sup_M|\Delta_{\theta_0}f|c_2^\frac{n-1}{2(n+1)}\Vol(M,\theta_0)^\frac{1}{n+1}\sqrt{F_2(\theta(t))}.
\end{eqnarray}
{\bf Estimate of II}: From H\"{o}lder's inequality we deduce that
\begin{eqnarray}
 \label{DelaBd3}
|\mathrm{II}|&\leqslant&2\sup_M|\nabla_{\theta_0}f|\int_Mu^{-1/n}|\nabla_{\theta_0}u||\lambda(t)f-\R|u^{1+1/n}\dv\nonumber\\
&\leqslant&2\sup_M|\nabla_{\theta_0}f|\bigg(\int_Mu^{-2/n}|\nabla_{\theta_0}u|^2\dv\bigg)^\frac12\sqrt{F_2(\theta(t))}.
\end{eqnarray}
Now, we have to control the term $\int_Mu^{-2/n}|\nabla_{\theta_0}u|^2\dv$. To see this, we split the argument into two cases.

\noindent{\bf Case 1}: $n\neq2$. In this case, by integrating by parts, \eqref{WebScal} and the fact that $\Ro\geqslant0$ we have
\begin{eqnarray}
 \label{DelaBd4}
 &&\int_Mu^{-2/n}|\nabla_{\theta_0}u|^2\dv=\frac{n}{n-2}\int_M\langle\nabla_{\theta_0}(u^{1-2/n}), \nabla_{\theta_0}u\rangle\dv\nonumber\\
 &&\qquad=\frac{n^2}{2(n+1)(n-2)}\int_Mu^{1-2/n}\Big[-(2+2/n)\Delta_{\theta_0}u\Big]\dv\nonumber\\
 &&\qquad=\frac{n^2}{2(n+1)(n-2)}\int_M\R u^2-\Ro u^{2-2/n}\dv\nonumber\\
 &&\qquad\leqslant\frac{n^2}{2(n+1)|n-2|}\int_M|\R-\lambda(t)f|u^2+|\lambda(t)f|u^2\nonumber\\
 &&\qquad\qquad+\Ro u^{2-2/n}\dv.
\end{eqnarray}
In the same way as in the estimate of I, we obtain by Lemma \ref{BdofF2} that
$$\int_M|\R-\lambda(t)f|u^2\dv\leqslant c_2^\frac{n-1}{2(n+1)}\Vol(M,\theta_0)^\frac{1}{n+1}\sqrt{c_4}.$$
and
$$\int_M\Ro u^{2-2/n}\dv\leqslant c_2^\frac{n-1}{n+1}\Vol(M,\theta_0)^\frac{2}{n+1}\sup_M\Ro.$$
Moreover, we use Corollary \ref{S12Bd} to bound
$$\int_M|\lambda(t)f|u^2\dv\leqslant\Lambda_0\left(\sup_M|f|\right)\|u\|_{\FS}^2\leqslant\Lambda_0\left(\sup_M|f|\right)c_3^2.$$
Substituting these estimates into \eqref{DelaBd4} gives
$$\int_Mu^{-2/n}|\nabla_{\theta_0}u|^2\dv\leqslant C.$$
Hereafter, we use $C$ to denote the constant depending only on $n, f, M$ and $u_0$ which may vary from line to line.

\noindent{\bf Case 2}: $n=2$. In this case, it follows from integration by parts and \eqref{WebScal} that
\begin{eqnarray}\label{DelaBd5}
 &&\int_Mu^{-1}|\nabla_{\theta_0}u|^2\dv=\int_M\langle\nabla_{\theta_0}\log u,\nabla_{\theta_0}u\rangle\dv\nonumber\\
 &&\qquad=\frac13\int_M\log u\Big(-3\Delta_{\theta_0}u\Big)\dv\nonumber\\
 &&\qquad=\frac13\int_M\log u\Big(\R u^2-\Ro u\Big)\dv\nonumber\\
 &&\qquad\leqslant\frac13\int_M|\log u|\Big(|\R-\lambda(t)f|u^2+|\lambda(t)f|u^2+\Ro u\Big)\dv\nonumber.\\
\end{eqnarray}
To estimate the terms on the right hand side of the inequality above, we consider the function
$$\sigma(s)=s^\alpha|\log s|^\beta,~~s>0, \alpha, \beta>0.$$
It is easy to see that
\begin{equation}
\label{sigma}
0\leqslant\sigma(s)\leqslant s^{\alpha+\beta},~~\mbox{if}~~s\geqslant1\quad\mbox{and}\quad 0<\sigma(s)\leqslant\Big((\alpha e)^{-1}\beta\Big)^\beta, ~~\mbox{if}~~s\in(0,1).
\end{equation}
Now, using H\"{o}lder's inequality, \eqref{sigma} with $\alpha=1, \beta=2$, \eqref{BdofF2} and Lemma \ref{VolBd}, we can obtain
\[
 \begin{split}
  \int_M|\log u|&|\R-\lambda(t)f|u^2\dv\\
  &=\int_Mu^\frac12|\log u||\R-\lambda(t)f|u^{3/2}\dv\\
  &\leqslant\bigg(\int_Mu|\log u|^2\dv\bigg)^\frac12\sqrt{F_2(
  \theta(t))}\\
  &\leqslant\bigg(\int_Mu^3+4e^{-2}\dv\bigg)^\frac12\sqrt{c_4}\\
 &=\Big(\Vol(M,\theta(t))+4e^{-2}\Vol(M,\theta_0)\Big)^\frac12\sqrt{c_4}\\
 &\leqslant\Big(c_2+4e^{-2}\Vol(M,\theta_0)\Big)^\frac12\sqrt{c_4}.
 \end{split}
\]
Applying \eqref{sigma} with $\alpha=2, \beta=1$ and using Lemmas \ref{lambdaBd} and \ref{VolBd}, we have
\[
\begin{split}
\int_M|\lambda(t)f|u^2|\log u|&\dv\leqslant\Lambda_0\sup_M|f|\int_Mu^3+(2e)^{-1}\dv\\
&\leqslant \Lambda_0\sup_M|f|\Big(c_2+(2e)^{-1}\Vol(M,\theta_0)\Big).
\end{split}
\]
Finally, by \eqref{sigma} with $\alpha=1, \beta=1$ and Corollary \ref{S12Bd} we get
\[
 \begin{split}
  \int_M\Ro u|\log u|&\dv\leqslant\sup_M\Ro\int_Mu^2+e^{-1}\dv\\
  &\leqslant\sup_M\Ro\Big(c_3^2+e^{-1}\Vol(M,\theta_0)\Big).
 \end{split}
\]
Substituting the three estimates above into \eqref{DelaBd5} yields
$$\int_Mu^{-2/n}|\nabla_{\theta_0}u|^2\dv\leqslant C.$$
Plugging the estimates in Case 1 and Case 2 into \eqref{DelaBd3}, we get
$$|\mathrm{II}|\leqslant C\sqrt{F_2(\theta(t))}.$$
This together with \eqref{DelaBd0}, \eqref{DelaBd1} and \eqref{DelaBd2} gives the assertion.
\end{proof}

Now, we are ready to show that the Webster scalar curvature
is bounded from the below uniformly. Before doing so, we define
\begin{equation}
 \label{gamma}
 \gamma=\min\Bigg\{-\Lambda_0\sup_M|f|, \inf_M(R_{\theta^0}-\lambda(0)f), -\sqrt{\frac43\Big(\Lambda_0^2(\sup_Mf)^2+\Lambda_1\sup_M|f|\Big)}~\Bigg\}
\end{equation}
\begin{lemma}\label{WebScalLowerBd}
 The Webster scalar curvature $\R$ of the contact form $\theta(t)$ satisfies
 $$\R-\lambda(t)f\geqslant\gamma.$$
\end{lemma}
\begin{proof}
 Define the differential operator
 $$\mathscr{L}=\partial_t-(n+1)\Delta_{\theta(t)}+(\lambda(t)f-\gamma).$$
 Then it follows from Lemmas \ref{EvoWebScal}, \ref{lambdaBd} and \ref{DeroflambdaBd}  and the choice of $\gamma$ in \eqref{gamma} that
 \[
\begin{split}
 \mathscr{L}(\lambda(t)f-\R+\gamma)&=\partial_t(\lambda(t)f-\R)-(n+1)\Delta_{\theta(t)}(\lambda(t)f-\R)\\
 &\qquad+(\lambda(t)f-\gamma)(\lambda(t)f-\R+\gamma)\\
 &=(\lambda(t)f)^2-\gamma^2+(\gamma-\R)\R+\lambda^\prime(t)f\\
 &\leqslant\Lambda_0^2\Big(\sup_Mf\Big)^2-\frac34\gamma^2+\Lambda_1\sup_M|f|\leqslant0.
\end{split}
 \]
Since  $\lambda(t)f-\gamma\geqslant0$ and $\lambda(t)f-\R+\gamma|_{t=0}\leqslant0$ by \eqref{gamma},
we can apply maximum principle to the operator $\mathscr{L}$
to get the result.
\end{proof}

With all the preparations above, we are in a position to prove the global existence of the flow \eqref{WebScalFlow}. The key step is to bound the conformal factor $u(\cdot,t)$.
\begin{lemma}\label{Bdofu}
 For any $T>0$, there exists a positive constant $C=C(T)$ such that
 $$C^{-1}\leqslant u(\cdot,t)\leqslant C, \quad\mbox{for all}~t\in[0,T].$$
\end{lemma}
\begin{proof}
 From \eqref{WebScalFlow} and Lemma \ref{WebScalLowerBd}, it follows that
 $$\frac{\partial}{\partial t}\log u\leqslant-\frac n2\gamma.$$
 So, for $t\in[0,T]$, we have
 $$u(x,t)\leqslant\Big(\sup_Mu_0\Big)\exp\Big(-\frac n2\gamma T\Big).$$
 Next, we get the lower bound.  To do so, we let
 $$P:=\Ro+\sup_{t\in[0,T]}\sup_M\Big(-(\lambda(t)f+\gamma)u^{2/n}\Big).$$
 Then by \eqref{WebScal} and Lemma \ref{WebScalLowerBd} we can obtain
 \[
 \begin{split}
  0&\leqslant(\R-\lambda(t)f-\gamma)u^{1+2/n}\\
  &=\R u^{1+2/n}-(\lambda(t)f+\gamma)u^{1+2/n}\\
&=-(2+2/n)\Delta_{\theta_0}u+\Ro u-(\lambda(t)f+\gamma)u^{1+2/n}\\
&\leqslant-(2+2/n)\Delta_{\theta_0}u+Pu.
 \end{split}
 \]
 Thus, we can apply Proposition A.1 in \cite{Ho1} to conclude that
 $$C\Big(\inf_Mu(t)\Big)\Big(\sup_Mu(t)\Big)^{1+2/n}\geqslant\int_Mu(t)^{2+2/n}\dv.$$
 This together with Lemma \ref{VolBd} and the upper bound above implies that
 $$\inf_Mu(t)\geqslant C^{-1}.$$
 This proves the assertion.
\end{proof}

\begin{lemma}
 For $p>n+1$, there exists a positive constant $C$ such that
 $$\frac{d}{dt}F_p(\theta(t))\leqslant CF_p(\theta(t))+CF_p(\theta(t))^\frac{p-n}{p-n-1}.$$
 Here, if $\Ro>0$, then $C$ is independent of time; while if $\Ro\equiv0$, then the differential inequality above holds on $[0, T]$ for any given $T>0$ with $C$ depending on $T$.
\end{lemma}
\begin{proof}
 {\bf Case 1.}\quad $\Ro>0$.  In this case, from \eqref{CRYamabeIneq}, \eqref{DerofFp} and Lemmas \ref{lambdaBd} and  \ref{DeroflambdaBd}, it follows that
 \begin{equation}\label{DiffIneqFp}
\begin{aligned}
 \frac{d}{dt}F_p(\theta(t))&=-\frac{2n(p-1)}{p}\int_M(2+2/n)\Big|\nabla_{\theta(t)}|\lambda(t)f-\R|^{p/2}\Big|^2\\
 &\quad+\R|\lambda(t)f-\R|^p\dvg\\
 &\quad+\bigg(p+\frac{2n(p-1)}{p}\bigg)\int_M\R|\lambda(t)f-\R|^p\dvg\\
 &\quad+\lambda^\prime(t)p\int_Mf(\lambda(t)f- \R)|\lambda(t)f-\R|^{p-2}\dvg\\
&\quad+(n+1)\int_M|\lambda(t)f-\R|^p(\lambda(t)f-\R)\dvg\\
&\leqslant-\frac{2n(p-1)\mu(M)}{p}F_{\frac{p(n+1)}{n}}(\theta(t))^\frac{n}{n+1}+\bigg(p+\frac{2n(p-1)}{p}\bigg)\Lambda_0\sup_M|f|F_p(\theta(t))\\
&\quad+Cp\sqrt{F_2(\theta(t))}F_{p-1}(\theta(t))+\bigg(n+1+p+\frac{2n(p-1)}{p}\bigg)F_{p+1}(\theta(t)).
\end{aligned}
 \end{equation}
By H\"{o}lder's inequality we can estimate
$$\sqrt{F_2(\theta(t))}F_{p-1}(\theta(t))\leqslant CF_p(\theta(t))^\frac1p F_p(\theta)^\frac{p-1}{p}=CF_p(\theta(t)).$$
Furthermore, using the H\"{o}lder's inequality and Young's inequality we obtain
\[
 \begin{split}
  F_{p+1}(\theta(t))&\leqslant F_{\frac{p(n+1)}{n}}(\theta(t))^\frac{n}{p}F_p(\theta(t))^\frac{p-n}{p}\\
  &\leqslant \epsilon F_{\frac{p(n+1)}{n}}(\theta(t))^\frac{n}{n+1}+C(\epsilon) F_p(\theta(t))^\frac{p-n}{p-n-1}.
 \end{split}
\]
Substituting the two estimates above with a suitable $\epsilon$ into \eqref{DiffIneqFp}, we thus complete the proof of case 1.

\noindent{\bf Case 2.}\quad $\Ro\equiv0$.  In this case, the method in case 1 does not work since the CR-Yamabe constant $\mu(M)=0$. However, we may make use of the time-dependent bound for the conformal factor $u$ and Folland-Stein embedding theorem. Precisely speaking, we use Lemma \ref{Bdofu} and \eqref{FSEmbedding} to obtain
 $$
 \begin{aligned}
  \int_M\Big|\nabla_{\theta(t)}|\lambda(t)f&-\R|^{p/2}\Big|^2\dvg\geqslant C(T) \int_M\Big|\nabla_{\theta_0}|\lambda(t)f-\R|^{p/2}\Big|^2\dv\\
  &=C(T)\Big\||\lambda(t)f-\R|^{p/2}\Big\|_{\FS}^2+C(T)F_p(\theta(t))\\
  &\geqslant C(T) F_{\frac{p(n+1)}{n}}(\theta(t))^\frac{n+1}{n}+C(T)F_p(\theta(t)).
 \end{aligned}
 $$
 With help of the inequality above, we then follow the proof of case 1 to obtain the assertion.
\end{proof}

With all these, we can follow the exactly same scheme in \cite{Ho2} to prove the global existence of the flow \eqref{WebScalFlow}.
\begin{proposition}\label{GlobExist}
Given an initial smooth function $u_0\in X_f$, there exists a smooth solution $u(\cdot,t)$ of the flow \eqref{WebScalFlow} on $[0,+\infty)$.
\end{proposition}


\section{$L^p$ convergence}
In this section, we devote ourselves to proving the $L^p$ convergence of the flow \eqref{WebScalFlow}, that is $F_p(\theta(t))\rightarrow0,$ as $t\rightarrow+\infty$. To realize this, we first show the following differential inequality.
\begin{lemma}
 There exists a positive constant $C$, independent of time $t$, such that
 \begin{equation}\label{DiffIneqFp1}
  \frac{d}{dt}F_p(\theta(t))\leqslant CF_p(\theta(t)),~~2\leqslant p\leqslant n+1.
 \end{equation}
\end{lemma}
\begin{proof}
 Recalling \eqref{DerofFp} and using Lemmas \ref{VolBd}, \ref{lambdaBd}, \ref{DeroflambdaBd}, \ref{WebScalLowerBd}
 and  H\"{o}lder's inequality,  we obtain
 \[
  \begin{split}
   \frac{d}{dt}F_p(\theta(t))&\leqslant p\Lambda_0\Big(\sup_Mf\Big)F_p(\theta(t))+p\Big(\sup_Mf\Big)C\sqrt{F_2(\theta(t))}F_{p-1}(\theta(t))\\
   &\quad-\gamma(n+1-p)F_p(\theta(t))\\
   &\leqslant CF_p(\theta(t)).
  \end{split}
 \]
Here, we have used the fact that $\gamma<0$ and $p\leqslant n+1$. In addition, it is clear that $C$ is independent of time $t$.
\end{proof}

\begin{lemma}\label{LpConverge}
If either of two conditions below holds
\begin{itemize}
 \item[(i)] $\Ro>0$ and $0<p<\infty$;
 \item[(ii)] $\Ro\equiv0$ and $0<p\leqslant n+1$,
\end{itemize}
then  $F_p(\theta(t))\rightarrow0$, as $t\rightarrow+\infty$.
\end{lemma}
\begin{proof}
{\bf Case 1.}\quad $\Ro>0$.  In this case, one may follow the proof of \cite[Lemma 3.2]{Ho2}. Hence, we omit the details.

\noindent{\bf Case 2.}\quad $\Ro\equiv0$. Note that the volume of $(M, \theta(t))$ is bounded. Hence,  by the H\"{o}lder's inequality, it suffices to show the conclusion for any $2\leqslant p\leqslant n+1$. And in view of \eqref{DiffIneqFp1}, we only need to show that $F_p$ is integrable on $[0, +\infty)$. Indeed, if this is the case, then there would be a time sequence $(t_j)_j\subset(0, +\infty)$ such that $F_p(\theta(t_j))\rightarrow0$ as $j\rightarrow+\infty$. Integrating \eqref{DiffIneqFp1} from $t_j$ to $t$ with $t>t_j$ yields
 $$F_p(\theta(t))\leqslant F_p(\theta(t_j))+\int_{t_j}^tF_p(\theta(t))~\mathrm{d}t.$$
 Then sending $j$ to the infinity gives that $F_p(\theta(t))\rightarrow0$ as $t\rightarrow+\infty$. To see the integrability of $F_p(\theta(t))$, we will perform mathematical induction for $2\leqslant p\leqslant n+1$ with $n\geqslant 2$ (it is trivial when $n=1$). When $p=2$, the integrability on $[0,+\infty)$ follows from \eqref{EnergyBd}. Now, we assume that $F_k(\theta(t))$ is integrable for some $2<k<n+1$ and claim that $F_{k+1}(\theta(t))$ is integrable. Indeed, setting $p=k$ in \eqref{DerofFp}, we get the estimate
\[
   \int_M|\lambda(t)f-\R|^k(\R-\lambda(t)f)\dvg\leqslant CF_k(\theta(t))-\frac{d}{dt}F_k(\theta(t)).
 \]
 This together the induction assumption and \eqref{DiffIneqFp1} with $p=k$ implies that
 $$\int_0^\infty\int_M|\lambda(t)f-\R|^k(\R-\lambda(t)f)\dvg\mathrm{d}t\leqslant C.$$
By Lemma \ref{WebScalLowerBd}, we have
 \[
  \begin{split}
   \int_M|\lambda(t)f-&\R|^{k+1}\dvg\\
   &=\int_M|\lambda(t)f-\R|^k(\R-\lambda(t)f)\dvg\\
   &\quad+2\int_M|\lambda(t)f-\R|^k(\R-\lambda(t)f)^{-}\dvg\\
   &\leqslant\int_M|\lambda(t)f-\R|^k(\R-\lambda(t)f)\dvg+2(-\gamma) F_k(\theta(t)),
  \end{split}
 \]
 where $h^{-}=\max\{-h,0\}$ for a function $h$. From this estimate, we conclude that $F_{k+1}(\theta(t))$ is integrable on $[0,+\infty)$.
\end{proof}


\section{Blow-up analysis}

In this section, we perform the so-called blow-up analysis.  To do so, let us start with the CR version of concentration-compactness theorem.

\subsection{Concentration-Compactness}
Consider the semi-linear equation
\begin{equation}\label{SLpde}
 -\Delta_{\theta}u+hu=fu^{1+2/n}
\end{equation}
and the associated functional
$$I_f[u]=\frac12\int_M \left(|\nabla u|_{\theta}^2+hu^2\right)\dvg-\frac{n}{2(n+1)}\int_Mfu^{2+2/n}\dvg.$$
It is known (c.f. \cite{Jerison&Lee2}) that any nonnegative and nontrivial solution of the equation on Heisenberg group
\begin{equation}\label{EqonHeisenberg}
 -\Delta_{\theta_{\mathbb{H}^n}}u=u^{1+2/n}
\end{equation}
must be of the form
\begin{equation}\label{StandardBubble}
 \U_\zeta^\mu=\mu^n\Big|(s-s_0)+i\Big(|z-z_0|^2+(\mu/n)^2\Big)\Big|^{-n},
\end{equation}
where $\zeta=(z_0, s_0)$ with $z_0\in\mathbb{C}^n, s_0\in\mathbb{R}$ and $\mu>0$. Furthermore, the functions of the form \eqref{StandardBubble} are, up to left translations and dilations $\T_\delta:$ $(z, s)\mapsto(\delta z, \delta^2s)$, the only extremals of the the sharp  Sobolev inequality on Heisenberg group
\begin{equation}\label{SharpCRSobolev}
\bigg( \int_{\mathbb{H}^n}|u|^{2+2/n}\dvg\bigg)^{\frac{n}{n+1}}\leqslant K_n\int_{\mathbb{H}^n}|\nabla u|_{\theta}^2\dvg,
\end{equation}
where $K_n$ is the same constant as in \eqref{CRYamabeIneq}. The energy on Heisenberg group is denoted by
$$
E(u)=\frac12\int_{\mathbb{H}^n}|\nabla u|_{\theta_{\mathbb{H}^n}}^2\dvh-\frac{n}{2(n+1)}\int_{\mathbb{H}^n}|u|^{2+2/n}\dvh.
$$
If we choose $\mu>0$ such that
\begin{equation}\label{normalize}
\int_{\mathbb{H}^n}|\U_\zeta^\mu|^{2+2/n}\dvh=K_n^{-(n+1)},
\end{equation}
then by \eqref{EqonHeisenberg} and \eqref{SharpCRSobolev}  we get
$$E\Big(\U_\zeta^\mu\Big)=\frac{K_n^{-(n+1)}}{2(n+1)}.$$

\begin{theorem}\label{concentration}
 Let $(u_k)_{k\in\mathbb{N}}$ be an $\FS$-bounded Palais-Smale sequence of nonnegative functions for $I_f$, where $h, f\in C^\infty(M)$ with $\sup_{M} f>0$.  Then, there exist $m\in\mathbb{N}$ and a nonnegative weak solution $u_\infty\in \FS$ of \eqref{SLpde}, such that, up to a subsequence, the following conclusions hold for $1\le j\le m$:

  \noindent{\upshape(i)} there exist $x_k^j\rightarrow x^j\in M$ such that $f(x^j)>0$ as $k\rightarrow+\infty$;

  \noindent{\upshape(ii)} there exist $\epsilon_k^j>0$ such that $\epsilon_k^j\rightarrow0$ and
  \begin{equation}\label{limitofepsilonkj}
    \frac{\epsilon^i_k}{\epsilon^j_k}+\frac{\epsilon^j_k}{\epsilon_k^i}+\frac{\mathrm{d}(x_k^i,x_k^j)^2}{\epsilon^i_k\epsilon^j_k}\rightarrow+\infty,~\mbox{for all}~i\neq j,~\mbox{as}~k\rightarrow+\infty,
  \end{equation}
where $\mathrm{d}(\cdot,\cdot)$ is the Carnot-Carath\'{e}odory distance on $M$ with respect to the contact form $\theta_0$;

  \noindent{\upshape(iii)} with
  $$\U_k^j(x)=\big(\epsilon_k^j\big)^{-n}f(x^j)^{-n/2}\U^j\bigg(\T_{(\epsilon_k^j)^{-1}}\exp_{x_k^j}^{-1}(x)\bigg)$$
  for $x\in\exp_{x_k^j}(B^+_{2r_0}(0))$, there holds
  \begin{equation}\label{limitofpalaissmalesequence}
 \Big\|u_k-u_\infty-\sum_{j=1}^m\eta_k^j\U_k^j\Big\|_{\FS}\rightarrow0,~~\mbox{as}~ k\rightarrow+\infty,
  \end{equation}
  where $\U^j=\U_{\zeta_j}^{\mu_j}$ satisfies \eqref{normalize}, $r_0>0$ is small, $\exp_{x_k^j}: B_{2r_0}(0)\subset\mathbb{H}^n\mapsto M$ are the local charts centered at $x_k^j\in M$, $\T$ is the dilation as before, and $\eta_k^j$ are smooth cut-off functions such that $\eta_k^j\equiv1$ in $\exp_{x_k^j}(B_{r_0}(0))$ and $\eta_k^j\equiv0$ in $M\backslash\exp_{x_k^j}(B_{2r_0})$. Moreover,
  \begin{equation}\label{limitofenergy2}
  I_f[u_k]=I_f[u_\infty]+\frac{K_n^{-(n+1)}}{2(n+1)}\sum_{j=1}^mf(x^j)^{-n}+o(1),
  \end{equation}
  where $o(1)\rightarrow0$ as $k\rightarrow+\infty$.
\end{theorem}
\begin{proof}
To prove Theorem \ref{concentration}, one can just follow the proof of Proposition 5.1 in \cite{HSW} which is contained in Appendix B of \cite{HSW};
the only difference between Theorem \ref{concentration} here and Proposition 5.1 in \cite{HSW} is that we are considering (\ref{SLpde}) in which $f$ is a function, while in Proposition 5.1 in \cite{HSW} they are considering
(\ref{SLpde}) in which $f$ is a positive constant.
Note that the assertion  $f(x_j)>0$ in Theorem \ref{concentration} (i)
follows from the fact that if
$$ -\Delta_{\theta_{\mathbb{H}^n}}u=cu^{1+2/n}$$
has a positive solution for some constant $c$, then $c$
must be positive.
\end{proof}

\subsection{Weak solutions}
A simple characterization of the asymptotic behavior of energy $\E(u)$ is given as below
\begin{lemma}\label{Limitoflambda}
Let $u(\cdot, t)$ be a smooth solution of the flow \eqref{WebScalFlow}. Then there hold
\begin{itemize}
 \item[(i)] $\E(u(t))=\lambda(t)+o(1)$, where $o(1)\rightarrow0$ as $t\rightarrow\infty$;
 \item[(ii)] there exists a constant $\lambda_\infty$ such that $\lambda(t)\rightarrow\lambda_\infty$ as $t\rightarrow\infty$.
\end{itemize}
\end{lemma}
\begin{proof}
 (i)\quad From \eqref{WebScal}, \eqref{UnitIntoff} and Lemma \ref{LpConverge} it follows that
 $$\E(u(t))=\int_M\R\dvg=\int_M\Big(\R-\lambda(t)f\Big)\dvg+\lambda(t)=\lambda(t)+o(1),$$
  where $o(1)\rightarrow0$, as $t\rightarrow\infty$.
  \smallskip

 \noindent(ii)\quad By Lemma \ref{EnegyDecay}, we know that $\E(u(t))$ is monotone non-increasing. In addition, \eqref{CRYamabeIneq} implies that $\E(u(t))$ has a lower bound. Hence $\E(u(t))$ converges as $t\rightarrow\infty$. Then the assertion in (i) concludes that $\lambda(t)\rightarrow\lambda_\infty$ for some constant $\lambda_\infty$ as $t\rightarrow\infty$.
 \end{proof}

 For a nonnegative function  $w\in\FS$, we define the energy functional
 \begin{equation*}
  J_f[w]=\frac12\int_M\left((2+2/n)|\nabla w|_{\theta_0}^2+\Ro w^2-\frac{n}{n+1}\lambda_\infty fw^{2+2/n}\right)\dv.
 \end{equation*}
Then, it is easy to compute
\begin{equation}\label{dJ}
 \mathrm{d}J_f[w](\phi)=\int_M(2+2/n)\langle\nabla w, \nabla\phi\rangle_{\theta_0}+\Ro w\phi-\lambda_\infty fw^{1+2/n}\phi\dv.
\end{equation}
Moreover, we have
\begin{lemma}\label{J_f}
 For any smooth solution $u(\cdot,t)$ of the flow \eqref{WebScalFlow}, there hold
  $$J_f[u(t)]=\frac12\E(u(t))-\frac{n}{2(n+1)}\lambda_\infty\quad\text{and}\quad J_f[u(t)]=\frac{1}{2(n+1)}\lambda_\infty+o(1),$$
  where $o(1)\rightarrow0$ as $t\rightarrow+\infty$.
\end{lemma}
\begin{proof}
  From \eqref{UnitIntoff} and definition of $J_f$, it follows that
  \begin{eqnarray*}
  J_f[u(t)]&=&\int_M(1+1/n)|\nabla u|_{\theta_0}^2+\frac12\Ro u^2-\frac{n}{2(n+1)}\lambda_\infty fu^{2+2/n}\dv\\
  &=&\frac12\E(u(t))-\frac{n}{2(n+1)}\lambda_\infty.
  \end{eqnarray*}
  Rewriting the above equality as
  $$J_f[u(t)]=\frac12(\E(u(t))-\lambda_\infty)+\frac{1}{2(n+1)}\lambda_\infty.$$
 Thanks to Lemma \ref{Limitoflambda}(i) that $\E(u(t))\to\lambda_\infty$ as $t\rightarrow+\infty$, we thus finished the proof.
\end{proof}

Let $\{t_k:k\in\mathbb{N}\}$ be a time sequence such that $t_k\rightarrow+\infty$ as $k\rightarrow+\infty$. Also, we set $u_k:=u(t_k), \theta_k:=\theta(t_k), \lambda_k:=\lambda(t_k)$ and $R_k:=R_{\theta_k}$. Then we will show the following:
\begin{lemma}
  \label{palaissmalesequence}
  $(u_k)_k$ is a Palais-Smale sequence for the energy functional $J_f$.
\end{lemma}
 \begin{proof}
  We need to show that \\
  $\bullet$~~The sequence $(J_f[u_k])_k$ is bounded;\\
  $\bullet$~~$dJ_f[u_k]\rightarrow0$ strongly in $\FS^{-1}$ as $k\rightarrow+\infty$.

  The boundedness of $(J_f[u_k])_k$ is straightforward by Lemma \ref{J_f} and \eqref{EnergyBd}. Hence, we are left to check the strong convergence $dJ_f[u_k]\rightarrow0$ in $\FS^{-1}$  as $k\rightarrow+\infty$. To see this, we fix any $\phi\in \FS$. Then \eqref{WebScal},  \eqref{FSEmbedding}
  and H\"older's inequality  yield
  \begin{eqnarray*}
    dJ_f[u_k](\phi)&=&\int_{M}\Big[\big(-(2+2/n)\Delta_{\theta_0} u_k+\Ro u_k-\lambda_\infty fu_k^{1+2/n}\Big]\phi\dv\\
    &=&\int_{M}(R_k-\lambda_\infty f)u_k^{1+2/n}\phi\dv\\
    &\le&||R_k-\lambda_k f+(\lambda_k-\lambda_\infty) f||_{L^\frac{2(n+1)}{n+2}(M,\theta_k)}||\phi||_{L^\frac{2(n+1)}{n}( M, \theta_0)}\\
    &\le&C\Big(F_{\frac{2(n+1)}{n+2}}(\theta_k))^\frac{n+2}{2(n+1)}+|\lambda_k-\lambda_\infty|\sup_{M}fc_2^\frac{n+2}{2(n+1)}\Big)~||\phi||_{\FS}.
  \end{eqnarray*}
  Notice that $2(n+1)/(n+2)<2$, it follows from Lemma \ref{LpConverge} that $F_{\frac{2n}{n+1}}(g_k)=o(1)$ as $k\rightarrow+\infty$. Moreover, by Lemma \ref{Limitoflambda}, we have $\lambda_k-\lambda_\infty=o(1)$, as $k\rightarrow+\infty$. Hence, the assertion follows.
 \end{proof}

 \begin{lemma}\label{weaksolution}
  There exists $u_\infty\in \FS$ such that $u_k$ converges to $u_\infty$ weakly in $\FS$ as $k\rightarrow+\infty$. Moreover, if $u_\infty\not\equiv0$, then \eqref{PWC} possess a positive solution.
\end{lemma}
 \begin{proof} Since $(u_k)_k$ is bounded in $\FS$ by Corollary \ref{S12Bd}, there exists $u_\infty\in S_1^2(M)$ such that, up to a subsequence, $u_k\rightharpoonup u_\infty$ weakly in $\FS$ as $k\rightarrow+\infty$. By Lemma \ref{palaissmalesequence}, we have $dJ_f[u_k]=o(1)$ in $\FS^{-1}$. Then, it follows from \eqref{dJ} that
  $$\int_M(2+2/n)\langle\nabla u_k,\nabla\phi\rangle_{\theta_0}+\Ro u_k\phi~\dv=\int_{ M}\lambda_\infty fu_k^{1+2/n}\phi~\dv+o(1).$$
  Letting $k\rightarrow+\infty$, we obtain that $u_\infty$ solves the Eq.\eqref{PWC} in the weak sense
  $$\int_M(2+2/n)\langle\nabla u_\infty,\nabla\phi\rangle_{\theta_0}+\Ro u_\infty\phi~\dv=\int_{ M}\lambda_\infty fu_\infty^{1+2/n}\phi~\dv.$$
 Moreover, standard regularity theory (c.f. Theorem 5.15 in \cite{Jerison&Lee1}) implies that $u_\infty$ is smooth if $f$ is smooth. Hence, $u_\infty$ satisfies
\begin{equation}\label{5.9}
-(2+2/n)\Delta_{\theta_0}u_\infty+\Ro u_\infty=\lambda_\infty fu_\infty^{1+2/n}.
\end{equation}
  Now, if $u_\infty\not\equiv0$, we can apply Proposition A.1 in \cite{Ho1}
  to conclude that $u_\infty>0$.

 Next, we claim that $\lambda_\infty>0$. If $\Ro>0$, it follows from
 (\ref{UnitIntoff}), \eqref{CRYamabeIneq} and Lemma \ref{VolBd} that
 $$\E(u(t))\geqslant\mu(M)\Vol(M,\theta(t))^\frac{n}{n+1}\geqslant\mu(M)\Big(\sup_Mf\Big)^{-\frac{n}{n+1}}.$$
 Then by Lemma \ref{Limitoflambda} we get that
 $$\lambda_\infty=\lim_{t\rightarrow\infty}\E(u(t))\geqslant\mu(M)\Big(\sup_Mf\Big)^{-\frac{n}{n+1}}>0.$$
 If $\Ro\equiv0$, we first note that there exists a  subsequence, still denoted by $\{t_k\}$,  such that $u_k\rightharpoonup u_\infty$ weakly in $\FS$ and strongly in $L^2(M)$. Now, our argument is split into two cases.

 \noindent{\bf Case 1}. $u_\infty\equiv0$. In this case, using  (\ref{UnitIntoff}), \eqref{FSEmbedding}, Lemma \ref{VolBd} and the fact that $\|u_k\|_{L^2(M)}\rightarrow0$ as $k\rightarrow\infty$, we obtain
 \[
  \begin{split}
   \int_M|\nabla_{\theta_0}u_k|^2\dv&\geqslant{\mathscr C}^{-1}\Vol(M,\theta_k)^\frac{n}{n+1}-\int_Mu_k^2\dv\\
   &\geqslant {\mathscr C}^{-1}\Big(\sup_Mf\Big)^{-\frac{n}{n+1}}+o(1).
  \end{split}
 \]
 This implies that
 \[
  \begin{split}
   \lambda_\infty&=\lim_{\nu\rightarrow\infty}\E(u_k)=\lim_{k\rightarrow\infty}\int_M\big(2+2/n\big)|\nabla_{\theta_0}u_k|^2\dv\\
   &\geqslant\big(2+2/n\big){\mathscr C}^{-1}\Big(\sup_Mf\Big)^{-\frac{n}{n+1}}>0.
  \end{split}
 \]

 \noindent{\bf Case 2}. $u_\infty\not\equiv0$. In this case, if follows from \eqref{5.9} that $u_\infty>0$ solves
 $$-(2+2/n)\Delta_{\theta_0}u_\infty=\lambda_\infty fu_\infty^{1+2/n}.$$
To prove that $\lambda_\infty>0$, we suppose on contrary $\lambda_\infty=0$,
which gives $\Delta_{\theta_0}u_\infty=0$. Hence, $u_\infty$ must be a positive constant, say $c>0$.  Since $\lambda_\infty=0$,
it follows from Lemma \ref{Limitoflambda}(i) that
 $$\int_M|\nabla_{\theta_0}\big(u_k-u_\infty\big)|^2\dv=\int_M|\nabla_{\theta_0}u_k|^2\dv=\frac{n\E(u_k)}{2n+2}\rightarrow0$$
 as $k\rightarrow\infty$.
 This together with \eqref{FSEmbedding} implies that
 \[
  \begin{split}
   &\bigg(\int_M|u_k-u_\infty|^{2+2/n}\dv\bigg)^\frac{n}{n+1}\\
   &\leqslant{\mathscr C}^{-1}\bigg(\int_M|\nabla_{\theta_0}\big(u_k-u_\infty\big)|^2\dv
   \quad+\int_M|u_k-u_\infty|^2\dv\bigg)\rightarrow 0
  \end{split}
 \]
 as $k\to\infty$.
This together with \eqref{UnitIntoff} implies that $\int_Mfu_\infty^{2+2/n}\dv=1$. Since $u_\infty\equiv c>0$, we have $\int_Mf\dv=c^{-2-2/n}>0$, contradicting  the assumption $\int_M f\dv<0$. This proves the claim, i.e. $\lambda_\infty>0$.

 Therefore, if $u_\infty\not\equiv0$, then $\lambda_\infty^{n/2}u_\infty$ produces a positive solution to  \eqref{PWC}.
\end{proof}

 Lemma \ref{weaksolution} shows that Theorems \ref{main1} and \ref{main2} will immediately follow provided that we can show that there exists a time sequence $(t_k)_k$ such that $ u_k\rightharpoonup u_\infty\not\equiv0$ weakly in $\FS$. However, this is generally hard to be proved via direct estimates. Hence, we shall adopt a contradiction argument. To do so, we assume, from now on, that $u(t)\rightharpoonup 0$ weakly in $\FS$ as $t\rightarrow+\infty$. Under this assumption, it is easy to see from \eqref{FSEmbedding} and Lemma \ref{VolBd} that $u_k$ cannot  converge to $0$ strongly in $\FS$, and the concentration phenomenon will occur.

\subsection{Single concentration point}\label{SingleConcenPoint}

Set
$$\lambda_*=\frac{2(n+1)}{n}\Big(\sup_Mf\Big)^{-\frac{n}{n+1}}K_n^{-1},$$
where $K_n=1/(2\pi n^2)$ is the optimal constant appeared in (\ref{CRYamabeIneq}).
Then we have the following:
\begin{lemma}
 \label{LowBdoflambda}
 $\lambda_\infty\geqslant\lambda_*$.
\end{lemma}
\begin{proof}
Since $u\rightharpoonup0$ weakly in $\FS$, $u\rightarrow0$ strongly in $L^2(M)$ as $t\rightarrow0$. This fact together with Lemma \ref{Limitoflambda} implies that
\begin{eqnarray}\label{LimitofE}
 \lambda_\infty&=&\lim_{t\rightarrow+\infty}\E(u(t))=\lim_{t\rightarrow+\infty}\int_M(2+2/n)|\nabla u|^2_{\theta_0}+\Ro u^2\dv\nonumber\\
 &=&\lim_{t\rightarrow+\infty}\int_M(2+2/n)|\nabla u|^2_{\theta_0}\dv
\end{eqnarray}
On the other hand, given $\epsilon>0$, by \eqref{UnitIntoff}, \eqref{CRYamabeIneq}, Lemma \ref{VolBd} and the fact that $\|u\|_{L^2}=o(1)$ we get
\begin{eqnarray}\label{LowBdofE}
 \int_M(2+2/n)|\nabla u|^2_{\theta_0}\dv&\geqslant&\frac{2(n+1)}{n(K_n+\epsilon)}\Big(\mathrm{vol}(M,\theta(t))\Big)^{\frac{n}{n+1}}+o(1)\nonumber\\
 &\geqslant&\frac{2(n+1)}{n(K_n+\epsilon)}\big(\sup_Mf\big)^{-\frac{n}{n+1}}+o(1)
\end{eqnarray}
Combining \eqref{LimitofE} and \eqref{LowBdofE} yields
$$\lambda_\infty\geqslant\frac{2(n+1)}{n(K_n+\epsilon)}\big(\sup_Mf\big)^{-\frac{n}{n+1}}.$$
Letting $\epsilon$ to zero, we thus complete the proof.
\end{proof}

With a suitable choice of initial data, we can guarantee the single concentration point of the flow.
\begin{lemma}\label{SingleConcen}
 If the initial data $u_0$ satisfies $\E(u_0)<2^\frac{1}{n+1}\lambda_*$, then $m=1$, where $m$ is the positive integer in Theorem \ref{concentration}.
\end{lemma}
\begin{proof}
 First, note that
 $J_f[w]=(2+2/n)I_{\frac{n}{2(n+1)}\lambda_\infty f}[w].$
 Thus, we can apply Theorem \ref{concentration} to the functional $J_f[u(t)]$ to obtain
 $$J_f[u(t)]=\frac{K_n^{-(n+1)}}{2(n+1)}\sum_{j=1}^m\bigg(\frac{2(n+1)}{n}\bigg)^{n+1}\Big(\lambda_\infty f(x_j)\Big)^{-n}+o(1).$$
 This together with Lemma \ref{J_f} implies that
 \begin{eqnarray}\label{lambdainfty}
  \frac{\lambda_\infty}{2(n+1)}&=&\lim_{t\rightarrow+\infty}J_f[u(t)]=\frac{K_n^{-(n+1)}}{2(n+1)}\sum_{j=1}^m\bigg(\frac{2(n+1)}{n}\bigg)^{n+1}\Big(\lambda_\infty f(x_j)\Big)^{-n}\\
  &\geqslant&\frac{mK_n^{-(n+1)}}{2(n+1)}\bigg(\frac{2(n+1)}{n}\bigg)^{n+1}\lambda_\infty^{-n}\Big(\sup_Mf\Big)^{-n}.\nonumber
 \end{eqnarray}
Hence
$$\lambda_\infty^{\frac{1}{n+1}}\geqslant m^\frac{1}{n+1}\frac{2(n+1)}{n}\Big(\sup_Mf\Big)^{-\frac{n}{n+1}}K_n^{-1}=m^\frac{1}{n+1}\lambda_*.$$
On the other hand, it follows from Lemma \ref{Limitoflambda} and \eqref{EnergyBd} that
$$\lambda_\infty=\lim_{t\rightarrow+\infty}\E(u(t))\leqslant\E(u_0)<2^\frac{1}{n+1}\lambda_*.$$
Combining the estimates above yields
$m<2$, i.e. $m=1$.
\end{proof}

We have the following convention:
For the rest of the paper, our initial data is chosen to satisfy
\begin{equation}\label{initialdata}
 \int_Mfu_0^{1+2/n}\dv=1~~\mathrm{and}~~\E(u_0)<2^\frac{1}{n+1}\lambda_*.
\end{equation}
With such a choice of initial data, we know by Lemma \ref{SingleConcen} that there would be only one concentration point denoted by $a_\infty$. Then \eqref{lambdainfty} can be reduced as
$$\frac{\lambda_\infty}{2(n+1)}=\frac{K_n^{-(n+1)}}{2(n+1)}\bigg(\frac{2(n+1)}{n}\bigg)^{n+1}\Big(\lambda_\infty f(a_\infty)\Big)^{-n}.$$
Solving this  gives
\begin{equation}\label{laminfty}
\lambda_\infty=\frac{2(n+1)}{n}K_n^{-1}f(a_\infty)^{-\frac{n}{n+1}}.
\end{equation}
Moreover, we may choose suitable cut-off function $\eta(t)$, $\epsilon^*(t)>0$ with $\epsilon^*(t)\rightarrow0$ and $a^*(t)\in M$ with $a^*(t)\rightarrow a_\infty$ such that as $t\rightarrow\infty$
\begin{equation}\label{MinusOneBubble}
\Big\|u(t)-\varphi^*(t)\Big\|_{\FS}\rightarrow0,
\end{equation}
where
$$\varphi^*(t, x)=\bigg(\frac{n\lambda_\infty f(a_\infty)}{2(n+1)}\bigg)^{-n/2}\eta(t)(\epsilon^*(t))^{-n}\U(\T_{{\epsilon^*(t)}^{-1}}\exp_{a^*(t)}^{-1}(x))$$
and
$$\U=\big|s+i\big(|z|^2+(1/n)^2\big)\big|^{-n}.$$
\subsection{Bubbles}~~
\smallskip

\noindent{\bf Case 1}\quad$\mu(M)>0$. \quad Given any $a\in M$, we can find a contact form $\theta_a={\omega^*_a}^{2/n}\theta_0$ conformal to $\theta_0$ such that it induces CR normal coordinates $(z,s)$ around $a\in M$.  Let $G^*_a$ be the Green's function with pole at $a$. Then we have
\begin{equation}\label{5.17}
 -(2+2/n)\Delta_{\theta_a}G^*_a(x)+R^*_{\theta_a}G^*_a(x)=\delta_a.
\end{equation}
where $R^*_{\theta_a}$ is the Webster scalar curvature of $\theta_a$.
Moreover, the Green's function satisfies the estimates (see Appendix A in \cite{HSW})
\begin{equation}
|G^*_a(x)-\rho_a(x)^{-2n}-A^*_a|\leqslant C\rho_a(x)
\end{equation}
and
\begin{equation}
 |\nabla(G^*_a(x)-\rho_a(x)^{-2n})|_{\theta_a}\leqslant C,
\end{equation}
where $A^*_a$ is the CR mass satisfying $\inf_{a\in M}A^*_a>0$ under the assumptions in Theorem \ref{main1}, and $\rho_a(x)$ the distance in the CR normal coordinates at $a$ defined as
$$\rho_a(x)=(s^2+|z|^4)^{1/4}. $$
We define $\chi:~~\mathbb{R}\mapsto[0,1]$ to be a smooth cut-off function satisfying $\chi(\tau)=1$ for $\tau\leqslant1$ and $\chi(\tau)=0$ for $\tau\geqslant2$. For $\delta>0$, we then define $\chi_\delta(\tau)=\chi(\tau/\delta)$. And, for $\epsilon\ll\delta$ we let
\begin{equation}\label{bubble}
 \varphi_{a,\epsilon}=\omega^*_a(x)\psi_{a,\epsilon},
\end{equation}
where
\begin{equation}\label{psi}
\psi_{a,\epsilon}=\epsilon^n\Bigg[\frac{\chi_{\delta}(\rho_a)}{\big(s^2+((\epsilon/n)^2+|z|^2)^2\big)^{n/2}}+\big(1-\chi_{\delta}(\rho_a)\big)G^*_a\Bigg].
\end{equation}
\smallskip

\noindent{\bf Case 2}\quad$\mu(M)=0$.\quad Given any $a\in M$, we consider the local charts $(z, s)$ around $a$.  Motivated by \cite{ES}, we let
$$\varphi_{a,\epsilon}=\epsilon^n\Bigg[\frac{\chi_{\delta}(\rho_a)}{\big(s^2+((\epsilon/n)^2+|z|^2)^2\big)^{n/2}}+\big(1-\chi_{\delta}(\rho_a)\big)\Big(\rho_a^{-2n}+\Lambda\Big)\Bigg],$$
where $\Lambda$ is a positive constant to be determined, $\chi_\delta$ and $\rho_a$ are the same notations as in the Case 1.
\smallskip

To provide a unified proof for both cases, we hereafter let
\[
\omega_a=\left\{
\begin{array}{ll}
 \omega^*_a,&\mbox{if }\mu(M)>0,\\
 1,&\mbox{if }\mu(M)=0,
\end{array}
\right.
R_{\theta_a}=\left\{
\begin{array}{ll}
 R^*_{\theta_a},&\mbox{if }\mu(M)>0,\\
 0,&\mbox{if }\mu(M)=0,
\end{array}
\right.
\]
and
\[
G_a=\left\{
\begin{array}{ll}
 G^*_a,&\mbox{if }\mu(M)>0,\\
 \rho_a^{-2n}+\Lambda,&\mbox{if }\mu(M)=0,
\end{array}
\right.
A_a=\left\{
\begin{array}{ll}
 A^*_a,&\mbox{if }\mu(M)>0,\\
 \Lambda,&\mbox{if }\mu(M)=0.
\end{array}
\right.
\]
We have the following:
\begin{definition}
 For $u\in \FS$ and  sufficiently small $\mu>0$, we define
 \begin{eqnarray*}
 D_u(\mu)=\bigg\{(\alpha, \epsilon, a)\in\Big(\mathbb{R}_+,\mathbb{R}_+,M\Big):~~\epsilon,~~\bigg|\frac{n\lambda\alpha^{2/n} f(a)}{2(n+1)}-1\bigg|,~~ \|u-\alpha\varphi_{a,\epsilon}\|<\mu\bigg\}
 \end{eqnarray*}
\end{definition}
From \eqref{MinusOneBubble}, it follows that
$$D_{u(t)}(\mu)\neq\emptyset$$
for all sufficiently small $\mu>0$ and all sufficiently large $t$.

\begin{lemma}\label{ConformalOperatorOnBubble}
There holds
\begin{eqnarray*}
\SL_{\theta_0}\varphi_{a,\epsilon}&=&(2+2/n)\varphi_{a,\epsilon}^{1+2/n}+(2+2/n)\epsilon^nA_a\omega_a^{1+2/n}\Delta_{\theta_a}\chi_{\delta}(\rho_a)\\
&&+\omega_a^{1+2/n}\epsilon^n\Big(I_1+I_2+I_3+I_4+I_5+I_6+I_7+I_8\Big),\\
\end{eqnarray*}
where
\[
\begin{array}{l}
I_1=-(2+2/n)\chi_\delta(\rho_a)\bigg[\Delta_{\theta_a}\Big(\dfrac{1}{(s^2+((\epsilon/n)^2+|z|^2)^2)^{n/2}}\Big)+\dfrac{\epsilon^2}{(s^2+((\epsilon/n)^2+|z|^2)^2)^{n/2+1}}\bigg], \\
I_2=\dfrac{\chi_\delta(\rho_a)R_{\theta_a}}{(s^2+((\epsilon/n)^2+|z|^2)^2)^{n/2}},\\[0.8em]
I_3=(2+2/n)\Delta_{\theta_a}\chi_\delta(\rho_a)\Big(G_a-\rho_a^{-2n}-A_a\Big),\\
I_4=-(2+2/n)\Delta_{\theta_a}\chi_\delta(\rho_a)\Big(\dfrac{1}{(s^2+((\epsilon/n)^2+|z|^2)^2)^{n/2}}-\rho_a^{-2n}\Big),\\
I_5=2(2+2/n)\Big\langle\nabla\chi_\delta(\rho_a), \nabla(G_a-\rho_a^{-2n})\Big\rangle_{\theta_a},\\
I_6=-2(2+2/n)\Big\langle\nabla\chi_\delta(\rho_a), \nabla\Big(\dfrac{1}{(s^2+((\epsilon/n)^2+|z|^2)^2)^{n/2}}-\rho_a^{-2n}\Big)\Big\rangle_{\theta_a},\\
I_7=-(2+2/n)\epsilon^2\bigg[\Big(\dfrac{\chi_{\delta}(\rho_a)}{(s^2+((\epsilon/n)^2+|z|^2)^2)^{n/2}}+(1-\chi_{\delta}(\rho_a))G_a\Big)^{1+2/n}\\
\qquad\qquad\qquad\qquad-\dfrac{\chi_\delta(\rho_a)}{(s^2+((\epsilon/n)^2+|z|^2)^2)^{n/2+1}}\bigg],\\
I_8=\Big(1-\chi_\delta(\rho_a)\Big)\Big[(2+2/n)\Delta_{\theta_a}G_a-R_{\theta_a}G_a\Big].
\end{array}
\]
Hereafter, $\SL_{\theta_0}$ is the conformal sub-Laplacian with respect to $\theta_0$.
\end{lemma}
\begin{proof}
 See \cite[Proposition A.1]{HSW}.
\end{proof}
\smallskip

Now, we will choose some suitable parameters $\alpha, \epsilon$ and $a$ to simplify the proceeding arguments.
\begin{proposition}\label{OptimalChoice}
 For each $t>0$, there exists $\mu_0=\mu_0(t)$ with $\mu_0(t)\rightarrow0$ as $t\rightarrow+\infty$ such that
 $$\underset{(\tilde{\alpha},\tilde{\epsilon},\tilde{a})\in D_{u(t)}(2\mu_0)}{\inf}\int_Mu(t)^{2/n}|u(t)-\tilde{\alpha}\varphi_{\tilde{a},\tilde{\epsilon}}|^2\dv$$
 admits a unique minimizer $\big(\alpha(t),\epsilon(t), a(t)\big)\in D_{u(t)}(\mu_0)$.
\end{proposition}
\begin{proof}
 Since the proof is almost identical with the one in \cite[Proposition 3.10]{Ma}, we omit the proof and refer the readers to \cite{Ma}.
\end{proof}

For simplicity, we set
$$\varphi=\varphi_{a,\epsilon},\quad v=u-\alpha\varphi.$$
\begin{definition}\label{DerofBubbles}
 For any $\epsilon>0$ and $a\in M$, we define
 \begin{itemize}
  \item [(i)] $(\rm{d}_1, \rm{d}_2, \rm{d}_3)=(1, \epsilon\partial_\epsilon, \epsilon\nabla_a)$;
  \item [(ii)] $\phi_k=\rm{d}_k\varphi$ for $k=1,2,3$.
 \end{itemize}
 Here, we regard $a$ as a point in $\mathbb{H}^n$
 by using CR normal coordinates $(z,s)$, and
 $\nabla_a$ is the usual sub-gradient in $\mathbb{H}^n$.
\end{definition}
The next lemma  provides us estimates on $\phi_k$, which is vital for proving the main theorems.
\begin{lemma}\label{EstimatesofBubbles}
 One has
 \begin{itemize}
  \item [(i)] $|\rm{d}_k\phi_{\it l}|\leqslant C\varphi$ \mbox{ for }$k, l=1, 2, 3$;
  \item [(ii)] $\displaystyle\int_M\varphi^{2/n}\phi_k^2\dv=c_k+O(\epsilon^{2n+2}+\epsilon^4)$ \mbox{ for }$k=1, 2, 3$;
  \item [(iii)] $\displaystyle\int_M\varphi^{1+2/n}\phi_k\dv=O(\epsilon^{2n+2})\mbox{ for }k=2, 3$;
  \item [(iv)] $\displaystyle\int_M\varphi^{2/n}\phi_k\phi_l\dv=O(\epsilon^{2n+2}+\epsilon^4)\mbox{ for all }k\neq l$.
 \end{itemize}
Here $c_k$ are some positive constants.
\end{lemma}

Due to the length of the proof, Lemma \ref{EstimatesofBubbles}
will be proved in Appendix \ref{A}.

\begin{lemma}\label{SLphikv}
Suppose that $n=1$.
 For $t$ large, one has
 \begin{itemize}
  \item [(i)] $\displaystyle\int_M \SL_{\theta_0}\phi_k v\dv=O\Big(\delta^{6-2n}\epsilon^{2n}+\epsilon^6\Big)+O(\|v\|^2)$,
  \item[(ii)] $\displaystyle
   \int_Mfu^{1+2/n}\phi_k\dv=\int_Mf(\alpha\varphi)^{1+2/n}\phi_k\dv$

  \hspace{14em} $
   +O\Big(|\nabla f(a)|_{\theta_0}^2\epsilon^2+\|v\|^2\Big)
  $.
 \end{itemize}
\end{lemma}
\begin{proof}
 (i) In view of Lemma \ref{ConformalOperatorOnBubble}, we have
 \begin{eqnarray*}
  &&\int_M\SL_{\theta_0}\phi_kv\dv=\int_M\Big({\rm d}_k\SL_{\theta_0}\varphi\Big) v\dv\\
  &&\qquad=\int_M\bigg[(2+2/n){\rm d_k}\varphi^{1+2/n}+(2+2/n){\rm d}_k\Big(\epsilon^nA_a\omega_a^{1+2/n}\Delta_{\theta_a}\chi_\delta(\rho_a)\Big)\\
  &&\qquad\qquad+\sum_{i=1}^8{\rm d}_k\Big(\omega_a^{1+2/n}\epsilon^nI_i\Big)\bigg]v\dv.
 \end{eqnarray*}
{\bf Case $k=1$}. First, notice that $R_{\theta_a}=0$ and
$$\Delta_{\theta_a}G_a=\bigg[\Delta_{\theta_{\mathbb{H}^n}}+O(\rho_a^2)\bigg]\rho_a^{-2n}=O(1),$$
if $\mu(M)=0$ and $n=1$. Thus, $I_8=O(1)1_{\delta\leqslant\rho_a\leqslant2\delta}$ if $\mu(M)=0$ and $n=1$. On the other hand,
by (\ref{5.17}),
$I_8=0$ if $\mu(M)>0$.
Recalling that ${\rm d}_1=1$, one may follow the proof in \cite[Proposition A.1]{HSW} to obtain
\begin{eqnarray}\label{Lphi1v}
  &&\int_M\SL_{\theta_0}\phi_1v\dv=\int_M(2+2/n)\varphi^{1+2/n}v\dv\nonumber\\
  &&\qquad+\int_M(2+2/n)\epsilon^nA_a\omega_a^{1+2/n}\Delta_{\theta_a}\chi_\delta(\rho_a)v\dv\nonumber\\
  &&\qquad+O\Bigg(\int_M\bigg[\bigg(\frac{\epsilon^2}{s^2+((\epsilon/n)^2+|z|^2)^2}\bigg)^{n/2}\rho_a^21_{\{\rho_a\leqslant2\delta\}}+\delta^{-1}\epsilon^n1_{\{\delta\leqslant\rho_a\leqslant2\delta\}}\nonumber\\
&&\qquad\qquad+\bigg(\frac{\epsilon^2}{s^2+((\epsilon/n)^2+|z|^2)^2}\bigg)^{1+n/2}1_{\{\rho_a\geqslant\delta\}}\bigg]v\dva\Bigg).
 \end{eqnarray}
 Using H\"{o}lder's inequality, we can estimate
 \begin{eqnarray}
  \label{rhole2delta}
  &&\int_M\bigg(\frac{\epsilon^2}{s^2+((\epsilon/n)^2+|z|^2)^2}\bigg)^{n/2}\rho_a^21_{\{\rho_a\leqslant2\delta\}}v\dva\nonumber\\
  &&\qquad=O\bigg(\epsilon^{2n}\int_{\rho_a\leqslant2\delta}\frac{\rho_a^4}{(s^2+((\epsilon/n)^2+|z|^2)^2)^n}\dva+\|v\|^2\bigg)\nonumber\\
  &&\qquad=O\Big(\delta^{6-2n}\epsilon^{2n}+\|v\|^2\Big),
 \end{eqnarray}
whenever $n\leq 3$.
Similarly, we can obtain
\begin{eqnarray}
  \label{deltalerhole2delta}
  &&\int_M\delta^{-1}\epsilon^n1_{\{\delta\leqslant\rho_a\leqslant2\delta\}}v\dva\nonumber\\
  &&\qquad=O\bigg(\delta^{-2}\epsilon^{2n}\int_{\delta\leqslant\rho_a\leqslant2\delta}\dva+\|v\|^2\bigg)\nonumber\\
  &&\qquad=O\Big(\delta^{2n}\epsilon^{2n}+\|v\|^2\Big)
 \end{eqnarray}
 and
  \begin{eqnarray}
  \label{rhogedelta}
  &&\int_M\bigg(\frac{\epsilon^2}{s^2+((\epsilon/n)^2+|z|^2)^2}\bigg)^{1+n/2}1_{\{\rho_a\geqslant\delta\}}v\dva\nonumber\\
  &&\qquad=O\bigg(\epsilon^{2n+4}\int_{\rho_a\geqslant\delta}\frac{1}{(s^2+((\epsilon/n)^2+|z|^2)^2)^{n+2}}\dva+\|v\|^2\bigg)\nonumber\\
  &&\qquad=O\Big(\epsilon^{2n+4}+\|v\|^2\Big).
 \end{eqnarray}
 Now, by integrating by parts and H\"{o}lder's inequality, we have
 \begin{eqnarray}
  \label{Deltachi}
  &&\int_M\epsilon^nA_a\omega_a^{1+2/n}\Delta_{\theta_a}\chi_\delta(\rho_a)v\dv\nonumber\\
  &&\qquad=\int_M\epsilon^nA_a\omega_a^{-1}\Delta_{\theta_a}\chi_\delta(\rho_a)v\dva\nonumber\\
  &&\qquad=\int_M\epsilon^nA_a\Big[\langle\nabla\omega_a^{-1},\nabla \chi_\delta(\rho_a)\rangle_{\theta_a}v+\omega_a^{-1}\langle\nabla v,\nabla \chi_\delta(\rho_a)\rangle_{\theta_a}\Big]\dva\nonumber\\
  &&\qquad=O\bigg(\int_{\delta\leqslant\rho_a\leqslant2\delta}\epsilon^n|\nabla \chi_\delta(\rho_a)|\Big(|v|+|\nabla v|_{\theta_a}\Big)\dva\bigg)\nonumber\\
  &&\qquad=O\bigg(\int_{\delta\leqslant\rho_a\leqslant2\delta}\epsilon^{2n}\dva+\|v\|^2\bigg)\nonumber\\
  &&\qquad=O\Big(\delta^{2n+2}\epsilon^{2n}+\|v\|^2\Big).
 \end{eqnarray}
 Substituting \eqref{rhole2delta}-\eqref{Deltachi} into \eqref{Lphi1v} gives
 $$\int_M\SL_{\theta_0}\phi_1v\dv=\int_M(2+2/n)\varphi^{1+2/n}v\dv+O(\delta^{6-2n}\epsilon^{2n})+O(\|v\|^2).$$
 {\bf Case $k=2$}. Recalling that $d_2=\epsilon\partial_\epsilon$ we have
 \begin{eqnarray}\label{Lphi2v}
  &&\int_M\SL_{\theta_0}\phi_2v\dv=\int_M2(1+1/n)(1+2/n)\varphi^{2/n}\phi_2v\dv\nonumber\\
  &&\qquad+\int_M2(n+1)\epsilon^nA_a\omega_a^{1+2/n}\Delta_{\theta_a}\chi_\delta(\rho_a)v\dv\nonumber\\
  &&\qquad+\int_Mn\omega_a^{1+2/n}\epsilon^n\Big(\sum_{i=1}^8I_i\Big)v\dv+E
 \end{eqnarray}
 where
 $$E=\int_M\omega_a^{1+2/n}\epsilon^n\Big(\sum_{i=1}^7{\rm d}_2I_i\Big)v\dv,$$
thanks to the facts that $I_8=0$, if $\mu(M)>0$ and $d_2I_8=0$, if $\mu(M)=0$. In view of the proof of {\bf Case $k=1$}, we are left to estimate the term $E$. From Lemma \ref{ConformalOperatorOnBubble}, it follows that
 \begin{eqnarray*}
  E&=&\int_M\omega_a^{-1}\epsilon^n{\rm d}_2\Bigg\{-(2+2/n)\chi_\delta(\rho_a)\bigg(\Delta_{\theta_a}\Big(\frac{1}{(s^2+((\epsilon/n)^2+|z|^2)^2)^{n/2}}\Big)\\
  &&\qquad+\frac{\epsilon^2}{(s^2+((\epsilon/n)^2+|z|^2)^2)^{n/2+1}}\bigg)+\frac{\chi_\delta(\rho_a)R_{\theta_a}-(2+2/n)\Delta_{\theta_a}\chi_\delta(\rho_a)}{(s^2+((\epsilon/n)^2+|z|^2)^2)^{n/2}}\\
  &&\qquad-2(2+2/n)\Big\langle\nabla\chi_\delta(\rho_a), \nabla\Big(\frac{1}{(s^2+((\epsilon/n)^2+|z|^2)^2)^{n/2}}\Big)\Big\rangle_{\theta_a}\\
  &&\qquad-(2+2/n) n \epsilon^2\bigg[\bigg(\frac{\chi_\delta(\rho_a)}{(s^2+((\epsilon/n)^2+|z|^2)^2)^{n/2}}+(1-\chi_\delta(\rho_a))G_a\bigg)^{1+2/n}\\
  &&\qquad\qquad\qquad\qquad-\frac{\chi_\delta(\rho_a)}{(s^2+((\epsilon/n)^2+|z|^2)^2)^{n/2+1}}\bigg]\Bigg\}v\dva\\
  &=&\int_M\omega_a^{-1}\epsilon^{n+2}\Bigg\{\frac{4(n+1)}{n^2}\chi_\delta(\rho_a)\bigg[\Delta_{\theta_a}\Big(\frac{(\epsilon/n)^2+|z|^2}{(s^2+((\epsilon/n)^2+|z|^2)^2)^{n/2+1}}\Big)\\
  &&\qquad-\frac{n\big(s^2+((\epsilon/n)^2+|z|^2)^2\big)-(1+2/n)\epsilon^2((\epsilon/n)^2+|z|^2)}{(s^2+((\epsilon/n)^2+|z|^2)^2)^{n/2+2}}\bigg]\\
  &&\qquad-\frac{\Big(\chi_\delta(\rho_a)R_{\theta_a}-(2+2/n)\Delta_{\theta_a}\chi_\delta(\rho_a)\Big)((\epsilon/n)^2+|z|^2)}{(s^2+((\epsilon/n)^2+|z|^2)^2)^{n/2+1}}\Bigg\}v\\
  &&\qquad-\frac{8(n+1)}{n^2}\epsilon^{n+2}\Big(\Big\langle\nabla\chi_\delta(\rho_a), \nabla(\omega_a^{-1}v)\Big\rangle_{\theta_a}+\Delta_{\theta_a}(\chi_\delta(\rho_a))\omega_a^{-1}v\Big)\\
  &&\qquad\qquad\qquad\qquad\times\Big(\frac{(\epsilon/n)^2+|z|^2}{(s^2+((\epsilon/n)^2+|z|^2)^2)^{n/2+1}}\Big)\\
  &&\qquad-(2+2/n)\omega_a^{-1}\epsilon^n{\rm d}_2\bigg\{\epsilon^2\bigg[\bigg(\frac{\chi_\delta(\rho_a)}{(s^2+((\epsilon/n)^2+|z|^2)^2)^{n/2}}+(1-\chi_\delta(\rho_a))G_a\bigg)^{1+2/n}\\
  &&\qquad\qquad\qquad\qquad-\frac{\chi_\delta(\rho_a)}{(s^2+((\epsilon/n)^2+|z|^2)^2)^{n/2+1}}\bigg]\bigg\}v\dva\\
  &=&\int_{\{\rho_a\leqslant\delta\}}\frac{4(n+1)}{n^2}\omega_a^{-1}\epsilon^{n+2}\bigg[\Delta_{\theta_a}\Big(\frac{(\epsilon/n)^2+|z|^2}{(s^2+((\epsilon/n)^2+|z|^2)^2)^{n/2+1}}\Big)\\
  &&\qquad\qquad-\frac{n\big(s^2+((\epsilon/n)^2+|z|^2)^2\big)-(1+2/n)\epsilon^2((\epsilon/n)^2+|z|^2)}{(s^2+((\epsilon/n)^2+|z|^2)^2)^{n/2+2}}\bigg]v\dva\\
  &&+\int_{\{\rho_a\leqslant\delta\}}-\frac{\omega_a^{-1}\epsilon^{n+2}R_{\theta_a}((\epsilon/n)^2+|z|^2)}{(s^2+((\epsilon/n)^2+|z|^2)^2)^{n/2+1}}v\dva+O(\epsilon^{2n+4}+\|v\|^2)\\
&:=&E_1+E_2+O(\epsilon^{2n+4}+\|v\|^2).
 \end{eqnarray*}
 Observing that
 \begin{eqnarray*}
&&\Delta_{\theta_a}\Big(\frac{(\epsilon/n)^2+|z|^2}{(s^2+((\epsilon/n)^2+|z|^2)^2)^{n/2+1}}\Big)=\Big(\Delta_{\mathbb{H}^n}+O(\rho_a^2)\Big)\Big(\frac{(\epsilon/n)^2+|z|^2}{(s^2+((\epsilon/n)^2+|z|^2)^2)^{n/2+1}}\Big)\\
&&\qquad=\frac{n\big(s^2+((\epsilon/n)^2+|z|^2)^2\big)-(1+2/n)\epsilon^2((\epsilon/n)^2+|z|^2)}{(s^2+((\epsilon/n)^2+|z|^2)^2)^{n/2+2}}+\frac{O(\rho_a^2)((\epsilon/n)^2+|z|^2)}{(s^2+((\epsilon/n)^2+|z|^2)^2)^{n/2+1}},
  \end{eqnarray*}
  we have the estimate
  \begin{eqnarray*}
E_1&=&O\bigg(\epsilon^{2n+4}\int_{\{\rho_a\leqslant\delta\}}\frac{\rho_a^4((\epsilon/n)^2+|z|^2)^2}{(s^2+((\epsilon/n)^2+|z|^2)^2)^{n+2}}\dva+\|v\|^2\bigg)\\
&=&O(\epsilon^{6}+\|v\|^2).
  \end{eqnarray*}
  Using the fact that $R_{\theta_a}=O(\rho_a^2)$ we obtain
  \begin{eqnarray*}
   E_2&=&O\bigg(\epsilon^{2n+4}\int_{\{\rho_a\leqslant\delta\}}\frac{\rho_a^4((\epsilon/n)^2+|z|^2)^2}{(s^2+((\epsilon/n)^2+|z|^2)^2)^{n+2}}\dva+\|v\|^2\bigg)\\
&=&O(\epsilon^{6}+\|v\|^2).
  \end{eqnarray*}
Substituting all the estimates above into \eqref{Lphi2v} gives
\begin{eqnarray*}
\int_M\SL_{\theta_0}\phi_2v\dv&=&\int_M2(1+1/n)(1+2/n)\varphi^{2/n}\phi_2v\dv\\
&&\qquad\qquad+O(\delta^{6-2n}\epsilon^{2n})+O(\epsilon^6)+O(\|v\|^2).
\end{eqnarray*}
{\bf Case $k=3$}. Recalling that $d_3=\epsilon\nabla_a$ where $a=(z_0,s_0)$, we have
 \begin{eqnarray}\label{Lphi3v}
  &&\int_M\SL_{\theta_0}\phi_3v\dv=\int_M2(1+1/n)(1+2/n)\varphi^{2/n}\phi_3v\dv\nonumber\\
  &&\qquad+\int_M(2+2/n)\epsilon^{n+1}\nabla_a\Big(A_a\omega_a^{1+2/n}\Delta_{\theta_a}\chi_\delta(\rho_a)\Big) v\dv\nonumber\\
  &&\qquad+\int_M(1+2/n)\omega_a^{2/n}\nabla_a(\omega_a)\epsilon^{n+1}\Big(\sum_{i=1}^8I_i\Big)v\dv+\overline{E}\nonumber\\
 \end{eqnarray}
 where
 $$\overline{E}=\int_M\omega_a^{1+2/n}\epsilon^n\Big(\sum_{i=1}^8{\rm d}_3I_i\Big)v\dv.$$
 Following the proof of {\bf Case $k=1$} we can get
 \begin{eqnarray*}
  \int_M(2+2/n)\epsilon^{n+1}\nabla_a\Big(A_a\omega_a^{1+2/n}\Delta_{\theta_a}\chi_\delta(\rho_a)\Big) v\dv=O(\epsilon^{2n+2}+\|v\|^2)
 \end{eqnarray*}
and
 \begin{eqnarray*}
 \int_M(1+2/n)\omega_a^{2/n}\nabla_a(\omega_a)\epsilon^{n+1}\Big(\sum_{i=1}^8I_i\Big)v\dv=O(\epsilon^{2n+2}+\|v\|^2).
 \end{eqnarray*}
 Now, if  $\mu(M)>0$, it follows from the fact $I_8=0$ that
\begin{eqnarray*}
 \overline{E}&=&\int_{\{\rho_a\leqslant\delta\}}\omega_a^{-1}\epsilon^n\Big(\sum_{i=1}^7{\rm d}_3I_i\Big)v\dv+O(\epsilon^{2n+2}+\|v\|^2)\\
&=&\int_{\{\rho_a\leqslant\delta\}}\omega_a^{-1}\epsilon^n{\rm d}_3\Bigg\{-(2+2/n)\bigg[\Delta_{\theta_a}\Big(\dfrac{1}{((s-s_0)^2+((\epsilon/n)^2+|z-z_0|^2)^2)^{n/2}}\Big)\\
&&\qquad+\dfrac{\epsilon^2}{((s-s_0)^2+((\epsilon/n)^2+|z-z_0|^2)^2)^{n/2+1}}\bigg]+\frac{R_{\theta_a}}{((s-s_0)^2+((\epsilon/n)^2+|z-z_0|^2)^2)^{n/2}}\Bigg\}\\
&&\qquad v\dva\Bigg|_{a=0}+O(\epsilon^{2n+2}+\|v\|^2).\\
&=&\int_{\{\rho_a\leqslant\delta\}}\omega_a^{-1}\epsilon^{n+1}\bigg\{-(2n+2)\bigg[\Delta_{\theta_a}\Big(\dfrac{\bar{z}((\epsilon/n)^2+|z|^2)}{(s^2+((\epsilon/n)^2+|z|^2)^2)^{n/2+1}}\Big)\\
&&\qquad+\frac{(1+2/n)\epsilon^2((\epsilon/n)^2+|z|^2)\bar{z}}{(s^2+((\epsilon/n)^2+|z|^2)^2)^{n/2+2}}\bigg]+\frac{\nabla_aR_{\theta_a}(s^2+((\epsilon/n)^2+|z|^2)^2)-nR_{\theta_a}((\epsilon/n)^2+|z|^2)\bar{z}}{(s^2+((\epsilon/n)^2+|z|^2)^2)^{n/2+1}}\Bigg\}\\
&&\qquad v\dva+O(\epsilon^{2n+2}+\|v\|^2)\\
&=& O\bigg(\epsilon^{2n+2}\int_{\{\rho_a\leqslant\delta\}}\frac{\rho_a^4((\epsilon/n)^2+|z|^2)^2|z|^2}{(s^2+((\epsilon/n)^2+|z|^2)^2)^{n+2}}\dva\bigg)\\
&&\qquad+O\bigg(\epsilon^{2n+2}\int_{\{\rho_a\leqslant\delta\}}\frac{\rho_a^2}{(s^2+((\epsilon/n)^2+|z|^2)^2)^{n}}\dva\bigg)+O(\epsilon^{2n+2}+\|v\|^2)\\
&=&O(\epsilon^6+\epsilon^{2n+2}+\|v\|^2).
\end{eqnarray*}
While $\mu(M)=0$, we notice that
\begin{eqnarray*}
d_3I_8&=&-\epsilon\Bigg[\bigg(\nabla_a(\chi_\delta(\rho_a))+\frac{n(1-\chi_\delta(\rho_a))}{2}\bigg)\Big(\Delta_{\theta_{\mathbb{H}^n}}+O(\rho^2_a)\Big)\Big(\rho^{-2n}_a+\rho^{-2n-4}_a\bar{z}|z|^2\Big)\Bigg]\Bigg|_{a=0}\\
&=&O\Big(\delta^{1-2n}\epsilon\Big),
\end{eqnarray*}
and hence $\bar{E}$  can be estimated as before.
Substituting all the  estimates above into \eqref{Lphi3v} yields
\begin{eqnarray*}
\int_M\SL_{\theta_0}\phi_3v\dv&=&\int_M2(1+1/n)(1+2/n)\varphi^{2/n}\phi_3v\dv\\
&&\qquad\qquad+O(\epsilon^{2n+2})+O(\epsilon^6)+O(\|v\|^2).
\end{eqnarray*}
Finally, we are going to estimate the first term on the right hand side of \eqref{Lphi1v}, \eqref{Lphi2v} and \eqref{Lphi3v}. To this end,
we use Proposition \eqref{OptimalChoice} and expand $u^{2/n}=(\alpha\varphi+v)^{2/n}$     to obtain
\begin{eqnarray*}
 0=\int_Mu^{2/n}\phi_kv\dv=\int_M(\alpha\varphi)^{2/n}\phi_kv\dv+O(\|v\|^2).
\end{eqnarray*}
This implies that
\begin{equation}\label{varphi2nphikv}
 \int_M\varphi^{2/n}\phi_kv\dv=O(\|v\|^2).
\end{equation}
By inserting \eqref{varphi2nphikv} into \eqref{Lphi1v}, \eqref{Lphi2v} and \eqref{Lphi3v}, we thus complete the proof of (i).
\bigskip

\noindent{(ii)} By expansion as above, we have
\begin{eqnarray}\label{fuphik}
\int_Mfu^{1+2/n}\phi_k\dv&=&\int_Mf(\alpha\varphi)^{1+2/n}\phi_k\dv\nonumber\\
&&+(1+2/n)\int_Mf(\alpha\varphi)^{2/n}\phi_kv\dv+O(\|v\|^2).\nonumber\\
\end{eqnarray}
It remains to estimate the second term on the right hand side of \eqref{fuphik}.
From \eqref{varphi2nphikv} and Lemma \ref{EstimatesofBubbles}, it follows that
\begin{eqnarray*}
 &&\int_Mf\varphi^{2/n}\phi_kv\dv\\
 &&\qquad\qquad=\int_M(f-f(a))\varphi^{2/n}\phi_kv\dv+\int_Mf(a)\varphi^{2/n}\phi_kv\dv\\
 &&\qquad\qquad=\int_{B_\delta(a)}(f-f(a))\varphi^{2/n}\phi_kv\dv+O(\epsilon^{2n+4}+\|v\|^2)\\
 &&\qquad\qquad=O\bigg(\int_{B_\delta(a)}|f-f(a)|\varphi^{1+2/n}v\dv\bigg)+O(\epsilon^{2n+4}+\|v\|^2).
\end{eqnarray*}
By H\"{o}lder's inequality, we can estimate
\begin{eqnarray*}
 &&\int_{B_\delta(a)}|f-f(a)|\varphi^{1+2/n}v\dv\\
 &&\qquad\leq \bigg[\int_{B_\delta(a)}|f-f(a)|^\frac{2(n+1)}{n+2}\varphi^{2+2/n}\dv\bigg]^\frac{n+2}{2(n+1)}\bigg[\int_M|v|^{2+2/n}\dv\bigg]^\frac{n}{2(n+1)}\\
 &&\qquad=O\Bigg(\bigg[\int_{B_\delta(a)}|f-f(a)|^\frac{2(n+1)}{n+2}\frac{\epsilon^{2n+2}}{(s^2+((\epsilon/n)^2+|z|^2)^2)^{n+1}}\dva\bigg]^\frac{n+2}{(n+1)}+\|v\|^2\Bigg)\\
 &&\qquad=O\Bigg(\bigg[\int_{B_\frac\delta\epsilon(a)}\frac{|f(\epsilon z,\epsilon^2s)-f(a)|^\frac{2(n+1)}{n+2}}{(s^2+((1/n)^2+|z|^2)^2)^{n+1}}\dva\bigg]^\frac{n+2}{(n+1)}+\|v\|^2\Bigg)\\
 &&\qquad=O\Bigg(\bigg[\int_{B_\frac\delta\epsilon(a)}\frac{\Big(
\epsilon|\nabla f(a)|_{\theta_0}|( z,\epsilon s)|\Big)^\frac{2(n+1)}{n+2}+O\Big(\big(\epsilon|( z,\epsilon s)|\big)^\frac{4(n+1)}{n+2}\Big)}{(s^2+((1/n)^2+|z|^2)^2)^{n+1}}\dva\bigg]^\frac{n+2}{(n+1)}+\|v\|^2\Bigg)\\
&&\qquad=O\Big(|\nabla f(a)|_{\theta_0}^2\epsilon^2+\epsilon^4+\|v\|^2\Big)=O\Big(|\nabla f(a)|_{\theta_0}^2\epsilon^2+\|v\|^2\Big).
\end{eqnarray*}
Plugging the two estimates above into \eqref{fuphik} gives
\begin{eqnarray*}
\int_Mfu^{1+2/n}\phi_k\dv&=&\int_Mf(\alpha\varphi)^{1+2/n}\phi_k\dv +O\Big(|\nabla f(a)|_{\theta_0}^2\epsilon^2+\|v\|^2\Big).
\end{eqnarray*}
This completes the proof of (ii).
 \end{proof}

To proceed further, we consider the so-called shadow flow, which will build the relation between the parameters such as $a, \epsilon, \alpha$ and the prescribed function $f$.

\begin{lemma}
\label{shadowflow}
Set
$$\sigma_k=-\int_M\Big(\SL_{\theta_0}u-\lambda fu^{1+2/n}\Big)\phi_k\dv,~~k=1, 2, 3.$$
For $t$ sufficiently large, there hold
\begin{itemize}
  \item [(i)] $\dot{\alpha}/\alpha=c_1^{-1}\alpha^{-(1+2/n)}\sigma_1\Big(1+O(\epsilon)\Big)+R_1,$
\item [(ii)] $\dot{\epsilon}/\epsilon=c_2^{-1}\alpha^{-(1+2/n)}\sigma_2\Big(1+O(\epsilon)\Big)+R_2,$
\item [(iii)] $\dot{a}/\epsilon=c_3^{-1}\alpha^{-(1+2/n)}\sigma_3\Big(1+O(\epsilon)\Big)+R_3,$
\end{itemize}
where $R_k=O\Big(\|v\|^2+F_2(\theta(t))\Big)_k$
and $c_k$ are constants appeared in Lemma
\ref{EstimatesofBubbles} (ii).
\end{lemma}
\begin{proof}
 Let $({\dot \zeta}^1, {\dot \zeta}^2, {\dot \zeta}^3)=(\dot \alpha, \alpha \dot \epsilon \epsilon^{-1}, \alpha \dot a \epsilon^{-1}).$ Using the flow $u_t=-(R_{\theta(t)}-\lambda f)u$ and the fact that $\int_Mu^{2/n}\phi_kv\dv=0$ and recalling $(\phi_1,\phi_2,\phi_3)=(\varphi, \epsilon\partial_\epsilon\varphi, \epsilon\nabla_a\varphi)$, we obtain
 \begin{eqnarray}
  \label{sigmak}
  \sigma_k&=&\int_M u_t u^{2/n}\phi_k\dv=\int_M \partial_t(\alpha\varphi+v) u^{2/n}\phi_k\dv\nonumber\\
  &=&{\dot \zeta}^l\int_Mu^{2/n}\phi_k\phi_l\dv-\int_M(\partial_tu^{2/n})\phi_kv\dv\nonumber\\
  &&-\int_Mu^{2/n}(\partial_t\phi_k)v\dv\nonumber\\
  &:=&{\dot \zeta}^lI-II-III.
 \end{eqnarray}
{\bf Estimate of $I$.}\quad It follows from expansion and Lemma \ref{EstimatesofBubbles} (i) and (iv) that
\begin{eqnarray*}
 I&=&\int_M\big(\alpha\varphi\big)^{2/n}\phi_k\phi_l\dv+O(\|v\|)_{k,l}\\
&=&\alpha^{2/n}\delta_{lk}\int_M\varphi^{2/n}\phi_k^2\dv+(1-\delta_{lk})\alpha^{2/n}\int_M\varphi^{2/n}\phi_k\phi_l\dv+O\Big(\|v\|\Big)_{k,l}\\
&=&\alpha^{2/n}c_k\delta_{lk}+O\Big(\epsilon+\|v\|\Big)_{k,l}.
\end{eqnarray*}
{\bf Estimate of $II$.}\quad By \eqref{VolBd}, H\"{o}lder's inequality and Lemma \ref{EstimatesofBubbles} (i), we have
\begin{eqnarray*}
 |II|&=&\bigg|\frac2n\int_M(R_{\theta(t)}-\lambda f)u^{2/n}\phi_kv\dv\bigg|\\
 &\leqslant& C\int_M|R_{\theta(t)}-\lambda f|u^{2/n}\varphi |v|\dv\\
 &\leqslant&C\int_M|R_{\theta(t)}-\lambda f|u^{2/n}|u-v||v|\dv\\
 &\leqslant&C\int_M|R_{\theta(t)}-\lambda f|u^{3}|v|\dv+C\int_M|R-\lambda f|u^{2}|v|^2\dv\\
  &\leqslant&C\|v\|\bigg(\int_M(\R-\lambda f)^2u^{4}\dv\bigg)^{1/2}\left(\int_Mu^{4}\dv\right)^{1/4}\\
  &&
  +C\|v\|^2\bigg(\int_M(\R-\lambda f)^2u^{4}\dv\bigg)^{1/2}\\
 &\leqslant& C\Big(F_2(\theta(t))+\|v\|^2\Big),
\end{eqnarray*}
since $n=1$.\\
{\bf Estimate of $III$.}\quad Using Lemma \ref{EstimatesofBubbles} (i) again,
we get
\begin{eqnarray*}
III&=&{\dot \zeta}^l\int_Mu^{2/n}{\rm d}_l\phi_kv\dv
=O\bigg(\int_Mu^{1+2/n}v\dv\bigg)_{k,l}{\dot \zeta}^l
=O\Big(\|v\|\Big)_{k,l}{\dot \zeta}^l.
\end{eqnarray*}

Now, substituting the estimates for $I, II, III$ into \eqref{sigmak} gives
$$\sigma_k=\Big[\alpha^{2/n}c_k\delta_{kl}+O\Big(\epsilon+\|v\|\Big)_{k,l}\Big]{\dot \zeta}^l+O\Big(F_2(\theta(t))+\|v\|^2\Big)_k.$$
Setting
$E_{k,l}=\alpha^{2/n}c_k\delta_{kl}+O\Big(\epsilon+\|v\|\Big)_{k,l},$ one then has
$$E_{k,l}{\dot \zeta}^l=\sigma_k+O\Big(F_2(\theta(t))+\|v\|^2\Big)_k$$
Using the fact that
$$E_{k,l}^{-1}=\alpha^{-2/n}c_k^{-1}\delta_{kl}+O\Big(\epsilon+\|v\|\Big)_{k,l}$$
and using the definition of $\sigma_k$ that
$$\sigma_k=O(F_2(\theta(t))^{1/2}),$$
we can obtain
\begin{eqnarray*}
{\dot \zeta}^l&=&\sigma_k\Big[\alpha^{-2/n}c_k^{-1}\delta_{kl}+O\Big(\epsilon+\|v\|\Big)_{k,l}\Big]+O\Big(F_2(\theta(t))+\|v\|^2\Big)_l\\
&=&\alpha^{-2/n}c_l^{-1}\sigma_l\Big(1+O(\epsilon)\Big)+O\Big(F_2(\theta(t))+\|v\|^2\Big)_l.
\end{eqnarray*}
In view of the definition of ${\dot \zeta}^l$, the proof is thus complete.
\end{proof}

To relate the parameters $\alpha, \epsilon$ and $a$ with the prescribed function $f$, it remains to refine the estimate of the shadow flows $\sigma_k$.
\begin{lemma}
 \label{analysingsigmak}
There exist positive constants $d_1, d_2, \dots, e_4$ such that, for $t$ large, there hold
 \begin{itemize}
  \item [(i)] $\sigma_1=\displaystyle(2+2/n)\alpha\bigg(\frac{n\lambda\alpha^{2/n}f(a)}{2(n+1)}-1\bigg)\int_M\varphi^{2+2/n}\dv+\alpha d_1A_a\epsilon^{2n}$

  \hspace{1em} $
   +\alpha^{1+2/n}\lambda e_1\epsilon^2\Delta_{\theta_0}f(a)+R_1$;
   \item [(ii)] $\sigma_2=\displaystyle(2+2/n)\alpha\bigg(\frac{n\lambda\alpha^{2/n}f(a)}{2(n+1)}-1\bigg)\int_M\varphi^{1+2/n}\phi_2\dv+\alpha d_2A_a\epsilon^{2n}$

  \hspace{1em} $
   +\alpha^{1+2/n}\lambda e_2\epsilon^2\Delta_{\theta_0}f(a)+R_2$;
   \item[(iii)] $\sigma_3=\displaystyle(2+2/n)\alpha\bigg(\frac{n\lambda\alpha^{2/n}f(a)}{2(n+1)}-1\bigg)\int_M\varphi^{1+2/n}\phi_3\dv $

  \hspace{1em} $
   +\alpha^{1+2/n}\lambda \Big(e_3\epsilon\nabla_{\theta_0}f(a)+e_4\epsilon^3\nabla_{\theta_0}\Delta_{\theta_0}f(a)\Big)+R_3$,
 \end{itemize}
where
$$R_k=O\Big(\delta^{4-2n}\epsilon^{2n}+\delta\epsilon^{2n}+\epsilon^4+|\nabla f(a)|^2_{\theta_0}\epsilon^2+\|v\|^2\Big)_k~~\mbox{ for }k=1, 2, 3.$$
\end{lemma}
\begin{proof}
 Recall that
 $$\sigma_k=-\int_M\Big(\SL_{\theta_0}u-\lambda fu^{1+2/n}\Big)\phi_k\dv.$$
 {\underline{\bf Estimate of $\int_M\SL_{\theta_0}u\phi_k\dv$}.}\quad By integrating by parts, we may write
 \begin{eqnarray}
  \label{SLuphik}
  &&\int_M\SL_{\theta_0}u\phi_k\dv=\int_M \SL_{\theta_0}\Big(\alpha\varphi+v\Big)\phi_k\dv\nonumber\\
  &&\quad\qquad= \alpha\int_M(\SL_{\theta_0}\varphi)\phi_k\dv+ \int_M(\SL_{\theta_0}\phi_k)v\dv.
 \end{eqnarray}
The first term on the right hand side of \eqref{SLuphik} can be estimated as follows. Using Lemmas \ref{ConformalOperatorOnBubble} and \ref{EstimatesofBubbles} and following the proof of \eqref{rhole2delta}-\eqref{rhogedelta}, we have
\begin{eqnarray*}
 && \alpha\int_M(\SL_{\theta_0}\varphi)\phi_k\dv\\
 &&\qquad=(2+2/n) \alpha\int_M\varphi^{1+2/n}\phi_k\dv\\&&\qquad\qquad+(2+2/n)\alpha\epsilon^n\int_M\omega_a^{1+2/n}A_a\Delta_{\theta_a}\chi_\delta(\rho_a)\phi_k\dv\\
 &&\qquad\qquad+O\Bigg(\int_M\bigg[\bigg(\frac{\epsilon^2}{s^2+((\epsilon/n)^2+|z|^2)^2}\bigg)^{n/2}\rho_a^21_{\{\rho_a\leqslant2\delta\}}+\delta^{-1}\epsilon^n1_{\{\delta\leqslant\rho_a\leqslant2\delta\}}\nonumber\\
&&\qquad\qquad\qquad\qquad+\bigg(\frac{\epsilon^2}{s^2+((\epsilon/n)^2+|z|^2)^2}\bigg)^{1+n/2}1_{\{\rho_a\geqslant\delta\}}\bigg]\phi_k\dva\Bigg)\\
&&\qquad:=(2+2/n) \alpha\int_M\varphi^{1+2/n}\phi_k\dv+I+O\Big(\delta^{4-2n}\epsilon^{2n}+\delta\epsilon^{2n}+\epsilon^{2n+2}\Big).
\end{eqnarray*}
It remains to estimate the term $I$ above. As always, we split the argument into three cases.

\noindent{\bf Case $k=1$.}\quad By denoting that $\psi=\psi_{a,\epsilon}$
which was defined in \eqref{psi},
we may write $I$ as
\begin{eqnarray*}
 I&=&(2+2/n)\alpha\epsilon^nA_a\int_M\Delta_{\theta_a}\chi_\delta(\rho_a)\psi\dva\\
 &=&(2+2/n)\alpha\epsilon^nA_a\bigg[\int_M\Delta_{\theta_a}\chi_\delta(\rho_a)\epsilon^n\rho_a^{-2n}\dva\\
 &&\qquad+\int_M\Delta_{\theta_a}\chi_\delta(\rho_a)\Big(\psi-\epsilon^n\rho_a^{-2n}\Big)\dva\\
 &:=&(2+2/n)\alpha\epsilon^nA_a\Big(I_1+I_2\Big).
\end{eqnarray*}
To estimate $I_1$, we use integration by parts to get
\begin{eqnarray*}
 I_1&=&-\epsilon^n\int_M\langle\nabla\rho_a^{-2n},\nabla\chi_\delta(\rho_a)\rangle_{\theta_a}\dva\\
 &=&2n\int_{\{\delta\leqslant\rho_a\leqslant2\delta\}}\epsilon^n\partial_{\rho_a}\Big(\chi_\delta(\rho_a)\Big)\rho_a^{-2n-1}|\nabla\rho_a|^2_{\theta_a}\dva\\
 &=&2n\epsilon^n|B^n|\int_\delta^{2\delta}\partial_{\rho}\Big(\chi_\delta(\rho)\Big){\rm d}\rho=-2n|\mathbb{S}^{2n+1}|\epsilon^n,
\end{eqnarray*}
where $|B^n|$ is the volume of the unit ball in $\mathbb{H}^n$.
For $I_2$, in view of the definition of $\psi$, we can estimate
\begin{eqnarray*}
 |I_2|&\leqslant&\frac{C\epsilon^n}{\delta^2}\Bigg[\int_{\{\delta\leqslant\rho_a\leqslant2\delta\}}\bigg|\frac{1}{(s^2+((\epsilon/n)^2+|z|^2)^2)^{n/2}}-\frac{1}{(s^2+|z|^4)^{n/2}}\bigg|\dva\\
 &&\qquad\qquad\qquad\qquad+\int_{\{\delta\leqslant\rho_a\leqslant2\delta\}}\Big|G_a-\rho_a^{-2n}\Big|\dva\Bigg]\\
 &\leqslant&\frac{C\epsilon^n}{\delta^2}\Bigg[\int_{\{\delta\leqslant\rho_a\leqslant2\delta\}}\bigg|\frac{\epsilon^2}{(s^2+|z|^4)^{n/2+1}}\bigg|\dva\\
 &&\qquad\qquad\qquad\qquad+\int_{\{\delta\leqslant\rho_a\leqslant2\delta\}}A_a+O(\rho_a)\dva\Bigg]\\
 &=&O\Big(\epsilon^{n+2}+\delta^{2n}\epsilon^n\Big).
\end{eqnarray*}
Combining all the estimates yields
$$I=-4(n+1)\alpha|B^n|\,A_a\epsilon^{2n}+O\Big(\delta^{2n}\epsilon^{2n}+\epsilon^{2n+2}\Big).$$

\noindent{\bf Case $k=2$.}\quad  From \eqref{phi2} and the conclusion in {\bf Case $k=1$}, it follows that
\begin{eqnarray*}
 I&=&(2+2/n)\alpha A_a\epsilon^n\int_M\omega_a^{1+2/n}\Delta_{\theta_a}\chi_\delta(\rho_a)\Bigg[n\varphi-\frac{2n\omega_a\epsilon^{n+2}\chi_{\delta}(\rho_a)((\epsilon/n)^2+|z|^2)}{(s^2+((\epsilon/n)^2+|z|^2)^2)^{n/2+1}}\Bigg]\dv\\
 &=&2(n+1)\alpha A_a\epsilon^n\int_M\Delta_{\theta_a}\chi_\delta(\rho_a)\psi\dva+O\Big(\epsilon^{2n+2}\Big)\\
 &=&-4n(n+1)\alpha|B^n|\,A_a\epsilon^{2n}+O\Big(\delta^{2n}\epsilon^{2n}+\epsilon^{2n+2}\Big).
\end{eqnarray*}

\noindent{\bf Case $k=3$.}\quad From \eqref{phi3}, it is easy to see that
$$I=(2+2/n)\alpha A_a\epsilon^n\int_M\omega_a^{1+2/n}\Delta_{\theta_a}\chi_\delta(\rho_a)\phi_3\dv=O(\epsilon^{2n+1}).$$
 Plugging all the estimates of three cases and the estimate in Lemma \ref{SLphikv} (i) into \eqref{SLuphik} yields
 \begin{eqnarray}
  \label{SLuphik1}
  -\int_M\SL_{\theta_0}u\phi_k\dv&=&-(2+2/n)\alpha\int_M\varphi^{1+2/n}\phi_k\dv+\alpha d_k A_a\epsilon^{2n}\nonumber\\
  &&\qquad+O\Big(\delta^{4-2n}\epsilon^{2n}+\delta\epsilon^{2n}+\epsilon^{2n+1}+\epsilon^6+\|v\|^2\Big),\nonumber\\
 \end{eqnarray}
 where $d_1=4(n+1)|B^n|,~~d_2=4n(n+1)|B^n|$ and $d_3=0$.
 \bigskip

 \noindent\underline{{\bf Estimate of $\int_Mfu^{1+2/n}\phi_k\dv$}.}\quad By Lemma \ref{SLphikv}, we can get
 \begin{equation} \label{fuphik1}
\begin{split}
 &\int_Mfu^{1+2/n}\phi_k\dv\\
 &=\int_Mf(\alpha\varphi)^{1+2/n}\phi_k\dv+O\Big(|\nabla f(a)|_{\theta_0}^2\epsilon^2+\|v\|^2\Big)\\
  &=\alpha^{1+2/n}\int_M(f-f(a))\varphi^{1+2/n}\phi_k\dv\\
  &\qquad+\alpha^{1+2/n}f(a)\int_M\varphi^{1+2/n}\phi_k\dv+O\Big(|\nabla f(a)|_{\theta_0}^2\epsilon^2+\|v\|^2\Big)\\
  &:=\alpha^{1+2/n}I+\alpha^{1+2/n}f(a)\int_M\varphi^{1+2/n}\phi_k\dv+O\Big(|\nabla f(a)|_{\theta_0}^2\epsilon^2+\|v\|^2\Big).
 \end{split}
 \end{equation}
To estimate $I$ in \eqref{fuphik1}, we expand $f(\epsilon z,\epsilon^2s)-f(a)$ in $B_{\delta/\epsilon}(a)$ to obtain
\begin{eqnarray*}
f(\epsilon z,\epsilon^2s)-f(a)&=&\nabla_{\theta_0} f(a)\cdot(\epsilon z,\epsilon^2 s)+\frac12(\nabla{\rm d})_{\theta_0}f(a)\cdot(\epsilon z,\epsilon^2s)^2\\
&&\quad+\frac16(\nabla^2{\rm d})_{\theta_0}f(a)\cdot(\epsilon z,\epsilon^2s)^3+O(|(\epsilon z,\epsilon^2s)|^4).
\end{eqnarray*}

\noindent{\bf Case $k=1$}.\quad Using the expansion above and symmetry, we have
\begin{eqnarray*}
 I&=&\int_M(f-f(a))\varphi^{2+2/n}\dv\\
 &=&\int_{B_\delta(a)}(f-f(a))\frac{\epsilon^{2n+2}}{(s^2+((\epsilon/n)^2+|z|^2)^2)^{n+1}}\dvh+O\Big(\epsilon^{2n+2}\Big)\\
 &=&\int_{B_\frac{\delta}{\epsilon}(a)}\frac{(f(\epsilon x, \epsilon^2 s)-f(a))}{(s^2+((1/n)^2+|z|^2)^2)^{n+1}}\dvh+O\Big(\epsilon^{2n+2}\Big)\\
 &=&e_1\Delta_{\theta_0}f(a)\epsilon^2+O\Big(\epsilon^4+\epsilon^{2n+2}\Big),
\end{eqnarray*}
where $$e_1=\frac{1}{2n}\int_{\mathbb{H}^n}\frac{|z|^2}{(s^2+((1/n)^2+|z|^2)^2)^{n+1}}\dvh.$$
\bigskip

\noindent{\bf Case $k=2$}.\quad By \eqref{phi2}, the expansion above and symmetry, we obtain
\begin{eqnarray*}
 I&=&\int_M(f-f(a))\varphi^{1+2/n}\phi_2\dv\\
 &=&\int_{M}(f-f(a))\varphi^{1+2/n}\bigg(n\varphi-\frac{2\omega_a\epsilon^{n+2}\chi_{\delta}(\rho_a)((\epsilon/n)^2+|z|^2)}{n(s^2+((\epsilon/n)^2+|z|^2)^2)^{n/2+1}}\bigg)\dv\\
 &=&n\int_{M}(f-f(a))\varphi^{2+2/n}\dv\\
 &&-\int_M(f-f(a))\varphi^{1+2/n}\frac{2\omega_a\epsilon^{n+2}\chi_{\delta}(\rho_a)((\epsilon/n)^2+|z|^2)}{n(s^2+((\epsilon/n)^2+|z|^2)^2)^{n/2+1}}\dv\\
 &=&n\int_{B_\delta(a)}(f-f(a))\frac{\epsilon^{2n+2}(s^2+|z|^4-(\epsilon/n)^4)}{(s^2+((\epsilon/n)^2+|z|^2)^2)^{n+2}}\dva+O\Big(\epsilon^{2n+2}\Big)\\
 &=&n\int_{B_\frac\delta\epsilon(a)}\frac{(f(\epsilon z,\epsilon^2s)-f(a))(s^2+|z|^4-(1/n)^4)}{(s^2+((1/n)^2+|z|^2)^2)^{n+2}}\dvh+O\Big(\epsilon^{2n+2}\Big)\\
 &=&e_2\Delta_{\theta_0}f(a)\epsilon^2+O\Big(\epsilon^4+\epsilon^{2n+2}\Big),
\end{eqnarray*}
where $$e_2=\frac{1}{2n}\int_{\mathbb{H}^n}\frac{|z|^2(s^2+|z|^4-(1/n)^4)}{(s^2+((1/n)^2+|z|^2)^2)^{n+2}}\dvh.$$
\bigskip

\noindent{\bf Case $k=3$}.\quad Recalling that $\phi_3={\rm d}_3\varphi=\epsilon\nabla_a\varphi|_{a=0}$ and applying the conclusion in {\bf Case $k=1$}, we have
\begin{eqnarray*}
 I&=&\int_M(f-f(a))\varphi^{1+2/n}\phi_3\dv=\int_M(f-f(a))\varphi^{1+2/n}{\rm d}_3\varphi\dv\\
 &=&\frac{n}{2(n+1)}\bigg(\int_M{\rm d}_3(f(a))\varphi^{2+2/n}\dv+{\rm d}_3\int_M(f-f(a))\varphi^{2+2/n}\dv\bigg)\\
 &=&e_3\nabla_{\theta_0}f(a)\epsilon+e_4\nabla_{\theta_0}\Delta_{\theta_0}f(a)\epsilon^3+O\Big(\epsilon^4+\epsilon^{2n+2}\Big),
\end{eqnarray*}
where $$e_3=\frac{n}{2(n+1)}\int_{\mathbb{H}^n}\frac{1}{(s^2+((1/n)^2+|z|^2)^2)^{n+1}}\dvh$$
and
$$e_4=\frac{1}{4(n+1)}\int_{\mathbb{H}^n}\frac{|z|^2}{(s^2+((1/n)^2+|z|^2)^2)^{n+1}}\dvh.$$
Substituting all the estimates above into \eqref{fuphik1} we obtain
\begin{eqnarray}
 \label{lambdafuphik}
 \lambda\int_Mfu^{1+2/n}\phi_k\dv\nonumber&=&\lambda\alpha^{1+2/n}D_k+\lambda\alpha^{1+2/n}f(a)\int_M\varphi^{1+2/n}\phi_k\dv\nonumber\\
&&+O\Big(|\nabla f(a)|_{\theta_0}^2\epsilon^2+\epsilon^4+\epsilon^{2n+2}+\|v\|^2\Big)
\end{eqnarray}
where
$$
D_k=\left\{
 \begin{array}{ll}
  e_1\Delta_{\theta_0}f(a)\epsilon^2,&\mbox{ for }k=1,\\
   e_2\Delta_{\theta_0}f(a)\epsilon^2,&\mbox{ for }k=2,\\
   e_3\nabla_{\theta_0}f(a)\epsilon+e_4\nabla_{\theta_0}\Delta_{\theta_0}f(a)\epsilon^3,&\mbox{ for }k=3.
 \end{array}
\right.
$$
Now, plugging \eqref{SLuphik1} and \eqref{lambdafuphik} into the expression of $\sigma_k$, we thus complete the proof.
\end{proof}

Since $\sigma_1$ can be controlled by $F_2(\theta(t))$, we may simplify the estimates of $\sigma_k$ as follows.
\begin{corollary}
 \label{simplifysigmak}
 For $t$ large, there hold
 \smallskip

 \begin{itemize}
  \item [(i)] $\sigma_2=d_2\alpha A_a\epsilon^{2n}+e_2\alpha^{1+2/n}\lambda\Delta_{\theta_0}f(a)\epsilon^2+R_2$\\[-0.1em]
  \item[(ii)] $\sigma_3=\alpha^{1+2/n}\lambda\Big(e_3\nabla_{\theta_0}f(a)\epsilon+e_4\nabla_{\theta_0}\Delta_{\theta_0}f(a)\epsilon^3\Big)+R_3,$
 \end{itemize}
where
$$R_k=\Big(\delta^{4-2n}\epsilon^{2n}+\delta\epsilon^{2n}+\epsilon^4+F_2(\theta(t))+\|v\|^2\Big)_k,~~\mbox{ for }k=2, 3.$$
\end{corollary}
\begin{proof}
 Notice by \eqref{VolBd} and H\"{o}lder's inequality that
\begin{equation*}
\begin{split}
|\sigma_1|&=\bigg|\int_M(\R-\lambda f)u^{3}\varphi\dv\bigg|\\
&\leqslant C\left(\int_M(\R-\lambda f)^2u^{4}\dv\right)^{1/2}\left(\int_Mu^{4}\dv\right)^{1/4}\leqslant
CF_2(\theta(t))^{1/2},
\end{split}
\end{equation*}
since $n=1$.
 Hence, it follows from Lemma \ref{analysingsigmak} (i) that
 \begin{equation}
  \label{asymptoticoffa}
  \frac{n\lambda\alpha^{2/n}f(a)}{2(n+1)}=1+O\Big(\epsilon^{2n}+\Big[|\nabla f(a)|^2_{\theta_0}+\Delta_{\theta_0}f(a)\Big]\epsilon^2+F_2(\theta(t))^{1/2}+\|v\|^2\Big).
 \end{equation}
Moreover, by Lemma \ref{EstimatesofBubbles} (iii), we have
$$\int_M\varphi^{1+2/n}\phi_k\dv=O(\epsilon^{2n+2}),~~\mbox{ for }k=2, 3.$$
Plugging this and \eqref{asymptoticoffa} into Lemma \ref{analysingsigmak} (ii) and (iii) yields the conclusion.
\end{proof}
\smallskip

To proceed further, we consider the functional
$$\E_f(u)=\frac{\E(u)}{\Big(\int_Mfu^{2+2/n}\dv\Big)^\frac{n}{n+1}}:=\frac{\E(u)}{f_u^{\frac{n}{n+1}}},$$
where $\E(u)$ is defined in \eqref{EnergyFunctional}. Also, for each $t$, we set
$$H_{u}\big(\mu_0\big)=
\left\{\psi: \int_M\psi\phi_k u^{2/n}\dv=0
~~\mbox{ for  all }k=1,2,3\right\},
$$
where $\mu_0$ is defined in Proposition \ref{OptimalChoice}. Recalling that $u=\alpha\varphi+v$, we then have
\begin{lemma}
 \label{derivativeonHu}
 For $t$ large and for $h_1, h_2\in H_{u}\big(\mu_0\big)$, one has
 \vspace{0.2em}

 \begin{itemize}
  \item [(i)] $\|\partial\E_f(\alpha\varphi)|_{H_u(\mu)}\|=O\Big(\delta^{4-3n}\epsilon^n+|\nabla f(a)|_{\theta_0}\epsilon+|\Delta_{\theta_0}f(a)|\epsilon^2+\epsilon^3+\epsilon^{n+2}$

  \hspace{8em}$+F_2(\theta(t))^{1/2}+\|v\|^2\Big);$
  \vspace{0.2em}

  \item[(ii)]$\displaystyle\frac12\partial^2\E_f(\alpha\varphi)h_1h_2=f_{\alpha\varphi}^{-\frac{n}{n+1}}\Bigg[\int_M(\SL_{\theta_0}h_1)h_2\dv$

 \hspace{6.5em}$\displaystyle
 -\frac{2(n+1)(n+2)}{n^2}\int_M\varphi^{2/n}h_1h_2\dv\Bigg]+O(\mu_0)\|h_1\|\,\|h_2\|$.
 \end{itemize}
\end{lemma}
\begin{proof}
 Given $h\in H_u(\mu_0)$ with $\|h\|=1$, a straightforward calculation then shows that
 \begin{equation}
  \label{1storderderivativeofEf}
  \frac12\partial\E_f(\alpha\varphi)h=f_{\alpha\varphi}^{-\frac{n}{n+1}}\bigg[\int_M\SL_{\theta_0}(\alpha\varphi)h\dv-\frac{\E(\alpha\varphi)}{f_{\alpha\varphi}}\int_Mf(\alpha\varphi)^{1+2/n}h\dv\bigg]
 \end{equation}
and
\begin{eqnarray}
 \label{2ndorderderivativeofEf}
 &&\frac12\partial^2\E_f(\alpha\varphi)h_1h_2\nonumber\\
 &&\quad= f_{\alpha\varphi}^{-\frac{n}{n+1}}\bigg[\int_M(\SL_{\theta_0}h_1)h_2\dv-(1+2/n)\frac{\E(\alpha\varphi)}{f_{\alpha\varphi}}\int_Mf(\alpha\varphi)^{2/n}h_1h_2\dv\bigg]\nonumber\\
 &&\qquad-2f_{\alpha\varphi}^{-\frac{2n+1}{n+1}}\bigg[\int_M\SL_{\theta_0}(\alpha\varphi)h_1\dv\int_Mf(\alpha\varphi)^{1+2/n}h_2\dv\nonumber\\
 &&\qquad\qquad\qquad\qquad+\int_M\SL_{\theta_0}(\alpha\varphi)h_2\dv\int_Mf(\alpha\varphi)^{1+2/n}h_1\dv\bigg]\nonumber\\
 &&\qquad+(4+2/n)\frac{\E(\alpha\varphi)}{f_{\alpha\varphi}^\frac{3n+2}{n+1}}\int_Mf(\alpha\varphi)^{1+2/n}h_1\dv\int_Mf(\alpha\varphi)^{1+2/n}h_2\dv.\nonumber\\
\end{eqnarray}
By \eqref{UnitIntoff}, we may write
\begin{equation}\label{Eu}
\begin{split}
\E(u)&=\int_MR_{\theta(t)}u^{2+2/n}\dv\\
&=\int_M(R_{\theta(t)}-\lambda f)u^{2+2/n}\dv+\lambda
=\lambda+O(F_2(\theta(t))^{1/2}).
\end{split}
\end{equation}
This together with Lemma \ref{SLphikv} (i) implies that
\begin{eqnarray*}
 \E(\alpha\varphi)&=&\int_M\SL_{\theta_0}(\alpha\varphi)(\alpha\varphi)\dv=\int_M\SL_{\theta_0}(u-v)(u-v)\dv\\
 &=&\int_Mu\SL_{\theta_0}u\dv-2\int_Mv\SL_{\theta_0}u\dv+\int_Mv\SL_{\theta_0}v\dv\\
 &=&\E(u)-2\int_Mv\SL_{\theta_0}(\alpha\varphi)\dv-\int_Mv\SL_{\theta_0}v\dv\\
 &=&\lambda+O\Big(\delta^{6-2n}\epsilon^{2n}+\epsilon^6+F_2(\theta(t))^{1/2}+\|v\|^2\Big).
\end{eqnarray*}
By Lemma \ref{SLphikv} (ii) and \eqref{UnitIntoff}, we can get
\begin{eqnarray*}
 f_{\alpha\varphi}&=&\int_Mf(\alpha\varphi)^{2+2/n}\dv=\int_Mf(\alpha\varphi)^{1+2/n}\alpha\phi_1\dv\\
 &=&\int_Mfu^{1+2/n}\alpha\varphi\dv+O\Big(|\nabla f(a)|^2_{\theta_0}\epsilon^2+\|v\|^2\Big)\\
 &=&\int_Mfu^{1+2/n}(u-v)\dv+O\Big(|\nabla f(a)|^2_{\theta_0}\epsilon^2+\|v\|^2\Big)\\
 &=&1-\int_Mfu^{1+2/n}v\dv+O\Big(|\nabla f(a)|^2_{\theta_0}\epsilon^2+\|v\|^2\Big)\\
 &=&1+O\Big(|\nabla f(a)|^2_{\theta_0}\epsilon^2+\|v\|^2\Big).
\end{eqnarray*}
Combining the two estimates above yields
\begin{equation}\label{ratioofEandf}
 \frac{\E(\alpha\varphi)}{f_{\alpha\varphi}}=\lambda+O\Big(\delta^{6-2n}\epsilon^{2n}+\epsilon^6+|\nabla f(a)|^2_{\theta_0}\epsilon^2+F_2(\theta(t))^{1/2}+\|v\|^2\Big).
\end{equation}
Substituting this into \eqref{1storderderivativeofEf} gives
\begin{eqnarray}
 \label{1storderderivativeofEf1}
 \frac{1}{2}\partial\E_f(\alpha\varphi)h&=&f_{\alpha\varphi}^{-\frac{n}{n+1}}\bigg[\int_M\SL_{\theta_0}(\alpha\varphi)h\dv-\lambda\int_Mf(\alpha\varphi)^{1+2/n}h\dv\bigg]\nonumber\\
 &&\quad+O\Big(\delta^{6-2n}\epsilon^{2n}+\epsilon^6+|\nabla f(a)|^2_{\theta_0}\epsilon^2+F_2(\theta(t))^{1/2}+\|v\|^2\Big).
\end{eqnarray}
By Lemma \ref{ConformalOperatorOnBubble} and following the proof of Lemma \ref{SLphikv} (i), we have
\begin{equation}
 \label{SLvarphih}
 \int_M\SL_{\theta_0}(\alpha\varphi)h\dv=(2+2/n)\alpha\int_M\varphi^{1+2/n}h\dv+O(\delta^{4-3n}\epsilon^n).
\end{equation}
By the expansion of $f$ around $a$ and the proof of Lemma \ref{analysingsigmak} (i), we can obtain
\begin{eqnarray}\label{fvarphih}
&&\int_M \lambda f(\alpha\varphi)^{1+2/n}h\dv\nonumber\\
&&\qquad=\lambda\alpha^{1+2/n}f(a)\int_M\varphi^{1+2/n}h\dv+\int_M \lambda (f-f(a))(\alpha\varphi)^{1+2/n}h\dv\nonumber\\
&&\qquad=\lambda\alpha^{1+2/n}f(a)\int_M\varphi^{1+2/n}h\dv+O\Big(|\nabla f(a)|_{\theta_0}\epsilon+|\Delta_{\theta_0}f(a)|\epsilon^2+\epsilon^3+\epsilon^{n+2}\Big).\nonumber\\
\end{eqnarray}
Plugging the estimates \eqref{SLvarphih} and \eqref{fvarphih} into \eqref{1storderderivativeofEf1} gives
\begin{eqnarray*}
 \frac{1}{2}\partial\E_f(\alpha\varphi)h&=&f_{\alpha\varphi}^{-\frac{n}{n+1}}(2+2/n)\alpha\bigg(1-\frac{n\lambda\alpha^{2/n}f(a)}{2(n+1)}\bigg)\int_M\varphi^{1+2/n}h\dv\nonumber\\
 &&\quad+O\Big(\delta^{4-3n}\epsilon^{n}+|\nabla f(a)|_{\theta_0}\epsilon+|\Delta_{\theta_0}f(a)|\epsilon^2+\epsilon^3+\epsilon^{n+2}+F_2(\theta(t))^{1/2}+\|v\|^2\Big).
\end{eqnarray*}
This together with \eqref{asymptoticoffa} proves the claim (i).

To prove claim (ii), we notice that for $h\in H_u(\mu_0)$
\begin{eqnarray*}
 &&\int_M\SL_{\theta_0}(\alpha\varphi)h\dv\\
 &&\qquad\qquad= \int_M(\SL_{\theta_0}u)h\dv- \int_M(\SL_{\theta_0}v)h\dv\\
 &&\qquad\qquad=\int_M\lambda fu^{1+2/n}h\dv+O\Big(F_2(\theta(t))^{1/2}+\|v\|\Big)\|h\|\\
 &&\qquad\qquad=\int_M\lambda(f-f(a))u^{1+2/n}h\dv+\lambda f(a)\int_M u^{1+2/n}h\dv\\
 &&\qquad\qquad\qquad+O\Big(F_2(\theta(t))^{1/2}+\|v\|\Big)\|h\|\\
 &&\qquad\qquad=\int_M\lambda(f-f(a))u^{1+2/n}h\dv+\lambda f(a)\int_M \alpha u^{2/n}\varphi h\dv\\
 &&\qquad\qquad\qquad+\lambda f(a)\int_M u^{2/n}vh\dv+O\Big(F_2(\theta(t))^{1/2}+\|v\|\Big)\|h\|\\
 &&\qquad\qquad=\int_M\lambda(f-f(a))u^{1+2/n}h\dv+\lambda f(a)\int_M  u^{2/n}v h\dv\\
  &&\qquad\qquad\qquad+O\Big(F_2(\theta(t))^{1/2}+\|v\|\Big)\|h\|.
\end{eqnarray*}
Using the H\"older's inequality we can estimate
\begin{eqnarray*}
 &&\bigg|\int_M\lambda(f-f(a))u^{1+2/n}h\dv\bigg|\\
 &&\qquad\qquad\leqslant C\int_M|f-f(a)|\varphi^{1+2/n}|h|\dv+C\int_M|v|^{1+2/n}|h|\dv\\
 &&\qquad\qquad\leqslant C\bigg(\int_M|f-f(a)|^\frac{2(n+1)}{n+2}\varphi^{2+2/n}\dv\bigg)^\frac{n+2}{2(n+1)}\|h\|+C\|v\|\,\|h\|\\
 &&\qquad\qquad=O(\epsilon)\|h\|+O(\|v\|)\|h\|
\end{eqnarray*}
and
\begin{eqnarray*}
 &&\bigg|\lambda f(a)\int_M  u^{2/n}v h\dv\bigg|\leqslant C\bigg(\int_Mu^{2+2/n}\dv\bigg)^{\frac{1}{n+1}}\|v\|\,\|h\|\leqslant C\|v\|\,\|h\|.
\end{eqnarray*}
Combining all three estimates above yields
$$\int_M\SL_{\theta_0}(\alpha\varphi)h\dv=O(\mu_0)\|h\|.$$
This together with \eqref{SLvarphih} and \eqref{fvarphih} implies that
\begin{eqnarray*}
 \int_Mf(\alpha\varphi)^{1+2/n}h\dv=\frac{n\alpha^{2/n}f(a)}{2(n+1)}\int_M\SL_{\theta_0}(\alpha\varphi)h\dv+O(\mu_0)\|h\|=O(\mu_0)\|h\|.
\end{eqnarray*}
Substituting this, \eqref{asymptoticoffa} and \eqref{ratioofEandf} into \eqref{2ndorderderivativeofEf} yields
\begin{eqnarray*}
 &&\frac12\partial^2\E_f(\alpha\varphi)h_1h_2\\
 &&\quad= f_{\alpha\varphi}^{-\frac{n}{n+1}}\bigg[\int_M(\SL_{\theta_0}h_1)h_2\dv-(1+2/n)\frac{\E(\alpha\varphi)}{f_{\alpha\varphi}}\int_Mf(\alpha\varphi)^{2/n}h_1h_2\dv\bigg]\\
 &&\qquad\qquad+O(\mu_0)\|h_1\|\,\|h_2\|\\
 &&\quad= f_{\alpha\varphi}^{-\frac{n}{n+1}}\bigg[\int_M(\SL_{\theta_0}h_1)h_2\dv-(1+2/n)\int_M\lambda f(\alpha\varphi)^{2/n}h_1h_2\dv\bigg]\\
  &&\qquad\qquad+O(\mu_0)\|h_1\|\,\|h_2\|\\
  &&\quad= f_{\alpha\varphi}^{-\frac{n}{n+1}}\bigg[\int_M(\SL_{\theta_0}h_1)h_2\dv-(1+2/n)\lambda\alpha^{2/n}f(a)\int_M\varphi^{2/n}h_1h_2\dv\\
  &&\qquad\qquad-(1+2/n)\int_M\lambda (f-f(a))(\alpha\varphi)^{2/n}h_1h_2\dv\bigg]+O(\mu_0)\|h_1\|\|h_2\|\\
   &&\quad= f_{\alpha\varphi}^{-\frac{n}{n+1}}\bigg[\int_M(\SL_{\theta_0}h_1)h_2\dv-\frac{(2n+2)(n+2)}{n^2}\int_M\varphi^{2/n}h_1h_2\dv\bigg]\\
   &&\qquad\qquad+O(\mu_0)\|h_1\|\,\|h_2\|,
\end{eqnarray*}
which proves the claim (ii), and the proof of the lemma is complete.
\end{proof}

With the help of Lemma \ref{derivativeonHu}, we can show that the second  variation of the functional $\E_f$, when restricted in $H_u\big(\mu_0\big)$,  has a positive lower bound.
\begin{lemma}
 \label{positiveof2ndvariation}
 There exist constants $\gamma, t_0>0$ such that, for all $t>t_0$, there holds
 $$\partial^2\E_f(\alpha\varphi)|_{H_u(\mu_0)}\geqslant\gamma.$$
\end{lemma}
\begin{proof}
 Since the proof is almost the same as that of \cite[Proposition 5.5]{HSW}, we provide only some necessary steps.  The proof is given by contradiction. So, we assume the contrary. Then it follows from Lemma \ref{derivativeonHu} (ii) that
 there would exist $\mu_k\searrow0$ and $(w_k)\subset H_{u_k}(\mu_k)$ such that
 $$1=\int_M(2+2/n)|\nabla w_k|_{\theta_0}^2+\Ro w_k^2\dv\leqslant\frac{2(n+1)(n+2)}{n^2}\lim_{k\rightarrow\infty}\int_M\varphi^{2/n}w_k^2\dv.$$
 Here, $\mu_k=\mu_0(t_k)$ and $u_k=u(t_k)$. To proceed, we define a sequence of function $\hat{w}_k: T_{a_k}M\mapsto \mathbb{R}$ by
 $$\hat{w}_k=\mu_k^nw_k\Big(\exp_{a_k}(\mu_k z, \mu_K^2s)\Big)$$
for $(z, s)\in T_{a_k}M$. The sequence $(\hat{w}_k)$ satisfies
$$\lim_{k\rightarrow\infty}\int_{\{(z, s)\in\mathbb{H}^n:|(z ,s)|\leqslant N_k\}}(2+2/n)|\nabla \hat{w}_k|^2_{\theta_{\mathbb{H}^n}}\dvh\leqslant1.$$
and
$$\lim_{k\rightarrow\infty}\int_{\{(z, s)\in\mathbb{H}^n:|(z ,s)|\leqslant N_k\}}|\hat{w}_k|^{2+2/n}\dvh\leqslant C.$$
Hence, we can take weak limit to obtain a function $\hat{w}:\mathbb{H}^n\to \mathbb{R}$ such that
$$\int_{\mathbb{H}^n}\frac{\hat{w}^2}{s^2+(1+|z|^2)^2}\dvh>0.$$
In particular, $\hat{w}\neq0$. Using the orthogonality $w_k\in H_{u_k}(\mu_k)$ and following the proof of \cite[Proposition 5.5]{HSW}, we obtain
\begin{equation*}
\begin{split}
\int_{\mathbb{H}^n}\frac{\hat{w}(z, s)}{(s^2+(1+|z|^2)^2)^\frac{n+2}{2}}\dvh&=0,\\
\int_{\mathbb{H}^n}\frac{(1+|z|^2)\bar{z}-\sqrt{-1}\bar{z}s}{(s^2+(1+|z|^2)^2)^\frac{n+4}{2}}\hat{w}(z, s)\dvh&=0,\\
  \int_{\mathbb{H}^n}\frac{1-|z|^4-s^2}{(s^2+(1+|z|^2)^2)^\frac{n+4}{2}}\hat{w}(z, s)\dvh&=0,\\
\int_{\mathbb{H}^n}\frac{(1+|z|^2)z+\sqrt{-1}zs}{(s^2+(1+|z|^2)^2)^\frac{n+4}{2}}\hat{w}(z, s)\dvh&=0,\\
   \int_{\mathbb{H}^n}\frac{s}{(s^2+(1+|z|^2)^2)^\frac{n+4}{2}}\hat{w}(z, s)\dvh&=0,
\end{split}
\end{equation*}
which implies that $\hat{w}=0$, which is a contradiction.
 \end{proof}

 Now,  we are ready to obtain the apriori estimate for $v$.
 \begin{corollary}\label{aprioriestimateofv}
 For $t$ large, there holds
  $$\|v\|=O\Big(\delta^{4-3n}\epsilon^n+|\nabla f(a)|_{\theta_0}\epsilon+|\Delta_{\theta_0}f(a)|\epsilon^2+\epsilon^3+\epsilon^{n+2}+F_2(\theta(t))^{1/2}\Big).$$
 \end{corollary}
\begin{proof}
 Since $v\in H_u\big(\mu_0\big)$, it follows from Lemma \ref{positiveof2ndvariation} and Young's inequality that
 \begin{eqnarray*}
  \partial\E_f(u)v&=&\partial\E_f(\alpha\varphi+v)v\\
  &=&\partial\E_f(\alpha\varphi)v+\partial^2\E_f(\alpha\varphi)v^2+o(\|v\|^2)\\
  &\geqslant&\partial\E_f(\alpha\varphi)v+\gamma\|v\|^2+o(\|v\|^2).\\
  &\geqslant&O\Big(\|\partial\E_f(\alpha\varphi)|_{H_u(\mu)}\|^2\Big)+\frac{\gamma}{2}\|v\|^2+o(\|v\|^2).
 \end{eqnarray*}
On the other hand,  by \eqref{UnitIntoff}, \eqref{Eu} and H\"older's inequality,
we can estimate
\begin{eqnarray*}
 \Big|\frac12\partial\E_f(u)v\Big|&=&\bigg|\int_M(\SL_{\theta_0}u)v\dv-\frac{\E(u)}{f_u}\int_Mfu^{1+2/n}v\dv\bigg|\\
 &=&\bigg|\int_M(\SL_{\theta_0}u)v\dv-\Big(\lambda+O(F_2(\theta(t))^{1/2})\Big)\int_Mfu^{1+2/n}v\dv\bigg|\\
 &\leqslant&\bigg|\int_M(\R-\lambda f)u^{1+2/n}v\dv\bigg|+O\Big(F_2(\theta(t))^{1/2}\|v\|\Big)\\
 &=&O\Big(F_2(\theta(t))^{1/2}\|v\|\Big)=O\Big(F_2(\theta(t))^{1/2}\Big)+o(\|v\|^2).
\end{eqnarray*}
Combining the two estimates above and using Lemma \ref{derivativeonHu} (i), we thus prove the assertion.
\end{proof}

With the help of the previous results, we may simplify the shadow flow as follows.
\begin{corollary}
\label{simpleshadowflow}
For $t$ large, there hold
\smallskip

\begin{itemize}
 \item [(i)] $\displaystyle\frac{\dot\epsilon}{\epsilon}=\lambda f(a)\Big[\frac{nd_2}{2(n+1)c_2}A_a\epsilon^{2n}+\frac{e_2}{c_2}\frac{\Delta_{\theta_0}f(a)}{f(a)}\epsilon^2\Big](1+o(\epsilon))+R_2,$\\[0.1em]
 \item [(ii)] $\displaystyle\frac{\dot a}{\epsilon}=\lambda f(a)\Big[\frac{e_3}{c_3}\frac{\nabla_{\theta_0}f(a)}{f(a)}\epsilon+\frac{e_4}{c_3}\frac{\nabla_{\theta_0}\Delta_{\theta_0}f(a)}{f(a)}\epsilon^3\Big](1+o(\epsilon))+R_3,$
\end{itemize}
where
$$R_2, R_3=O\Big(\delta^{8-6n}\epsilon^{2n}+|\nabla f(a)|^2_{\theta_0}\epsilon^2+|\Delta_{\theta_0}f(a)|^2\epsilon^4+\epsilon^6+\epsilon^{2n+4}+F_2(\theta(t))\Big).$$
\end{corollary}
\begin{proof}
 It follows from Lemma \ref{shadowflow}, Corollary \ref{simplifysigmak}, Corollary \ref{aprioriestimateofv} and \eqref{asymptoticoffa}.
\end{proof}

\section{Proof of Theorems \ref{main1} and \ref{main2}}

In this section, we finish the proof of Theorems \ref{main1} and \ref{main2}.\\ {\bf Step 1}.\quad We aim to construct an initial data $u_0$ satisfying \eqref{initialdata}. Given a sufficiently small $\epsilon_0>0$, we let
$$u_0=\alpha_0\varphi_{a_0,\epsilon_0}.$$
where $a_0\in M$ is chosen to be such that
\begin{equation}\label{condition}
\mathrm{d}\Big(a_0, \Big\{x\in M: f(x)=\sup_Mf\Big\}\Big)<\epsilon_0.
\end{equation}
Here, $\mathrm{d}$ is the Carnot-Carath\'{e}dory distance with respect
to $\theta_0$, and $\varphi_{a_0, \epsilon_0}$ is defined as in \eqref{bubble}, and $\alpha_0$ is chosen such that
$$\int_Mfu_0^{2+2/n}\dv=1$$
i.e.,
\begin{equation}\label{alpha0}
\alpha_0^{-(2+2/n)}=\int_Mf\varphi_{a_0,\epsilon_0}^{2+2/n}~\dv.
\end{equation}
We claim that such $u_0$ satisfies \eqref{initialdata}. To see this, we note by \eqref{condition} that
\begin{eqnarray*}
 &&\int_Mf\varphi_{a_0,\epsilon_0}^{2+2/n}~\dv=\sup_Mf\int_M\varphi_{a_0,\epsilon_0}^{2+2/n}\dv+O(\epsilon_0)\\
 &&\qquad\qquad=\sup_Mf\int_{\mathbb{H}^n}\bigg(\frac{1}{(s^2+((1/n)^2+|z|^2)^2)^{n/2}}\bigg)^{2+2/n}\dvh+O(\epsilon_0)\\
 &&\qquad\qquad=\bigg(\frac{n\mu(S^{2n+1})}{2(n+1)}\bigg)^{n+1}\sup_Mf+O(\epsilon_0).
\end{eqnarray*}
Substituting the above into \eqref{alpha0} yields
$$\alpha_0^2=\bigg(\frac{n\mu(S^{2n+1})}{2(n+1)}\bigg)^{-n}\Big(\sup_Mf\Big)^{-\frac{n}{n+1}}+O(\epsilon_0).$$
This together with Lemma \ref{ConformalOperatorOnBubble} implies that
\begin{eqnarray*}
 \E(u_0)&=&\frac{2(n+1)}{n}\alpha_0^2\int_M\varphi_{a_0,\epsilon_0}^{2+2/n}\dv+O(\epsilon_0)\\
 &=&\frac{2(n+1)}{n}\alpha_0^2\bigg(\frac{n\mu(S^{2n+1})}{2(n+1)}\bigg)^{n+1}+O(\epsilon_0)\\
 &=&\mu(S^{2n+1})\Big(\sup_Mf\Big)^{-\frac{n}{n+1}}+O(\epsilon_0)\\
 &=&\frac{2(n+1)}{n}\Big(\sup_Mf\Big)^{-\frac{n}{n+1}}K_n^{-1}+O(\epsilon_0)\\
 &=&\lambda_*+O(\epsilon_0).
\end{eqnarray*}
Hence, by choosing $\epsilon_0$ sufficiently small,  we clearly have $\E(u_0)<2^\frac{1}{n+1}\lambda_*$. This proves the claim.
\smallskip

\noindent{\bf Step 2}.\quad In view of Lemma \ref{weaksolution}, we assume by contradiction that $u(t)\rightharpoonup 0$ weakly in $\FS$ as $t\rightarrow+\infty$. Then, from  the choice of $u_0$ in {\bf Step 1} and the arguments in Subsection \ref{SingleConcenPoint}, we conclude that there is a unique concentration point $a_\infty\in M$ with $a(t)\rightarrow a_\infty$ as $t\rightarrow+\infty$. We claim that $a_\infty\in\displaystyle \Big\{x\in M: f(x)=\sup_Mf\Big\}$. Otherwise, it follows from Lemma \ref{Limitoflambda} and \eqref{laminfty} that
$$\E(u(t))=\frac{2(n+1)}{n}K_n^{-1}\Big[f(a_\infty)\Big]^{-\frac{n}{n+1}}+o(1).$$
Hence, by choosing $\epsilon_0$ suitably small and for $t$ large, we have
$$\E(u(t))>\lambda_*+O(\epsilon_0)=\E(u_0),$$
which contradicts that the energy $\E(u)$ is non-increasing
by Lemma \ref{EnegyDecay}.
\smallskip

\noindent{\bf Step 3}.\quad We are now ready to derive the contradiction. To this end, we define
$$\zeta(t)=\frac{\ln \epsilon(t)}{f\big(a(t)\big)}.$$
Then, it is easy to see that $\zeta\rightarrow-\infty$, as $t\rightarrow+\infty$.  On the other hand, when $n=1$, it follows from Corollary \ref{simpleshadowflow} that
\begin{eqnarray*}
 \zeta^\prime(t)&=&\frac{1}{f(a)}\bigg(\frac{\dot\epsilon}{\epsilon}-\epsilon\ln\epsilon\frac{\nabla_{\theta_0}f(a)}{f(a)}\cdot\frac{\dot a}{\epsilon}\bigg)\\
 &\geqslant&\gamma\epsilon^2\bigg(A_a+\frac{\Delta_{\theta_0}f(a)}{f(a)}+\frac{|\nabla f(a)|_{\theta_0}^2}{f^2(a)}\ln(1/\epsilon)\bigg)\Big(1+o(\epsilon)\Big)+O(F_2(\theta(t)),
\end{eqnarray*}
for some positive constant $\gamma$. Notice that we have showed in {\bf Step 2} that  $a(t)\rightarrow a_\infty\in\Big\{x\in M: f(x)=\sup_Mf\Big\}$ as $t\rightarrow+\infty$. Our argument is now split into two cases.
\smallskip

\noindent{\bf Case} $\mu(M)>0$. If we take $c_*=\inf_{x\in M}A^*_x$,
then it follows from the assumptions in Theorem \ref{main1}
that $c_*>0$.
With this $c_*$, it follows from assumption \eqref{SupSet} that $\Delta_{\theta_0}f(a_\infty)/f(a_\infty)>-\inf_{x\in M}A^*_x$. Hence, for $t$ sufficiently large, it follows that $\Delta_{\theta_0}f(a)/f(a)>-\inf_{x\in M}A^*_x$. This together with the fact that $A_a=A^*_a$ implies that with a new constant $\gamma_1$ there holds
\begin{equation}\label{lowBdzetaprime}
 \zeta^\prime(t)\geqslant\gamma_1\epsilon^2\bigg(c_*+\frac{|\nabla f(a)|_{\theta_0}^2}{f^2(a)}\ln(1/\epsilon)\bigg)\Big(1+o(\epsilon)\Big)+O(F_2(\theta(t))\geqslant O(F_2(\theta(t))
\end{equation}
for $t$ sufficiently large.
\smallskip

\noindent{\bf Case} $\mu(M)=0$. Notice that
$$\left|\lim_{t\rightarrow+\infty}\frac{\Delta_{\theta_0}f(a)}{f(a)}
\right|=
\left|\frac{\Delta_{\theta_0}f(a_\infty)}{\sup_Mf}\right|<\frac{|\sup_M\Delta_{\theta_0}f|+1}{\sup_Mf}.$$
In this case, recall that $A_a=\Lambda$ is a positive constant to be determined. Hence, we may choose
$$\Lambda=\frac{|\sup_M\Delta_{\theta_0}f|+1}{\sup_Mf},$$
and we can obtain \eqref{lowBdzetaprime} in the same way.
\smallskip

\noindent By \eqref{EnergyBd}, $F_2(\theta(t))$ is integrable on $(0, +\infty)$. Therefore, we can integrate \eqref{lowBdzetaprime}
and conclude that $\zeta(t)$ is bounded from below,
which contradicts to the fact that $\zeta(t)\rightarrow-\infty$ as $t\rightarrow+\infty$.
The proof of the main theorems is thus complete.

\newpage
\appendix
\section{Proof of Lemma \ref{EstimatesofBubbles}}\label{A}

In this Appendix, we prove Lemma \ref{EstimatesofBubbles}. It follows
from Definition \ref{DerofBubbles} that $\phi_1=\varphi$, and a straightforward computation yields
\begin{eqnarray}
 \label{phi2}
 \phi_2&=&{\rm d}_2\varphi=\epsilon\partial_\epsilon\varphi\nonumber\\
 &=&\epsilon\partial_\epsilon\Bigg\{\omega_a\epsilon^n\Bigg[\frac{\chi_{\delta}(\rho_a)}{\big(s^2+((\epsilon/n)^2+|z|^2)^2\big)^{n/2}}+(1-\chi_{\delta}(\rho_a))G_a\Bigg]~\Bigg\}\nonumber\\
 &=&n\varphi+\omega_a\epsilon^{n+1}\chi_{\delta}(\rho_a)\partial_\epsilon\bigg(s^2+((\epsilon/n)^2+|z|^2)^2\bigg)^{-n/2}\nonumber\\
 &=&n\varphi-\frac2n\omega_a\epsilon^{n+2}\chi_{\delta}(\rho_a)\Big((\epsilon/n)^2+|z|^2\Big)\bigg(s^2+((\epsilon/n)^2+|z|^2)^2\bigg)^{-n/2-1}
\end{eqnarray}
and
\begin{eqnarray}
 \label{phi3}
 \phi_3&=&{\rm d}_3\varphi=\epsilon\nabla_a\varphi|_{a=0}\quad (a=(z_0, s_0))\nonumber\\
 &=&\epsilon\nabla_a\Bigg\{\omega_a\epsilon^n\Bigg[\frac{\chi_{\delta}(\rho_a)}{\big((s-s_0)^2+((\epsilon/n)^2+|z-z_0|^2)^2\big)^{n/2}}+(1-\chi_{\delta}(\rho_a))G_a\Bigg]~\Bigg\}\Bigg|_{a=0}\nonumber\\
 &=&\frac{\epsilon\nabla_a\omega_a}{\omega_a}\varphi+\omega_a\epsilon^{n+1}\Bigg[\nabla_a\Big(\chi_\delta(\rho_a)\Big)\bigg(\Big(s^2+((\epsilon/n)^2+|z|^2)^2\Big)^{-n/2}-G_a\bigg)\nonumber\\
 &&+\Big(1-\chi_\delta(\rho_a)\Big)\nabla_aG_a+n\chi_\delta(\rho_a)\frac{((\epsilon/n)^2+|z|^2)\bar{z}}{(s^2+((\epsilon/n)^2+|z|^2)^2)^{n/2+1}}\Bigg].
\end{eqnarray}
(i) {\bf Case $k=1$}. It is trivial for $l=1$, hence we consider $l=2, 3$. Since ${\rm d}_1\phi_l=\phi_l$, we have by \eqref{phi2} that
\begin{eqnarray*}
 |{\rm d}_1\phi_2|&\leqslant& n\varphi+\frac2n\omega_a\epsilon^n\chi_\delta(\rho_a)\Big(s^2+((\epsilon/n)^2+|z|^2)^2\Big)^{-n/2}\\
 &\leqslant&n\varphi+\frac2n\varphi\leqslant C\varphi.
\end{eqnarray*}
This gives the estimate for $l=2$. As for $l=3$, we first notice the facts that $$\nabla_a\Big(\chi_\delta(\rho_a)\Big)=0~~\mbox{ for }\rho_a<\delta\mbox{ or }\rho_a>2\delta;$$ and for $\delta\leqslant\rho_a\leqslant 2\delta$, there holds
$$\Big|\nabla_a\Big(\chi_\delta(\rho_a)\Big)\bigg(\Big(s^2+((\epsilon/n)^2+|z|^2)^2\Big)^{-n/2}-G_a\bigg)\Big|\leqslant C(\delta)$$
and
$$\chi_\delta(\rho_a)\Big(s^2+((\epsilon/n)^2+|z|^2)^2\Big)^{-n/2}+(1-\chi_\delta(\rho_a))G_a\geqslant C(\delta).$$
Hence
$$\bigg|\omega_a\epsilon^{n+1}\nabla_a\Big(\chi_\delta(\rho_a)\Big)\bigg(\Big(s^2+((\epsilon/n)^2+|z|^2)^2\Big)^{-n/2}-G_a\bigg)\bigg|\leqslant C(\delta)\varphi.$$
Moreover, $|\nabla_a G_a/G_a|\leqslant C(\delta)$ for $\rho_a\geqslant\delta$ and
$$\Bigg|\frac{((\epsilon/n)^2+|z|^2)\epsilon\bar{z}}{\big(s^2+((\epsilon/n)^2+|z|^2)^2\big)^{n/2+1}}\Bigg|\leqslant\frac12\Big(s^2+((\epsilon/n)^2+|z|^2)^2\Big)^{-n/2}.$$
This implies that
\begin{eqnarray*}
 &&\Bigg|\omega_a\epsilon^n\bigg(\epsilon\Big(1-\chi_\delta(\rho_a)\Big)\nabla_aG_a+n\chi_\delta(\rho_a)\frac{((\epsilon/n)^2+|z|^2)\epsilon\bar{z}}{\big(s^2+((\epsilon/n)^2+|z|^2)^2\big)^{n/2+1}}\bigg)\Bigg|\\
 &&\qquad\leqslant C(\delta)\omega_a\epsilon^n\bigg(\big(1-\chi_\delta(\rho_a)\big)G_a+\chi_\delta(\rho_a)\big(s^2+((\epsilon/n)^2+|z|^2)^2\big)^{-n/2}\bigg)\\
 &&\qquad=C(\delta)\varphi.
\end{eqnarray*}
It is easy to see that $|\epsilon(\nabla_a\omega_a/\omega_a)\varphi|\leqslant C(\delta)\varphi$. Substituting all the estimates above into \eqref{phi3}, we arrive at
$$|{\rm d}_1\phi_3|\leqslant C(\delta)\varphi.$$
\bigskip

\noindent{\bf Case $k=2$}. Since ${\rm d}_2\phi_1={\rm d}_2\varphi=\phi_2$, we are left to do estimates for $l=2, 3$.
From \eqref{phi2} it follows that
\begin{eqnarray}
 \label{d2phi2}
 {\rm d}_2\phi_2&=&{\rm d}_2\bigg(n\varphi-\frac2n\omega_a\epsilon^{n+2}\chi_{\delta}(\rho_a)\Big((\epsilon/n)^2+|z|^2\Big)\Big(s^2+((\epsilon/n)^2+|z|^2)^2\Big)^{-n/2-1}\bigg)\nonumber\\
 &=&n\phi_2-2n(n+2)\omega_a\epsilon^{n}(\epsilon/n)^2\chi_{\delta}(\rho_a)\Big((\epsilon/n)^2+|z|^2\Big)\Big(s^2+((\epsilon/n)^2+|z|^2)^2\Big)^{-n/2-1}\nonumber\\
 &&-4n\omega_a\epsilon^{n}(\epsilon/n)^4\chi_{\delta}(\rho_a)\Big(s^2+((\epsilon/n)^2+|z|^2)^2\Big)^{-n/2-1}\nonumber\\
 &&+4n(n+2)\omega_a\epsilon^{n}(\epsilon/n)^4\chi_{\delta}(\rho_a)\Big((\epsilon/n)^2+|z|^2\Big)^2\Big(s^2+((\epsilon/n)^2+|z|^2)^2\Big)^{-n/2-2}.\nonumber\\
\end{eqnarray}
Therefore
\begin{eqnarray*}
 |{\rm d}_2\phi_2|&\leqslant&n|\phi_2|+2n(n+2)\omega_a\epsilon^{n}\chi_{\delta}(\rho_a)\Big(s^2+((\epsilon/n)^2+|z|^2)^2\Big)^{-n/2}\\
 &&+4n\omega_a\epsilon^{n}\chi_{\delta}(\rho_a)\Big(s^2+((\epsilon/n)^2+|z|^2)^2\Big)^{-n/2}\\
 &&+4n(n+2)\omega_a\epsilon^{n}\chi_{\delta}(\rho_a)\Big(s^2+((\epsilon/n)^2+|z|^2)^2\Big)^{-n/2}\\
 &\leqslant&C\varphi.
\end{eqnarray*}
Now, by \eqref{phi3} we can compute
\begin{eqnarray}
 {\rm d}_2\phi_3&=&{\rm d}_2\Bigg\{
 \frac{\epsilon\nabla_a\omega_a}{\omega_a}\varphi+\omega_a\epsilon^{n+1}\Bigg[\nabla_a\Big(\chi_\delta(\rho_a)\Big)\bigg(\Big(s^2+((\epsilon/n)^2+|z|^2)^2\Big)^{-n/2}-G_a\bigg)\nonumber\\
 &&\qquad+\Big(1-\chi_\delta(\rho_a)\Big)\nabla_aG_a+n\chi_\delta(\rho_a)\frac{((\epsilon/n)^2+|z|^2)\bar{z}}{\big(s^2+((\epsilon/n)^2+|z|^2)^2\big)^{n/2+1}}\Bigg]~~\Bigg\}\nonumber\\
&=&-n\frac{\epsilon\nabla_a\omega_a}{\omega_a}\varphi+\frac{\epsilon\nabla_a\omega_a}{\omega_a}\phi_2+(n+1)\phi_3-2n\omega_a\epsilon^{n+1}\nabla_a\Big(\chi_\delta(\rho_a)\Big)(\epsilon/n)^2\Big((\epsilon/n)^2+|z|^2\Big)\nonumber\\
&&\times\Big(s^2+((\epsilon/n)^2+|z|^2)^2\Big)^{-n/2-1}+n^2\omega_a\epsilon^n\chi_\delta(\rho_a)\bigg(\frac{2(\epsilon/n)^3\bar{z}}{\big(s^2+((\epsilon/n)^2+|z|^2)^2\big)^{n/2+1}}\nonumber\\
&&-(n+2)\frac{2(\epsilon/n)^3\bar{z}((\epsilon/n)^2+|z|^2)^2}{\big(s^2+((\epsilon/n)^2+|z|^2)^2\big)^{n/2+2}}\bigg).
\end{eqnarray}
So, by using the conclusions in {\bf Case $k=1$}, we may estimate in a similar way as we did for $\phi_3$:
\begin{eqnarray*}
 | {\rm d}_2\phi_3|&\leqslant& C\varphi+2n\omega_a\epsilon^{n+1}\Big|\nabla_a\Big(\chi_\delta(\rho_a)\Big)\Big|\Big(s^2+((\epsilon/n)^2+|z|^2)^2\Big)^{-n/2}\\
 &&\qquad+n^2(n+3)\omega_a\epsilon^{n}\chi_\delta(\rho_a)\Big(s^2+((\epsilon/n)^2+|z|^2)^2\Big)^{-n/2}\\
 &\leqslant&C\varphi.
\end{eqnarray*}
\bigskip

\noindent{\bf Case $k=3$}. Since ${\rm d}_3\phi_1={\rm d}_3\varphi=\phi_3$, it suffices to consider $l=2, 3$.  In view of \eqref{phi2} we have
\begin{eqnarray}
 \label{d3phi2}
 {\rm d}_3\phi_2&=&{\rm d}_3\bigg(n\varphi-\frac2n\omega_a\epsilon^{n+2}\chi_{\delta}(\rho_a)\Big((\epsilon/n)^2+|z-z_0|^2\Big)\Big((s-s_0)^2+((\epsilon/n)^2+|z-z_0|^2)^2\Big)^{-n/2-1}\bigg)\bigg|_{a=0}\nonumber\\
 &=&n\phi_3-2n^2\nabla_a\Big(\omega_a\chi_{\delta}(\rho_a)\Big)\epsilon^{n}(\epsilon/n)^3\Big((\epsilon/n)^2+|z|^2\Big)\Big(s^2+((\epsilon/n)^2+|z|^2)^2\Big)^{-n/2-1}\nonumber\\
 &&\qquad-2n^2\omega_a\chi_\delta(\rho_a)\epsilon^{n}(\epsilon/n)^3\bar{z}\Big(s^2+((\epsilon/n)^2+|z|^2)^2\Big)^{-n/2-1}\nonumber\\
 &&\qquad-2n^2(n+2)\omega_a\chi_\delta(\rho_a)\epsilon^{n}(\epsilon/n)^3\bar{z}\Big((\epsilon/n)^2+|z|^2\Big)^2\Big(s^2+((\epsilon/n)^2+|z|^2)^2\Big)^{-n/2-2}.
\end{eqnarray}
This yields that
\begin{eqnarray*}
 | {\rm d}_3\phi_2|&\leqslant&C \varphi+2n^2\omega_a\epsilon^n\bigg(\bigg|\frac{\epsilon\nabla_a\big(\omega_a\chi_{\delta}(\rho_a)\big)}{n\omega_a}\bigg|+(n+3)\chi_\delta(\rho_a)\bigg)\Big(s^2+((\epsilon/n)^2+|z|^2)^2\Big)^{-n/2}\\
 &\leqslant&C\varphi.
\end{eqnarray*}
To finish the proof of (i), it remains to estimate ${\rm d}_3\phi_3$. From \eqref{phi3} we obtain
\begin{eqnarray}
 \label{d3phi3}
 {\rm d}_3\phi_3&=&{\rm d}_3\Bigg\{\frac{\epsilon\nabla_a\omega_a}{\omega_a}\varphi+\omega_a\epsilon^{n+1}\Bigg[\nabla_a\Big(\chi_\delta(\rho_a)\Big)\bigg(\Big((s-s_0)^2+((\epsilon/n)^2+|z-z_0|^2)^2\Big)^{-n/2}-G_a\bigg)\nonumber\\
 &&+\Big(1-\chi_\delta(\rho_a)\Big)\nabla_aG_a+n\chi_\delta(\rho_a)\frac{\big((\epsilon/n)^2+|z-z_0|^2\big)\big(\overline{z-z_0}\big)}{\big((s-s_0)^2+((\epsilon/n)^2+|z-z_0|^2)^2\big)^{n/2+1}}\Bigg]~\Bigg\}\Bigg|_{a=0}\nonumber\\
 &=&\epsilon^2 \Big[\nabla_a^2(\ln\omega_a)-\big(\nabla_a(\ln\omega_a)\big)^2\Big]\varphi+\big(\epsilon^2+\epsilon\big)\nabla_a\Big(\ln\omega_a\Big)\phi_3\nonumber\\
 &&+\omega_a\epsilon^{n+2}\Bigg[\nabla_a^2\Big(\chi_\delta(\rho_a)\Big)\bigg(\Big(s^2+((\epsilon/n)^2+|z|^2)^2\Big)^{-n/2}-G_a\bigg)+\nabla_a\Big(\chi_\delta(\rho_a)\Big)\nonumber\\
 &&\times\bigg(2n\bar{z}\Big((\epsilon/n)^2+|z|^2\Big)\Big(s^2+((\epsilon/n)^2+|z|^2)^2\Big)^{-n/2-1}-2\nabla_aG_a\bigg)\nonumber\\
 &&+\Big(1-\chi_\delta(\rho_a)\Big)\nabla_a^2G_a+n\chi_\delta(\rho_a)\bigg(-\bar{z}^2\Big(s^2+((\epsilon/n)^2+|z|^2)^2\Big)^{-n/2-1}\nonumber\\
 &&+(n+2)\bar{z}^2\Big((\epsilon/n)^2+|z|^2\Big)^2\Big(s^2+((\epsilon/n)^2+|z|^2)^2\Big)^{-n/2-2}~\bigg)~\Bigg]
\end{eqnarray}
A similar argument as before implies that
$$
  |{\rm d}_3\phi_3|\leqslant C\varphi.
$$
\bigskip

\noindent(ii) {\bf Case $k=1$}. Since $\phi_1=\varphi$, we have
\begin{eqnarray*}
 &&\int_M\varphi^{2/n}\phi_1^2\dv=\int_M\varphi^{2+2/n}\dv\\
 &&\qquad=\int_M\epsilon^{2n+2}\Bigg[\frac{\chi_{\delta}(\rho_a)}{\big(s^2+((\epsilon/n)^2+|z|^2)^2\big)^{n/2}}+\Big(1-\chi_{\delta}(\rho_a)\Big)G_a\Bigg]^{2+2/n}\dva\\
 &&\qquad=\int_{B_\delta(a)}\frac{\epsilon^{2n+2}}{\big(s^2+((\epsilon/n)^2+|z|^2)^2\big)^{n+1}}\dva+O(\epsilon^{2n+2})\\
 &&\qquad=\int_{B_\frac{n\delta}{\epsilon}(0)}\frac{n^{2n+2}}{\big(s^2+(1+|z|^2)^2\big)^{n+1}}\dvh+O(\epsilon^{2n+2})\\
 &&\qquad=\int_{\mathbb{H}^n}\frac{n^{2n+2}}{\big(s^2+(1+|z|^2)^2\big)^{n+1}}\dvh+O(\epsilon^{2n+2})\\
 &&\qquad:=c_1+O(\epsilon^{2n+2}).
\end{eqnarray*}

\bigskip

\noindent{\bf Case $k=2$}. By \eqref{phi2}, we can obtain
\begin{eqnarray*}
 &&\int_M\varphi^{2/n}\phi_2^2\dv=n^2\int_M\varphi^{2+2/n}\dv\\
 &&\qquad\qquad-4\int_M\varphi^{1+2/n}\frac{\omega_a\epsilon^{n+2}\chi_{\delta}(\rho_a)((\epsilon/n)^2+|z|^2)}{\big(s^2+((\epsilon/n)^2+|z|^2)^2\big)^{n/2+1}}\dv\\
 &&\qquad\qquad+\frac{4}{n^2}\int_M\varphi^{2/n}\frac{\omega_a^2\epsilon^{2n+4}\chi_{\delta}^2(\rho_a)((\epsilon/n)^2+|z|^2)^2}{\big(s^2+((\epsilon/n)^2+|z|^2)^2\big)^{n+2}}\dv\\
 &&\qquad:=I-II+III.
\end{eqnarray*}
By the estimate in Case $k=1$, one has
$$I=n^{2n+4}\int_{\mathbb{H}^n}\frac{1}{\big(s^2+(1+|z|^2)^2\big)^{n+1}}\dvh+O(\epsilon^{2n+2}).$$
Now, we estimate $II$ as follows.
\begin{eqnarray*}
 II&=&4\int_{B_\delta(a)}\frac{\epsilon^{2n+4}((\epsilon/n)^2+|z|^2)}{\big(s^2+((\epsilon/n)^2+|z|^2)^2\big)^{n+2}}\dva+O(\epsilon^{2n+4})\\
 &=&4n^{2n+4}\int_{B_\frac{n\delta}{\epsilon}(0)}\frac{(1+|z|^2)}{\big(s^2+(1+|z|^2)^2\big)^{n+2}}\dvh+O(\epsilon^{2n+4})\\
 &=&4n^{2n+4}\int_{\mathbb{H}^n}\frac{(1+|z|^2)}{\big(s^2+(1+|z|^2)^2\big)^{n+2}}\dvh+O(\epsilon^{2n+4})
\end{eqnarray*}
Similarly, we may estimate $III$  to obtain
\begin{eqnarray*}
  III&=&\frac{4}{n^2}\int_{B_\delta(a)}\frac{\epsilon^{2n+6}((\epsilon/n)^2+|z|^2)^2}{\big(s^2+((\epsilon/n)^2+|z|^2)^2\big)^{n+3}}\dva+O(\epsilon^{2n+6})\\
 &=&4n^{2n+4}\int_{B_\frac{\delta}{\epsilon}(0)}\frac{(1+|z|^2)^2}{\big(s^2+(1+|z|^2)^2\big)^{n+3}}\dvh+O(\epsilon^{2n+6})\\
 &=&4n^{2n+4}\int_{\mathbb{H}^n}\frac{(1+|z|^2)^2}{\big(s^2+(1+|z|^2)^2\big)^{n+3}}\dvh+O(\epsilon^{2n+6}).
\end{eqnarray*}
By setting
\begin{eqnarray*}
 c_2&=&n^{2n+4}\Bigg[\int_{\mathbb{H}^n}\frac{1}{\big(s^2+(1+|z|^2)^2\big)^{n+1}}\dvh\\
 &&-4\int_{\mathbb{H}^n}\frac{(1+|z|^2)}{\big(s^2+(1+|z|^2)^2\big)^{n+2}}\dvh\\
 &&+4\int_{\mathbb{H}^n}\frac{(1+|z|^2)^2}{\big(s^2+(1+|z|^2)^2\big)^{n+3}}\dvh\Bigg]\\
 &=&\int_{\mathbb{H}^n}\frac{n^{2n+4}}{\big(s^2+(1+|z|^2)^2\big)^{n+3}}\bigg[\Big(s^2+(1+|z|^2)^2\Big)^2\\
 &&\qquad-4(1+|z|^2)\Big(s^2+(1+|z|^2)^2\Big)+4(1+|z|^2)^2\bigg]\dvh\\
 &=&\int_{\mathbb{H}^n}\frac{n^{2n+4}\big(s^2+|z|^4-1\big)^2}{\big(s^2+(1+|z|^2)^2\big)^{n+3}}\dvh,
\end{eqnarray*}
we thus arrive at
$$\int_M\varphi^{2/n}\phi_2^2\dv=c_2+O(\epsilon^{2n+2}).$$
\bigskip

\noindent{\bf Case $k=3$}. It follows from \eqref{phi3} that
\begin{eqnarray*}
 &&\int_M\varphi^{2/n}\phi_3^2\dv=\int_M\varphi^{2/n}\Bigg\{\frac{\epsilon\nabla_a\omega_a}{\omega_a}\varphi+\omega_a\epsilon^{n+1}\Bigg[\nabla_a\Big(\chi_\delta(\rho_a)\Big)\\
 &&\qquad\qquad\bigg(\Big(s^2+((\epsilon/n)^2+|z|^2)^2\Big)^{-n/2}-G_a\bigg)+\Big(1-\chi_\delta(\rho_a)\Big)\nabla_aG_a\\
 &&\qquad\qquad+n\chi_\delta(\rho_a)\frac{((\epsilon/n)^2+|z|^2)\bar{z}}{\big(s^2+((\epsilon/n)^2+|z|^2)^2\big)^{n/2+1}}\Bigg]~\Bigg\}^2\dv\\
 &&\qquad=\epsilon^2\int_M\frac{|\nabla_a\omega_a|^2}{\omega_a^2}\varphi^{2+2/n}\dv+2\int_M\nabla_a\omega_a\varphi^{1+2/n}\cdot\epsilon^{n+2}\Bigg[\nabla_a\Big(\chi_\delta(\rho_a)\Big)\\
 &&\qquad\qquad\bigg(\Big(s^2+((\epsilon/n)^2+|z|^2)^2\Big)^{-n/2}-G_a\bigg)+\Big(1-\chi_\delta(\rho_a)\Big)\nabla_aG_a\\
 &&\qquad\qquad+n\chi_\delta(\rho_a)\frac{((\epsilon/n)^2+|z|^2)\bar{z}}{\big(s^2+((\epsilon/n)^2+|z|^2)^2\big)^{n/2+1}}\Bigg]\dv+\int_M\varphi^{2/n}\omega_a^2\epsilon^{2n+2}\\ &&\qquad\qquad\Bigg[\nabla_a\Big(\chi_\delta(\rho_a)\Big)
\bigg(\Big(s^2+((\epsilon/n)^2+|z|^2)^2\Big)^{-n/2}-G_a\bigg)+\Big(1-\chi_\delta(\rho_a)\Big)\nabla_aG_a\nonumber\\
 &&\qquad\qquad+n\chi_\delta(\rho_a)\frac{((\epsilon/n)^2+|z|^2)\bar{z}}{\big(s^2+((\epsilon/n)^2+|z|^2)^2\big)^{n/2+1}}\Bigg]^2\dv\\
&&\qquad :=I+II+III.
\end{eqnarray*}
{\it Estimate of $I$}. Notice that $\omega_a=1+O(\rho_a^4)$ and $|\nabla_a\omega_a|=O(\rho_a^3)$. Hence
\begin{eqnarray*}
 |I|&=&\epsilon^2\int_{B_{\delta}(a)}\frac{|\nabla_a\omega_a|^2}{\omega_a^2}\varphi^{2+2/n}\dv+O(\epsilon^{2n+4})\\
 &=&\epsilon^2O\bigg(\int_{B_{\delta}(a)}\rho_a^6~\varphi^{2+2/n}\dv\bigg)+O(\epsilon^{2n+4})\\
 &=&\epsilon^2O\bigg(\int_{B_{\delta}(0)}\frac{\epsilon^{2n+2}(s^2+|z|^4)^{3/2}}{\big(s^2+((\epsilon/n)^2+|z|^2)^2\big)^{n+1}}\dvh\bigg)+O(\epsilon^{2n+4})\\
 &=&\epsilon^8O\bigg(\int_{B_{\frac{n\delta}{\epsilon}}(0)}\frac{(s^2+|z|^4)^{3/2}}{\big(s^2+(1+|z|^2)^2\big)^{n+1}}\dvh\bigg)+O(\epsilon^{2n+4})\\
 &=&O(\epsilon^{\min(8,2n+4)}).
\end{eqnarray*}
{\it Estimate of $II$.} We have
\begin{eqnarray*}
 |II|&\leqslant&\int_{B_\delta(a)}2n\frac{|\nabla_a\omega_a|}{\omega_a}\epsilon^{2n+4}\frac{((\epsilon/n)^2+|z|^2)|z|}{(s^2+((\epsilon/n)^2+|z|^2)^2)^{n+2}}\dva+O(\epsilon^{2n+4})\\
 &=&O\bigg(\int_{B_\delta(a)}\rho_a^3{\omega_a}\epsilon^{2n+4}\frac{((\epsilon/n)^2+|z|^2)|z|}{(s^2+((\epsilon/n)^2+|z|^2)^2)^{n+2}}\dva\bigg)+O(\epsilon^{2n+4})\\
 &=&O\bigg(\int_{B_\delta(0)}\epsilon^{2n+4}\frac{(s^2+|z|^2)^\frac34((\epsilon/n)^2+|z|^2)|z|}{(s^2+((\epsilon/n)^2+|z|^2)^2)^{n+2}}\dva\bigg)+O(\epsilon^{2n+4})\\
 &=&O\bigg(\epsilon^4\int_{B_\frac{n\delta}{\epsilon}(0)}\frac{(s^2+|z|^2)^\frac34(1+|z|^2)|z|}{(s^2+(1+|z|^2)^2)^{n+2}}\dva\bigg)+O(\epsilon^{2n+4})\\
 &=&O(\epsilon^4).
\end{eqnarray*}
{\it Estimate of $III$.} We estimate
\begin{eqnarray*}
 III&=&\int_{B_\delta(a)}n^2\epsilon^{2n+4}\frac{((\epsilon/n)^2+|z|^2)^2|z|^2}{(s^2+((\epsilon/n)^2+|z|^2)^2)^{n+3}}\dva+O(\epsilon^{2n+4})\\
 &=&\int_{B_\frac{n\delta}{\epsilon}(0)}\frac{n^{2n+6}(1+|z|^2)^2|z|^2}{(s^2+(1+|z|^2)^2)^{n+3}}\dvh+O(\epsilon^{2n+4})\\
 &=&\int_{\mathbb{H}^n}\frac{n^{2n+6}(1+|z|^2)^2|z|^2}{(s^2+(1+|z|^2)^2)^{n+3}}\dvh+O(\epsilon^{2n+4})\\
 &:=&c_3+O(\epsilon^{2n+4}).
\end{eqnarray*}
Hence
$$\int_M\varphi^{2/n}\phi_3^2\dv=c_3+O(\epsilon^{2n+2}+\epsilon^4).$$
(iii) By Definition \ref{DerofBubbles} and (ii) above, we obtain
$$\int_M\varphi^{1+2/n}\phi_k\dv={\rm d}_k\int_M\varphi^{2+2/n}\dv=O(\epsilon^{2n+2}).$$
(iv) Notice that if $k=1$ or $l=1$, we have by (iii) above that
$$\int_M\varphi^{2/n}\phi_k\phi_l\dv=\int_M\varphi^{1+2/n}\phi_k\dv=O(\epsilon^{2n+2}).$$
It remains to estimate $\int_M\varphi^{2/n}\phi_2\phi_3\dv$. From \eqref{phi2} and \eqref{phi3} and symmetry, it follows that
\begin{eqnarray*}
 &&\int_M\varphi^{2/n}\phi_2\phi_3\dv=\int_M\varphi^{2/n}\bigg(n\varphi-\frac2n\omega_a\epsilon^{n+2}\chi_\delta(\rho_a)\frac{(\epsilon/n)^2+|z|^2}{(s^2+((\epsilon/n)^2+|z|^2)^2)^{n/2+1}}\bigg)\\
 &&\quad\qquad\times\bigg\{\frac{\epsilon\nabla_a\omega_a}{\omega_a}\varphi+\omega_a\epsilon^{n+1}\Big[\nabla_a\Big(\chi_\delta(\rho_a)\Big)\Big((s^2+((\epsilon/n)^2+|z|^2)^2)^{-n/2}-G_a\Big)\nonumber\\
 &&\quad\qquad+\Big(1-\chi_\delta(\rho_a)\Big)\nabla_aG_a+n\chi_\delta(\rho_a)\frac{((\epsilon/n)^2+|z|^2)\bar{z}}{(s^2+((\epsilon/n)^2+|z|^2)^2)^{n/2+1}}\Big]\bigg\}\dv\\
 &&\qquad=\int_M\varphi^{2/n}\bigg(n\varphi-\frac2n\omega_a\epsilon^{n+2}\chi_\delta(\rho_a)\frac{(\epsilon/n)^2+|z|^2}{(s^2+((\epsilon/n)^2+|z|^2)^2)^{n/2+1}}\bigg)\\
 &&\quad\qquad\times\bigg\{\frac{\epsilon\nabla_a\omega_a}{\omega_a}\varphi+\omega_a\epsilon^{n+1}\Big[\nabla_a\Big(\chi_\delta(\rho_a)\Big)\Big((s^2+((\epsilon/n)^2+|z|^2)^2)^{-n/2}-G_a\Big)\nonumber\\
 &&\quad\qquad+\Big(1-\chi_\delta(\rho_a)\Big)\nabla_aG_a\Big]\bigg\}\dv\\
 &&\qquad=\int_{B_\delta(a)}\frac{\epsilon\nabla_a\omega_a}{\omega_a}\varphi^{1+2/n}\bigg(n\varphi-\frac2n\omega_a\epsilon^{n+2}\frac{(\epsilon/n)^2+|z|^2}{(s^2+((\epsilon/n)^2+|z|^2)^2)^{n/2+1}}\bigg)\dv\\
 &&\qquad\quad+O(\epsilon^{2n+2})\\
 &&\qquad=O\bigg(\int_{B_\delta(a)}\frac{\epsilon^{2n+3}\rho_a^3|s^2+|z|^4-(\epsilon/n)^4|}{(s^2+((\epsilon/n)^2+|z|^2)^2)^{n+2}}\dva\bigg)+O(\epsilon^{2n+2})\\
 &&\qquad=O\bigg(\int_{B_\delta(0)}\frac{\epsilon^{2n+3}(s^2+|z|^4)^\frac34|s^2+|z|^4-(\epsilon/n)^4|}{(s^2+((\epsilon/n)^2+|z|^2)^2)^{n+2}}\dvh\bigg)+O(\epsilon^{2n+2})\\
 &&\qquad=O\bigg(\epsilon^4\int_{B_\frac{n\delta}{\epsilon}(0)}\frac{(s^2+|z|^4)^\frac34|s^2+|z|^4-1|}{(s^2+(1+|z|^2)^2)^{n+2}}\dvh\bigg)+O(\epsilon^{2n+2})\\
 &&\qquad=O(\epsilon^4).
\end{eqnarray*}

This  finishes the proof of Lemma \ref{EstimatesofBubbles}.

\end{document}